\documentclass[reqno,11pt]{amsart}
\usepackage{bbm}
\usepackage{url}
\usepackage{soul}
\usepackage{epstopdf}
\usepackage[ruled]{algorithm}
\usepackage[a4paper,bindingoffset=0.5cm,left=2cm,right=2cm,top=2.5cm,bottom=2cm,footskip=.8cm]{geometry}

\usepackage{stmaryrd}
\usepackage{caption}
\usepackage{pgfplots}
\usepackage{grffile}
\newlength\figureheight
\newlength\figurewidth
\pgfplotsset{%
        tick label style={font=\scriptsize},
        label style={font=\footnotesize},
        legend style={font=\footnotesize},
        every axis plot/.append style={very thick}
}

 \renewcommand{\leq}{\leqslant}
\renewcommand{\geq}{\geqslant}

\DeclareMathOperator*{\argmin}{arg\,min}

\allowdisplaybreaks

\makeatletter
\renewcommand{\fnum@figure}[1]{\textbf{\figurename~\thefigure}. }
\renewcommand{\fnum@table}[1]{\textbf{\tablename~\thetable}. }
\makeatother

\usepackage[utf8]{inputenc}
\usepackage{graphicx}
\usepackage{subfig}
\usepackage{type1cm}        
\usepackage{multicol}        
\usepackage[bottom]{footmisc}

\usepackage{newtxtext}       %
\usepackage{newtxmath}       
\usepackage{algorithm, algpseudocode, amsmath}

\newcommand{\vb}{\vspace{3.2mm}}

\usepackage[british,UKenglish,USenglish,english,american,dutch]{babel} 
\usepackage{mathrsfs}
\usepackage{lipsum}
\usepackage{verbatim}
\usepackage{amsmath, amsfonts, amsthm,mathtools}
\usepackage{syntonly}
\usepackage{float}
\usepackage{subcaption}
\usepackage{enumitem}
\usepackage{hyperref}

\usepackage{bbm}
\usepackage{eufrak}
\usepackage{tikz-cd} 
\usepackage{adjustbox}

\newtheorem{lemma}{Lemma}
\newtheorem{corollary}{Corollary}
\newtheorem{assumptions}{Assumptions}
\newtheorem{assumption}[assumptions]{Assumption}

\newtheorem{theorem}{Theorem}
\newtheorem{remark}{Remark}
\newtheorem{example}{Example}
\newtheorem{definition}{Definition}
\newtheorem{proposition}{Proposition}

\usepackage{hyperref}
\usepackage{multirow}
\usepackage{gensymb}

\begin{document}
\selectlanguage{USenglish}

        \title[Graphon Hawkes Processes: Sparse Estimation and Network Recovery
        ]{Graphon Hawkes Processes: Sparse Estimation and Network Recovery}       
\author{Justin R. Baars, Roger J.~A. Laeven, and Michel R.~H. Mandjes}
       
\begin{abstract}
Learning interaction structures from spatiotemporal event data is a key statistical problem in fields such as neuroscience, epidemiology, and seismology. 
We study this problem for Hawkes processes on a continuous spatial domain, where  events may trigger future events at different locations. 
Following the framework of \cite{Baars2024}, 
we model the spatial interaction mechanism using an underlying graphon that characterizes how the intensities' excitation strength varies across locations. 
By expanding the baseline intensity, temporal excitation kernel and graphon in orthogonal bases, we represent the problem as a high-dimensional parametric inference problem. 
We propose a sparse penalized least-squares estimator that recovers the effective interaction structure while controlling model complexity. 
Our analysis establishes non-asymptotic error bounds in a high-dimensional regime where the number of parameters may grow with the observation horizon, without requiring the conditional intensity to be almost surely bounded. 
We further establish sparsistency results, showing that the latent interaction network can be consistently recovered. 
Simulation studies demonstrate the good finite sample performance of the proposed method and an application to large-scale neuronal spike data, comprising 1.81 million events, showcases its practical effectiveness.

\vb


\vb

\noindent
{\sc Affiliations.} 
JB and RL are with the Dept.~of Quantitative Economics, University of Amsterdam. 
RL is also with E{\sc urandom}, Eindhoven University of Technology, and with C{\sc ent}ER, Tilburg University. 
MM is with the Mathematical Institute, Leiden University, and is also affiliated with the Korteweg-de Vries Institute for Mathematics, University of Amsterdam; E{\sc urandom}, Eindhoven University of Technology, Eindhoven; Amsterdam Business School, University of Amsterdam. 
The research of JB and RL is funded in part by the Netherlands Organization for Scientific Research under an NWO VICI grant (2020--2027).
The research of MM is funded in part by the NWO Gravitation project N{\sc etworks}, grant number 024.002.003.

\vb

\noindent
{\sc Email.} \url{j.r.baars@uva.nl},\  \url{r.j.a.laeven@uva.nl}, and\  \url{m.r.h.mandjes@math.leidenuniv.nl}.

\vb

\noindent
\textit{Date}: \today.

\end{abstract}

\maketitle
 

\section{Introduction}
Large-scale point process data arise in many scientific applications where events occur in spatiotemporal domains. 
In such settings, events are seldomly stochastically independent: the occurrence of one event may affect the likelihood of future events along the temporal and spatial dimensions. 
Identifying these interaction mechanisms, and the latent network structure that governs them, has become an important statistical task.
A widely used framework for modeling such interactions is the Hawkes process (\cite{Hawkes, Hawkes2}). 
In a multivariate Hawkes process, each component point process represents a node in a network, and past events in one component may influence the future intensities of other components through a system of interaction kernels. 
Because of its flexibility and interpretability, the multivariate Hawkes process has been applied in numerous areas, including neuronal spike recordings 
(\cite{Bonnet2022, Bonnet2023, Zhang2016}), 
financial (transaction) data (\cite{ACL15, Boswijk2026, finance2, Large}), seismic activity (\cite{seismology, Ikefuji}), and the spread of infectious diseases (\cite{covid}). 
In these applications, the key object of interest is often the interaction structure itself: A directed edge between two nodes in the network indicates that events in one component influence the future occurrence rate in another.

Despite the success of multivariate Hawkes processes, most existing models rely on a finite collection of component processes and therefore implicitly assume that the spatial domain has been discretized. 
However, in many applications, such as neuroscience, epidemiology, and seismology, the events occur in a genuinely continuous spatial domain and the interaction strength between two locations varies smoothly with their distance or other spatial characteristics. 
Discretizing space in such situations  introduces modeling artifacts and may obscure important structural features of the interaction mechanism. 
Moreover, increasing the spatial resolution in order to reduce discretization error quickly leads to high-dimensional models.

In this work, we study the estimation of the latent interaction structure in spatiotemporal point processes defined on a continuous spatial domain. 
Our starting point is the graphon Hawkes process framework introduced by \cite{Baars2024}, in which the connectivity structure of the process is described by a graphon. 
Let $(\mathscr X,\mathscr A,\mu)$ be a measure space representing the spatial domain and suppose that we observe a time-ordered point pattern $\omega_T=\{(t_i,x_i)\}_{i\in[n(T)]}$ consisting of $n(T)$ events in the observed time window $[0,T]$, where $t_i\in[0,T]$ denotes the event time and $x_i\in\mathscr X$ the spatial location. 
The data are assumed to arise from a spatiotemporal point process $N$ with conditional intensity density
\begin{equation}
\lambda(t,x):=\nu(x)+\sum_{t_i<t}W(x,x_i)h(t-t_i)=\nu(x)+\int_{(-\infty,t)\times\mathscr X}W(x,y)h(t-s)\,N(\mathrm ds,\mathrm dy).\end{equation}
Here, $\nu$ is a baseline intensity on $\mathscr X$, $h$ is a temporal excitation kernel, and $W$ is a graphon describing the strength of the interaction between spatial locations. 
Their precise domains and regularity conditions are given in Section~\ref{setup section}. 
When the process starts from an empty history, the integration area reduces to $(0,t)\times\mathscr X$.
The function $W$ plays the role of a latent connectivity structure governing the excitation dynamics of the process.  
Our statistical objective is to estimate the unknown functions $\nu$, $h$, and $W$ from the observed point pattern $\omega_T$.

The estimation problem is challenging for several reasons. 
First, the conditional intensity density and, in particular, the corresponding connectivity structure represented by the graphon are not observed directly but must be inferred from the observed pattern of events in time and space. 
Second, the graphon describes a population-level interaction structure rather than a single finite network.
In this sense, the network is \emph{doubly latent}: both the dynamics represented by the spatiotemporal conditional intensity density and the underlying graphon that governs it are hidden within the observed point pattern. 
Finally, the parameters $\nu$, $h$, and  $W$ belong to infinite-dimensional function spaces, which makes the estimation problem intrinsically high-dimensional.

To address these challenges, our starting point is a model representation based on orthogonal basis expansions. 
Let $(\varphi_k)_{k\in\mathbb N}$ be an orthonormal basis of $L^2(\mathscr X)$ and $(\chi_m)_{m\in\mathbb N}$ an orthonormal basis of $L^2(\mathbb R_+)$. 
We expand the model components as
\begin{align*}
&\nu(x)=\sum_{k\in\mathbb N}\alpha_k^*\varphi_k(x), \qquad
&&W(x,y)=\sum_{k,l\in\mathbb N}\beta_{k,l}^*\varphi_k(x)\varphi_l(y), \qquad
&&&h(t)=\sum_{m\in\mathbb N}\gamma_m^*\chi_m(t).\\
&\alpha_k^*:=\langle \nu,\varphi_k\rangle_{L^2(\mathscr X)},
\qquad
&&\beta_{k,l}^*:=\langle W,\varphi_k\otimes\varphi_l\rangle_{L^2(\mathscr X^2)},
\qquad
&&&\gamma_m^*:=\langle h,\chi_m\rangle_{L^2(\mathbb R_+)}.
\end{align*}
After truncating the expansions to $M=M(T)$ spatial basis functions and $H=H(T)$ temporal basis functions, the problem reduces to estimating the coefficients $(\alpha_k)_{k\in[M]}$, $(\beta_{k,l})_{k,l\in[M]}$, and $(\gamma_m)_{m\in[H]}$. 
The number of parameters grows rapidly with the spatial resolution, in particular through the $M^2$ interaction coefficients associated with the graphon. 

Such high-dimensional parameterizations require regularization in order to obtain stable and interpretable estimates. 
In many applications, the effective interaction structure is expected to be (approximately) sparse, meaning that only a small fraction of the multiplicity of all possible interactions play a substantial role (see, e.g., \cite{Buhlmann2011}). 
Motivated by this observation, we estimate the model parameters using a penalized least-squares approach with an $\ell^1$-type penalty on the interaction coefficients. 
This penalization promotes sparsity in the effective excitation mechanism in order to control model complexity.


Penalization methods have become a central tool in high-dimensional inference for point processes. 
An influential early contribution is \cite{Hansen}, which develops $\ell_1$-penalized estimators for multivariate counting processes whose intensities depend linearly on an unknown function expanded over a dictionary. 
Their analysis establishes oracle inequalities and introduces data-driven choices of penalty weights using martingale concentration bounds. 
Subsequent work has extended these ideas in various directions: \cite{Zhang2016} incorporates structural information through structured sparsity penalties, whereas \cite{Donnet2020} studies Bayesian nonparametric inference for Hawkes processes. 
Concentration inequalities for Hawkes processes derived in \cite{Chen2019} have also played an important role in subsequent developments, including high-dimensional inference procedures such as the decorrelated score test of \cite{Wang2024}. 
Related problems such as the classification of Hawkes process trajectories (\cite{Denis2024}), sparsity testing (\cite{Lotz}), and models with covariates (\cite{Fallahi}) have been investigated as well.

Among these contributions, the work of \cite{Cai2024} considers high-dimensional multivariate Hawkes process estimation using penalized least squares with spline basis expansions. 
Their framework allows for both excitatory and inhibitory interactions and accommodates time-varying baseline intensities. 
Importantly, they establish non-asymptotic error bounds and network recovery guarantees in regimes where the dimension of the multivariate process may grow faster than the observation horizon. 

Our approach shares several methodological aspects with this line of work, in particular the use of basis expansions and sparsity-inducing penalties to control high-dimensional interaction structures. 
However, the setting we consider differs fundamentally from \cite{Cai2024} in that our spatial domain $\mathscr X$ is continuous and infinite-dimensional, and that the connectivity structure is modeled through a graphon rather than a finite adjacency matrix. 
These differences remain after basis expansion and truncation: whereas these yield a finite-dimensional parameterization, the underlying interaction structure is still described by a graphon on a continuous, infinite-dimensional spatial domain.
In our setting, the finite network structure is replaced by a continuous-space interaction operator, requiring non-asymptotic control of a random Gram matrix that inherits the temporal and spatial dependences induced by the continuous-space graphon Hawkes process. 
Moreover, unlike \cite{Cai2024}, we do not impose an almost-sure upper bound on the conditional intensity. 
This necessitates an intrinsically different treatment of the stochastic terms arising in the estimation problem and leads to additional technical challenges in establishing non-asymptotic error bounds.

We summarize the main contributions of this paper. 
In Section~\ref{setup section}, we introduce the statistical framework and formally define the estimation problem together with the assumptions required for our analysis. 
Although the model allows for a continuum of spatial locations $\mathscr X$, the assumptions imposed are in several respects weaker than those used in related high-dimensional Hawkes process literature. 
In particular, unlike \cite{Cai2024}, we do not assume that the conditional intensity is almost surely bounded from above, 
thus allowing for linear Hawkes dynamics.

Section~\ref{consistency section} establishes non-asymptotic error bounds for the estimated intensity functions obtained from the truncated basis representation. 
The analysis is carried out in a high-dimensional regime where the number of basis coefficients may grow with the observation horizon. 
We overcome the technical difficulties resulting from our possibly unbounded conditional intensity through a truncation argument combined with probabilistic control of the relevant stochastic terms. 
In order to establish the non-asymptotic error bound, we require concentration inequalities for first- and second-order statistics of the spatiotemporal Hawkes process. 
These bounds are obtained using a thinning-based representation of the process and are of potential independent interest for the study of spatiotemporal point process models.

In Section~\ref{sparsistency section}, we study recovery of the latent interaction structure and establish \emph{sparsistency} of the proposed estimator: under suitable regularity and sparsity assumptions, we show that it identifies the true network edges with probability tending to one as the observation horizon grows. 
Thus, the procedure consistently recovers the underlying connectivity structure despite the high-dimensional parameterization induced by the basis expansion.

Section~\ref{simulation study section} presents a simulation study assessing the finite-sample performance and robustness of the proposed estimator across a range of sparsity levels, interaction structures, and forms of misspecification. 
The results demonstrate accurate coefficient recovery, support recovery, and decreasing estimation error with increasing observation horizon, with performance persisting under substantial departures from the specified model class. 

In Section~\ref{data study section}, we apply our method to a Neuropixels extracellular data set maintained by the Allen Institute for Brain Science \cite{AllenNeuroData}, consisting of time-stamped neuronal spike events with associated probe-depth information from an awake mouse exposed to controlled visual stimuli. 
The estimated graphons are highly consistent across substantially different spatial bases, with Legendre-polynomial and Haar-wavelet specifications recovering the same dominant connectivity structure despite their different representations. 
Together, these results provide evidence that the proposed procedure is able to recover interpretable and basis-robust spatial interaction patterns in both controlled simulations and real data.

We finish in Section~\ref{discussion section} with possible extensions and suggestions for further research.


\subsection*{Conventions} 
Throughout, we write $\mathbb{N} := \{1,2,3,\dots\}$, $\mathbb{N}_0 := \mathbb{N} \cup \{0\}$, and for $k \in \mathbb{N}$, $[k] := \{1, \dots, k\}$. 
The positive half-line is denoted by $\mathbb{R}_+ := [0,\infty)$. 
In general, temporal variables are denoted by $s, t$, and spatial variables by $x, y$. 
For matrices $A \in \mathbb{R}^{m \times n}$ and $B \in \mathbb{R}^{p \times q}$, $A \otimes B \in \mathbb{R}^{mp \times nq}$ denotes the tensor (Kronecker) product.

In the proofs, the $c_i$ and $c_{i,j}$ denote constants that may depend on model parameters but are independent of $T$, and the $q_i$ denote constants appearing in probabilistic bounds. 
These constants may vary between proofs and may be redefined in the main results for readability.

\subsection*{Appendices}
The proof for the concentration inequalities for first- and second-order statistics of the spatiotemporal Hawkes process together with some auxiliary results are collected in Appendix~\hyperref[app A]{A}.
The proofs of our main results that are not given in the main text  are collected Appendix~\hyperref[app B]{B}.
Finally,  some supplementary results corresponding to the simulation study conducted in Section \ref{simulation study section} are provided in  Appendix~\hyperref[app C]{C}.

\section{Setup and main assumptions}\label{setup section}

We observe a realization $\omega_T=\{(t_i,x_i)\}_{i\in[n(T)]}$ of a spatiotemporal point process $N$ on $[0,T]\times\mathscr X$, where $0\leqslant t_1\leqslant t_2\leqslant \cdots\leqslant t_{n(T)}\leqslant T$. Here $n(T)$ denotes the number of observed events in the time interval $[0,T]$. The spatial locations $x_i$ take values in a measurable space $(\mathscr X,\mathscr A,\mu)$, which we assume to be a Euclidean space equipped with the Lebesgue measure $\mu$, so that $\mathrm d\mu(x)=\mathrm dx$. Our framework can be extended in a straightforward way to a Polish space $\mathscr X$ endowed with a $\sigma$-finite Borel measure.

The starting point of our analysis is the linear graphon Hawkes process introduced in \cite{Baars2024}. Assuming that the process starts from an empty history, the dynamics of the spatiotemporal point process $N$ are specified through the conditional intensity density
\begin{equation}\label{conditional int dens}
\lambda(t,x):=\frac{\mathbb E[N(\mathrm dt,\mathrm dx)]}{\mathrm dt\,\mathrm dx}=\nu(x)+\sum_{t_i<t}W(x,x_i)h(t-t_i)=\nu(x)+\int_{(0,t)\times\mathscr X}W(x,y)h(t-s)\,N(\mathrm ds,\mathrm dy).\end{equation}
Here $\nu\in L^1_{\mathrm{loc}}(\mathscr X)$ denotes the baseline intensity, $h\in L^1(\mathbb R_+)$ is an excitation kernel, and $W:\mathscr X^2\to[0,C_W]$ is a (possibly directed) graphon describing the interaction structure between spatial locations.  The probabilistic structure of the process $N$ is uniquely determined by the conditional intensity density $\lambda$; see \cite{DaleyVereJones}, Proposition~7.3.IV.

Our objective is to estimate the parameters $\nu$, $W$, and $h$ from the observed data. In particular, we are interested in recovering the network structure encoded by the graphon $W$. While the primary focus of this paper is on estimating the interaction structure, we estimate the excitation kernel $h$ simultaneously with $\nu$ and $W$.

Below, in Assumption \ref{stability ass}(i), we  restrict attention to a stable regime. Define the linear operator
\begin{equation}\label{operator Thom marked general X}T_{\mathrm{hom}}:L^1(\mathscr X,\mu)\to L^1(\mathscr X,\mu):f(\cdot)\mapsto\|h\|_{L^1(\mathbb R_+)}\int_{\mathscr X}W(\cdot,y)f(y)\,\mathrm dy.\end{equation}
Stability of the process is ensured if $\rho(T_{\mathrm{hom}})<1$; see \cite{Baars2024}, Theorem~1.

Since $\mathscr X$ is a continuous space, the functions $\nu$, $W$, and $h$ belong to the infinite-dimensional spaces $L^2(\mathscr X)$, $L^2(\mathscr X^2)$, and $L^2(\mathbb R_+)$, respectively. To obtain a tractable estimation problem we adopt a nonparametric representation based on orthogonal basis expansions.

Let $(\varphi_k)_{k\in\mathbb N}$ be an orthonormal basis of $L^2(\mathscr X)$ and $(\chi_m)_{m\in\mathbb N}$ an orthogonal basis of $L^2(\mathbb R_+)$. In particular, 
\begin{align}
\langle \varphi_k,\varphi_l\rangle
&=\int_{\mathscr X}\varphi_k(x)\varphi_l(x)\,\mathrm dx=\delta_{kl},\nonumber\\
\langle \chi_k,\chi_l\rangle
&=\int_{\mathbb R_+}\chi_k(t)\chi_l(t)\,\mathrm dt=c_k\delta_{kl},
\label{orthonormality condition}
\end{align}
where $(c_k)_{k\in\mathbb N}\subset(0,\infty)$, and where $\delta_{kl}$ denotes the Kronecker delta.

Using these bases we expand the model components as
\begin{equation}
\nu(x)=\sum_{k\in\mathbb N}\alpha_k^*\varphi_k(x),\qquad
W(x,y)=\sum_{k,l\in\mathbb N}\beta_{k,l}^*\varphi_k(x)\varphi_l(y),\qquad
h(t)=\sum_{m\in\mathbb N}\gamma_m^*\chi_m(t).
\label{expansion basis infinite}
\end{equation}

To obtain a finite-dimensional representation we truncate these expansions. Let $M=M(T)$ denote the number of spatial basis functions and $H=H(T)$ the number of temporal basis functions retained. The truncated parametrization becomes
\begin{align}
\nu_\alpha(x)=\sum_{l\in[M]}\alpha_l\varphi_l(x),\qquad
W_\beta(x,y)=\sum_{k,l\in[M]}\beta_{k,l}\varphi_k(x)\varphi_l(y),\qquad
h_\gamma(t)=\sum_{m\in[H]}\gamma_m\chi_m(t).
\label{expansion basis finite}
\end{align}
Let $\theta=(\alpha,\beta,\gamma)\in\mathbb R^{M}\times\mathbb R^{M\times M}\times\mathbb R^H$. The conditional intensity associated with parameter $\theta$ is
\begin{equation}\label{lambda_c def}
\lambda_{\theta}(t,x)=\nu_\alpha(x)+\int_{(0,t)\times\mathscr X}W_\beta(x,y)h_\gamma(t-s)\,N(\mathrm ds,\mathrm dy).\end{equation}
The number of parameters grows rapidly with the spatial resolution, in particular through the $M^2$ interaction coefficients contained in $\beta$.

To avoid overfitting we impose a sparsity-inducing penalty on the interaction structure. In particular, following the general framework of penalized least-squares estimation in high-dimensional models \cite{Buhlmann2011}, Chapter 6, we aim to estimate $\theta$ by minimizing the penalized least-squares contrast
\begin{equation}\label{pen LS contrast}
\mathcal C(\omega_T,\theta)=\frac1T\left(-\int_{(0,T)\times\mathscr X}\lambda_\theta(t,x)\,N(\mathrm dt,\mathrm dx)+\frac12\int_{(0,T)\times\mathscr X}\lambda_\theta(t,x)^2\,\mathrm dt\,\mathrm dx\right)+\mathcal P(\theta),
\end{equation}
for a penalty function $\mathcal P(\theta)$.
A natural choice is the $\ell^1$ penalty
$$\mathcal P(\theta)=\eta\|\beta\otimes\gamma\|_{\ell^1},$$
where $\eta=\eta(T)>0$ promotes sparsity of the tensor coefficients $\beta_{k,l}\gamma_m$, yielding a lasso-type estimator. 
Other penalties can also be considered. 
In particular, an elastic net penalty combines the sparsity-inducing $\ell^1$ penalty with an $\ell^2$ regularization term, while a ridge penalty uses $\ell^2$ regularization alone and therefore does not induce sparsity. 
The error analysis of Section~\ref{consistency section} extends to these alternative penalty choices, although sparsistency is specific to sparsity-inducing penalties; see Remark~\ref{possible extensions remark} for further details.

Accordingly, we define the estimator $\hat\theta$ as any minimizer of
\begin{equation}\label{lasso LS contrast}
\hat \theta\in\argmin_{\substack{\theta\in\mathbb R^{M}\times\mathbb R^{M\times M}\times\mathbb R^H\\\|\gamma\|_{\ell^1}=1}}\left\{\frac1T\left(-\int_{(0,T)\times\mathscr X}\lambda_\theta(t,x)\,N(\mathrm dt,\mathrm dx)+\frac12\int_{(0,T)\times\mathscr X}\lambda_\theta(t,x)^2\,\mathrm dt\,\mathrm dx\right)+\eta\|\beta\otimes\gamma\|_{\ell^1}\right\},\end{equation}
where the constraint $\|\gamma\|_{\ell^1}=1$ ensures identifiability.

Define the random symmetric Gram matrix $G\in\mathbb R^{(1+MH)\times(1+MH)}$ by
\begin{align}
G_{0,0}&:=1,\nonumber\\
G_{0,(l,m)}&:=\frac1T\int_{(0,T)}\int_{(0,t)\times\mathscr X}
\varphi_l(y)\chi_m(t-s)\,N(\mathrm ds,\mathrm dy)\,\mathrm dt,\label{G random matrix}\\
G_{(l,m),(l',m')}
&:=\frac1T\int_{(0,T)}
\Psi_{l,m}(t)\Psi_{l',m'}(t)\,\mathrm dt,\nonumber
\end{align}
where
\begin{equation}\label{psi def}
\Psi_0(t)\equiv1,\qquad
\Psi_{l,m}(t)=\int_{(0,t)\times\mathscr X}
\varphi_l(y)\chi_m(t-s)\,N(\mathrm ds,\mathrm dy).
\end{equation}
Hence $G$ can be interpreted as the Gram matrix generated by the random functions $(\Psi_n)_{n\in\{0\}\cup([M]\times[H])}$.
This Gram matrix plays a key role in controlling differences between conditional intensity functions corresponding to different parameters, which can be seen from Lemma \ref{integrated diff and SDP} below.

We show that the loss function appearing in (\ref{lasso LS contrast}) is convex, which is important for efficient computation of the parameter minimizing the loss function. For notational convenience, we construct a parameter vector containing a tensor product of $\beta,\gamma$.
Let $\beta_{k,\cdot}=(\beta_{k,l})_{l\in[M]}$. We define 
\begin{align}\theta^\otimes:=\begin{bmatrix}
\alpha\\\beta\otimes\gamma\end{bmatrix}\in\mathbb R^{M+M^2H};\qquad
\theta^\otimes_k:=\begin{bmatrix}
\alpha_k\\\beta_{k,\cdot}\otimes\gamma\end{bmatrix}\in\mathbb R^{1+MH}.\label{theta tens}
\end{align}

\begin{lemma}\label{integrated diff and SDP}
Let $\tilde\theta=(\tilde\alpha,\tilde\beta,\tilde\gamma)$ and $\hat\theta=(\hat\alpha,\hat\beta,\hat\gamma)$ be two admissable parameter vectors. 
Let $(\varphi_k)_{k\in\mathbb N}$ be an orthonormal sequence.
Set $\Delta\alpha:= \tilde\alpha-\hat\alpha$, and $\Delta (\beta\otimes\gamma):=\tilde\beta\otimes\tilde\gamma-\hat\beta\otimes\hat\gamma$.  Let $\smash{G\in\mathbb R^{(1+MH)\times (1+MH)}}$ be the random Gram matrix defined in (\ref{G random matrix}), and let $\theta^\otimes$ and $\theta^\otimes_k$ be as in (\ref{theta tens}), with $\smash{\Delta\theta^\otimes:=\tilde \theta^\otimes-\hat\theta^\otimes}$ and  $\smash{\Delta\theta^\otimes_k:=\tilde \theta^\otimes_k-\hat\theta^\otimes_k}$. Then it holds that 
\begin{align}
\frac1{T}\int_{(0,T)\times\mathscr X}(\lambda_{\tilde \theta}(t,x)-\lambda_{\hat \theta}(t,x))^2\,\mathrm dt\,\mathrm dx =\sum_{k\in[M]}\left(\Delta \theta_k^\otimes\right)^\top G\Delta\theta^\otimes_k\geq\sigma_{\min}(G)\frac{\left(\|\Delta\alpha\|_{\ell^1}+\|\Delta (\beta\otimes\gamma)\|_{\ell^1}\right)^2}{\mathscr S},\label{eq13}
\end{align}
 where $\mathscr S\leq 2M^2H$ is the number of nonzero elements of $\Delta\theta^\otimes$.
\end{lemma}
The proof of Lemma \ref{integrated diff and SDP} can be found in Appendix~\hyperref[app B]{B}.
Note that $s$ may be of order lower than $M^2H$ in case $\Delta(\beta\otimes\gamma)$ is sparse.

The next result follows from expanding the loss function used in (\ref{lasso LS contrast}), applying the equality in (\ref{eq13}).
\begin{corollary}
 For $k\in[M]$ and $n\in\{0\}\cup([M]\times[H])$, let 
\begin{equation}
\vartheta_{k,n}:=\frac1T\int_{(0,T)\times\mathscr X}\varphi_k(x)\Psi_n(t)\,N(\mathrm dt,\mathrm dx),
\end{equation}
where $\Psi_n$ is defined in (\ref{psi def}).
Then $\vartheta\in\mathbb R^{M+M^2H}$ can be calculated from an observation of a spatiotemporal point process on $[0,T]\times\mathscr X$.
The loss function used in equation (\ref{lasso LS contrast})  can be rewritten to \begin{equation}\label{lasso LS contrast convex}
\theta\mapsto-\left(\theta^\otimes\right)^\top\vartheta+\frac12\sum_{k\in[M]}\left(\theta_k^ \otimes\right)^\top G\theta_k^\otimes+\eta\|\beta\otimes\gamma\|_{\ell^1}.
\end{equation} Hence, the loss function is convex.
\end{corollary}

Estimating $\nu$, $W$, and $h$ within the finite-dimensional parametrization in (\ref{expansion basis finite}) requires that the corresponding truncated model provides a sufficiently accurate approximation to the true intensity. This is essential for our consistency and sparsistency results: the estimator is defined on the finite-dimensional model, and hence these results can only hold if the approximation error induced by truncating the spatial and temporal expansions vanishes sufficiently rapidly as $M,H\to\infty$. We formalize this requirement as follows.
\begin{assumption}\label{ass approx}
There exists $\tilde \theta=\tilde\theta(T)\in\mathbb R^{M}\times\mathbb R^{M\times M}\times\mathbb R^H$ with $\|\tilde\gamma\|_{\ell^1}=1$ such that there exists a smoothness parameter $d\geq 2$ and positive constants $c_2,q_1,q_2$ such that with probability at least $1-q_1T\exp(-q_2T^{1/5})$,
\begin{equation}\label{eqn approx}
\frac1T\int_{(0,T)\times\mathscr X}(\lambda_{\tilde \theta}(t,x)-\lambda(t,x))^2\,\mathrm dt\,\mathrm dx\leq c_2M^{-2d}H^{-d},
\end{equation}
with $\tilde \beta_{k,l}=0$ whenever $\int_{\mathscr X}\int_{\mathscr X}W(x,y)\varphi_k(x)\varphi_l(y)\,\mathrm dx\,\mathrm dy=0$.
\end{assumption} 
The first part of the assumption controls the approximation error induced by truncating the spatial and temporal expansions and requires this error to decay sufficiently rapidly relative to the stochastic error. This is essential for both consistency and support recovery: the approximation error enters directly into the estimation bounds and, if it is not asymptotically negligible, neither convergence to the true model nor recovery of its support can in general be guaranteed. The additional support condition ensures that the approximating parameter $\tilde\theta$ does not introduce extraneous nonzero interaction coefficients, which is essential for the primal-dual witness argument underlying sparsistency.
The smoothness parameter $d$ governs the approximation rate and reflects the regularity of $\nu$, $W$, and $h$ relative to the chosen basis systems. The bound in (\ref{eqn approx}) can, for example, be established when these functions admit sufficiently accurate basis expansions; see \cite{Cai2024}, Assumption 1 and Appendix A10. Thus, Assumption~\ref{ass approx} imposes precisely the approximation rate required for the truncation error to be dominated by the statistical error while preserving the support structure of the true interaction kernel.

For our main results in Sections~\ref{consistency section} and~\ref{sparsistency section}, we make some structural assumptions on the conditional intensity $\lambda$ of the model and on both the orthonormal sequence $(\varphi_l)_{l\in\mathbb N}\subset L^2(\mathscr X)$ and on the orthogonal sequence $(\chi_m)_{m\in\mathbb N}\subset L^2(\mathbb R_+)$.

\begin{assumption}\label{stability ass}
Suppose that 
\begin{enumerate}[label=(\roman*)]
\item $\rho(T_{\mathrm{hom}})<1$;
\item there exists $\varrho_{T_{\mathrm{hom}}}\in(0,1)$ such that $\sup_{x\in\mathscr X}|T_{\mathrm{hom}}(\mathrm{id})(x)|\leq\varrho_{T_{\mathrm{hom}}}$;
\item $\mu(\mathscr X)<\infty$;
\item $\underline \nu:=\inf_{x\in\mathscr X}\nu(x)>0$ and $\overline\nu:=\sup_{x\in\mathscr X}\nu(x)<\infty$;
\item there exists a uniform bound $C_W$ such that $\sup_{x,y\in\mathscr X}W(x,y)\leq C_W$;
\item $h$ has bounded support contained in  $[0,b]$ and has $\|h\|_{L^p([0,b])}<\infty$ for $p=1,2,\infty$.
\end{enumerate}
\end{assumption}

We discuss the plausibility of Assumption \ref{stability ass}. Note that parts (i), (iii), and (v) have already been imposed in \cite{Baars2024} in order to prevent explosivity.
Next, parts (ii) and (iv) can be viewed as counterparts of \cite{Cai2024}, Assumptions~2 and 4. As noted in \cite{Baars2024}, the space of multivariate Hawkes processes can be embedded within the space of graphon Hawkes processes, making it natural to extend these assumptions to the present, more general framework.
In (vi) we assume that the excitation kernel $h$ has bounded support, an assumption that is quite common in the analysis of multivariate Hawkes processes \cite{Cai2024, Costa2020, Hansen}. Note that bounded support, together with $\|h\|_{L^\infty(\mathbb R_+)}<\infty$, already implies that $\|h\|_{L^1(\mathbb R_+)}<\infty$ and $\|h\|_{L^2(\mathbb R_+)}<\infty$. We nevertheless state these integrability conditions explicitly for clarity.

\begin{assumption}\label{min eigenvalue ass}
Suppose that
\begin{enumerate}[label=(\roman*)]
\item $(\varphi_k)_{k\in\mathbb N}\subset L^2(\mathscr X)$ is an orthonormal sequence for which there exists a constant $c>0$ such that $\max_{k\in[M]}\|\varphi_k\|_{L^\infty(\mathscr X)}^2\leq cM$;
\item $(\chi_m)_{m\in\mathbb N}\in L^1([0,b])\cap L^2([0,b])$ is an orthogonal sequence for which $\mathfrak h_1:=\|\chi_1\|_{L^1(\mathbb R_+)}=\|\chi_m\|_{L^1(\mathbb R_+)}$ for all $m\in\mathbb N$;
\item $\mathfrak h_p:=\sup_{m\in\mathbb N}\|\chi_m\|_{L^p([0,b]}<\infty$ for $p=2,\infty$;
\end{enumerate}
\end{assumption}
Assumption \ref{min eigenvalue ass}(i) holds, for instance, for multiscale wavelet families constructed from a bounded mother wavelet, such as the Haar or Daubechies wavelets. It also holds for typical families of orthogonal polynomials on $[0,1]$, such as (shifted and rescaled) Legendre, Jacobi, and Chebyshev polynomials; the former are used in Sections~\ref{simulation study section} and~\ref{data study section}. Assumption \ref{min eigenvalue ass}(ii) ensures that all functions $\chi_m$ have the same norm; for example, this common value could be chosen to equal $\|h\|_{L^1(\mathbb R+)}$. This requirement can be satisfied by a simple normalization and is therefore not restrictive. Assumption \ref{min eigenvalue ass}(iii) requires that Assumption \ref{stability ass}(vi) also holds for the collection $(\chi_m)_{m\in\mathbb N}\subset L^2([0,b])$, uniformly in $m\in\mathbb N$.
In summary, the approximating sequences $(\varphi_k)_{k\in\mathbb N}$ and $(\chi_m)_{m\in\mathbb N}$ are required to be as well behaved as the original parameters. 

\begin{remark}\label{nonorthogonal temporal functions}
It is possible to replace orthogonality in Assumption \ref{min eigenvalue ass}(ii)--(iii) by the assumption that  $H\in\mathbb N$ is independent of $T$, and that there exists a dominating function $\bar h\in L^1([0,b])$ such that $\sup_{m\in[H]}|\chi_m(x)|\leq \bar h(x)$, $[0,b]$--a.e. However, we choose to work with the orthogonality assumption, as that approach allows for greater flexibility.
\end{remark}

Under the measure $\mathbb P$ corresponding to the spatiotemporal point process $N$, which starts from the empty history, define the first- and second-order mean intensities $\bar\lambda^{(1)}(t,x)$ and $\bar\lambda^{(2)}(t_1,x_1,t_2,x_2)$ by, respectively, 
\begin{align}
\label{1st order mean intensity}
\bar\lambda^{(1)}(t,x)&:=\mathbb E[N(\mathrm dt,\mathrm dx)]/(\mathrm dt\,\mathrm dx);\\
\label{2nd order mean intensity}
\bar\lambda^{(2)}(t_1,x_1,t_2,x_2)&:=\mathbb E[N(\mathrm dt_1,\mathrm dx_1)N(\mathrm dt_2,\mathrm dx_2)]/(\mathrm dt_1\,\mathrm dx_1\,\mathrm dt_2\,\mathrm dx_2).
\end{align}
Furthermore, define the covariance $C^{(0)}(s,y,s',y')$ by 
\begin{equation}\label{covariance def}
C^{(0)}(t_1,x_1,t_2,x_2):=\bar\lambda^{(2)}(t_1,x_1,t_2,x_2)-\bar\lambda^{(1)}(t_1,x_1)\bar\lambda^{(1)}(t_2,x_2).
\end{equation}
For the covariance, we make the following assumption, which can be seen to be a counterpart of \cite{Cai2024}, Assumption 4.

\begin{assumption}\label{ass PSD C0}
Suppose that $C^{(0)}(t_1,x_1,t_2,x_2)$ is positive semi-definite, i.e, for any $f\in L^2(\mathbb R_+\times\mathscr X)$, 
\begin{equation}
\int_{(0,t)\times\mathscr X}\int_{(0,t)\times\mathscr X}f(t_1,x_1)\,C^{(0)}(t_1,x_1,t_2,x_2)\,f(t_2,x_2)\,\mathrm dt_1\,\mathrm dx_1\,\mathrm dt_2\,\mathrm dx_2\geq0.
\end{equation}
\end{assumption}
The positive semidefiniteness in Assumption~\ref{ass PSD C0} is not automatic. For a simple point process, the covariance measure contains both a singular, diagonal contribution, and the absolutely continuous off-diagonal component $C^{(0)}$; hence, positive semidefiniteness of the full covariance does not imply that of $C^{(0)}$. The assumption therefore rules out sufficiently strong negative dependence in the second-order structure. 
While the assumption is natural for self-exciting Hawkes processes with nonnegative interaction kernels, whose dependence is induced by the branching structure, it may fail for sufficiently inhibitory or repulsive processes, where the off-diagonal covariance can be negative.
Assumption~\ref{ass PSD C0} is used directly in the proof of Lemma~\ref{lemma min eigenvalue} to ensure a uniformly nondegenerate population Gram matrix. Without this condition, the required lower eigenvalue bound may fail, invalidating the subsequent error bounds and, consequently, the consistency and sparsistency results.

\section{Consistency}\label{consistency section}
In this section, we estimate $\lambda(t,x)$ using an observation $\omega_T$. We prove that we can consistently estimate the graphon from the spatiotemporal point process model by taking $\hat \theta$ as in (\ref{lasso LS contrast}), by showing that the estimated intensity $\lambda_{\hat\theta}$ converges to the true intensity $\lambda$, as $T\to\infty$. 

Our proof relies on concentration inequalities for first- and second-order statistics of the spatiotemporal Hawkes process. These inequalities are useful in the analysis of spatiotemporal Hawkes processes and may also be of independent interest.
Let $f_1:\mathbb R_+\times\mathscr X\to\mathbb R$ and $f_2:\mathbb R_+\times\mathscr X\times\mathscr X\to\mathbb R$ be bounded functions, with $\|f_1\|_{L^\infty},\|f_2\|_{L^\infty}<\infty$. Given a realization of $N$, we consider the empirical averages of these functions. Specifically, we study
\begin{align}
y_1:=y_1(T)&:=\frac1T\int_{(0,T)\times\mathscr X}f_1(t,x)\,N(\mathrm dt,\mathrm dx);\\
y_2:=y_2(T)&:=\frac1T\int_{(0,T)\times\mathscr X}\int_{(0,T)\times\mathscr X}f_2(t-s,x,y)\,N(\mathrm dt,\mathrm dx)\,N(\mathrm ds,\mathrm dy).
\end{align}
For these functionals of the point process $N$, we establish Bernstein-type concentration inequalities, as stated in Theorem \ref{concentration ineq}. The proof of Theorem \ref{concentration ineq} is provided in Appendix~\hyperref[app A]{A}.

\begin{theorem}[Concentration inequalities]\label{concentration ineq}
Grant Assumption \ref{stability ass}. Let $N$ be a Hawkes process with intensity given in (\ref{conditional int dens}), starting from the empty history. Let $f_1:\mathbb R_+\times\mathscr X\to\mathbb R$ and $f_2:\mathbb R_+\times\mathscr X\times\mathscr X\to\mathbb R$ be bounded functions, with $f_2$ having bounded support in its first variable; that is, there exists a bounded set $B$ such that $f_2(t,x,y)=0$ for all $t\notin B$. Then there exist positive constants $c_i>0$ and $q_j>0$, $i\in[2]$, $j\in[4]$, such that
\begin{align}
\mathbb P\left(|y_1-\mathbb Ey_1|\geq c_1T^{-2/5}\right)&\leq q_1T\exp\left(-q_2T^{1/5}\right); \label{1st order bernstein}\\
\mathbb P\left(|y_2-\mathbb Ey_2|\geq c_2T^{-2/5}\right)&\leq q_3T\exp\left(-q_4T^{1/5}\right).\label{2nd order bernstein}
\end{align}
\end{theorem}

\begin{remark}
The concentration rates in Theorem~\ref{concentration ineq} are not intrinsic to the argument. 
In fact, the proof of Theorem~\ref{concentration ineq} in Appendix~\hyperref[app A]{A} yields a {\em family} of exponential inequalities indexed by $\delta\in(0,1)$. 
More precisely, for the first- and second-order quantities, the deviations from the average of order $T^{-\delta}$ are controlled by exponential terms of order
$T\exp\left(-cT^{(1-\delta)/2}\right)$ and $T\exp\left(-cT^{(1-\delta)/3}\right),$
respectively, together with a term of order $\exp(-cT^{1-2\delta})$. 
Indeed, the choice $\delta=2/5$ in Theorem~\ref{concentration ineq} is made for convenience and can be modified to trade off the deviation rate against the strength of the tail bound. 
In particular, the resulting inequalities can be tailored to the rates required in subsequent arguments. 
The choice $\delta=2/5$ provides sufficient control for our purposes and is therefore fixed throughout the paper.
\end{remark}

We now state our first main theorem, which implies consistency of the penalized least-squares estimator to the true parameter, as $T\to\infty$. 
The proof of Theorem~\ref{consistency thm} can be found in Appendix~\hyperref[app B]{B}.

\begin{theorem}[Consistency]\label{consistency thm}
Consider an orthonormal sequence $(\varphi_l)_{l\in\mathbb N}\subset L^2(\mathscr X)$, an orthogonal sequence $(\chi_m)_{m\in\mathbb N}\in L^2([0,b])$ and a graphon Hawkes process $N$ on $[0,T]$ starting on an empty history, with a conditional intensity density as defined in (\ref{conditional int dens}), satisfying Assumptions \ref{ass approx}--\ref{ass PSD C0}. Let $\lambda_\theta$ be defined in (\ref{lambda_c def}), an suppose that we estimate $\theta$ by $\hat\theta$ as given in (\ref{lasso LS contrast}).
Select $M,H$ and $T$ such that $MH=o(T^{2/5})$ and $M^2H\log(T)/\sqrt T=o(1)$. Set $\eta:=c_1\log(T)/\sqrt T$ for some constant $c_1>0$. Then we can find constants $c_2,c_3>0$ and $q_1,q_2,q_3>0$ such that 
\begin{equation}
\frac1{T}\int_{(0,T)\times\mathscr X}(\lambda_{\hat \theta}(t,x)-\lambda(t,x))^2\,\mathrm dt\,\mathrm dx\leq c_2 M^{-2d}H^{-d}+c_3\frac{\log(T)^2M^2H}{T}
\label{thm consistency bound}
\end{equation}
holds with probability at least 
$1-q_1T\exp(-q_2T^{1/5})-q_3/T$.
\end{theorem}
\begin{remark}
The convergence rate in Theorem \ref{consistency thm} is optimized by matching the rates of the two contributing terms. This is achieved by choosing $$M^2H\asymp\left(\frac T{\log(T)^{2}}\right)^{\frac1{1+d}}.$$
Since $d\geq2$ (see Assumption \ref{ass approx}), this choice of $M^2H$ is compatible with the conditions of Theorem \ref{consistency thm}.
\end{remark}
\begin{remark}
The error bound in (\ref{thm consistency bound}) can be decomposed into two components. The first term corresponds to an {\em approximation error}, arising from representing the unknown background intensity and excitation kernels by a finite-dimensional basis; see Assumption \ref{ass approx}. 
The second term corresponds to the {\em estimation error}, capturing the stochastic variability in estimating the model parameters from a finite observation window. This decomposition closely parallels the bias–variance trade-off in high-dimensional regression, where the approximation error is driven by model complexity and smoothness assumptions, while the estimation error depends on the effective sample size and sparsity of the underlying model (\cite{Raskutti2012}). 
The different logarithmic factor in the second term is due to the concentration inequality for Hawkes processes that we use (rather than classical exponential inequalities for sub-Gaussian random variables in high-dimensional regression tasks).
\end{remark}

\begin{remark}
In the proof of Theorem \ref{consistency thm} in Appendix~\hyperref[app B]{B}, we define in \eqref{martingale Y_T} the martingale
$$Y_T := \int_{(0,T)\times\mathscr X} (\lambda_{\hat \theta}(t,x) - \lambda_{\tilde \theta}(t,x))(N(\mathrm dt,\mathrm dx) - \lambda(t,x)\,\mathrm dt\,\mathrm dx),$$ after which we aim to bound $I_2 = \frac{1}{T} Y_T$. 
It would have been considerably simpler if we could bound $I_2$ directly using exponential inequalities for martingales. However, doing so would require almost sure bounds on $\langle Y, Y \rangle_T$ and on the integrand $\lambda_{\hat \theta}(t, x) - \lambda_{\tilde \theta}(t, x)$, which in turn necessitate restrictive model assumptions such as $|\lambda(t, x)| \leq \lambda_{\max}$ for all $t, x$, a.s., cf.\ \cite{Cai2024}. 
We deliberately avoid imposing such assumptions. 
In contrast to \cite{Cai2024}, this choice entails a more intricate analysis involving truncation arguments, but it also allows for a substantially more general framework. 
Moreover, our approach can be applied to \cite{Cai2024} as well, showing that their assumption $\lambda(\cdot) \leq \lambda_{\max}$ a.s.\ is in fact unnecessary.
\end{remark}

\begin{remark}\label{possible extensions remark}
As the proof shows, Theorem \ref{consistency thm} allows for some variations and extensions. Consider $\theta^\otimes$ defined in (\ref{theta tens}).
Equation (\ref{Frob norm bound}) shows that if we penalize $\alpha$ as well, i.e., if $\hat \theta$ is calculated as 
\begin{align}
&(\hat\alpha,\hat\beta,\hat\gamma)=\hat \theta\nonumber\\
&\in\argmin_{\substack{\theta\in\mathbb R^{M}\times\mathbb R^{M\times M}\times\mathbb R^H\\\|\gamma\|_{\ell^1}=1}}\left\{\frac1T\left(-\int_{(0,T)\times\mathscr X}\lambda_\theta(t,x)\,N(\mathrm dt,\mathrm dx)+\frac12\int_{(0,T)\times\mathscr X}\lambda_\theta(t,x)^2\,\mathrm dt \,\mathrm dx\right)+\eta\|\theta^\otimes\|_{\ell^1}\right\},\label{lasso LS contrast alternative}
\end{align}
then we obtain the same probabilistic error bounds as in Theorem \ref{consistency thm}. 

Furthermore, the $\ell^1$ penalty may be replaced by an $\ell^2$-penalty, an elastic net penalty, or a standardized group lasso penalty (see \cite{Simon2013}), yielding similar probabilistic error bounds as in Theorem~\ref{consistency thm}.



\end{remark}

\section{Sparseness structure recovery}\label{sparsistency section}
In this section, we establish \emph{sparsistency}: consistency for the recovery of the sparsity structure of the graphon $W$, which characterizes the network structure of the graphon Hawkes process $N$. To formalize the notion of the ``true'' sparsity structure, we introduce a population analogue of the estimated coefficients $\beta_{k,l}$. Since our goal is to recover the sparsity structure of $W$, these population coefficients will serve as the basis for defining the underlying interaction structure.

\begin{definition}\label{active set def}
The \emph{active set} $\mathcal E$ is defined as the set of coordinates that have nonzero coefficients w.r.t.\ the true graphon $W$, i.e., 
\begin{equation}\label{active set eq}
\mathcal E:=\left\{(k,l)\in[M]^2: \int_{\mathscr X}\int_{\mathscr X}W(x,y)\varphi_k(x)\varphi_l(y)\,\mathrm dx\,\mathrm dy\neq0\right\}.
\end{equation}
Let $\mathcal S:=|\mathcal E|$ denote the cardinality of $\mathcal E$.
Furthermore, we write $\mathcal E_k:=\{l\in[M]:(k,l)\in\mathcal E\}$. 
Suppose that we estimate $\theta$ by $\hat\theta$ as given in (\ref{lasso LS contrast}).  We define the \emph{estimated active set} as 
\begin{equation}\label{est active set eq}
\hat{\mathcal E}:=\hat{\mathcal E}_T:=\left\{(k,l)\in[M]^2: \hat\beta_{k,l}\neq0\right\}.
\end{equation}
\end{definition}
In this section, we show that our proposed lasso method can correctly identify the true sparseness structure of the underlying graphon with probability tending to $1$ as $T\to\infty$. We introduce two more assumptions. The first is what is usually referred to as the \emph{minimal signal condition}; the second is what is known as the \emph{irrepresentable condition}.
\begin{assumption}\label{minimal signal}
Let $\tilde \theta=(\tilde\alpha,\tilde\beta,\tilde\gamma)\in\mathbb R^{M}\times\mathbb R^{M\times M}\times\mathbb R^H$ with $\|\tilde\gamma\|_{\ell^1}=1$ as in Assumption \ref{ass approx}. Suppose that there exist $\beta_{\min},\gamma_{\max}>0$ independent of $T$ such that $\min_{(k,l)\in\mathcal E}|\tilde \beta_{k,l}|\geq\beta_{\min}$ and $\max_{m\in[H]}|\tilde\gamma_m|\geq\gamma_{\max}$.
\end{assumption}

In the proof of Theorem \ref{recovery thm}, we argue that a small value of $|\Delta(\beta\otimes\gamma)_{k,l,m}| = |\tilde\beta_{k,l}\tilde\gamma_m - \hat\beta_{k,l}\hat\gamma_m|$ for $k,l \in \mathcal E$ implies that $\hat\beta_{k,l} \neq 0$. To this end, it is necessary to rule out the possibility that $\tilde\beta_{k,l}\tilde\gamma_m$ can be arbitrarily close to zero.
Note that the condition $\max_{m\in[H]}|\tilde\gamma_m|\geq\gamma_{\max}$ for $\gamma_{\max}:=1/H$ is immediate in case $H$ is independent of $T$. 

\begin{assumption}\label{irrepresentable cond}
We can find constants $c, q_1, q_2>0$ such that for all $k\in[M]$, \begin{equation}
\max_{\substack{l\in \mathcal E_k^{\mathsf{c}}\\ m\in[H]}}\left(G_{(l,m),0}^2+\sum_{\substack{l'\in\mathcal E_k\\m'\in[H]}}G_{(l,m),(l',m')}^2\right)^{1/2}\leq \frac c{\mathcal SH}
\end{equation}
on an event of probability at least  $1-q_1T\exp(-q_2T^{1/5})$.
\end{assumption}

This irrepresentable condition is a condition on the covariances between active and nonactive orthonormal functions.
Intuitively, this condition requires that irrelevant components for a given node cannot be well approximated by linear combinations of the relevant ones. 
This separation ensures that irrelevant components do not appear predictive merely because they are strongly correlated with the relevant processes, which is crucial for reliable support recovery with sparsity-inducing penalties; cf.\ e.g., \cite{
Ravikumar2009, ZhaoYu2006}.

We now state our second main theorem, which establishes \emph{sparsistency} of the proposed penalized least-squares estimator: it recovers the true sparsity structure of $W$ with probability tending to $1$ as $T\to\infty$. 
The proof of Theorem~\ref{recovery thm} can be found in Appendix~\hyperref[app B]{B}.

\begin{theorem}[Sparseness structure recovery]\label{recovery thm}
Consider an orthonormal sequence $(\varphi_l)_{l\in\mathbb N}\subset L^2(\mathscr X)$, an orthogonal sequence $(\chi_m)_{m\in\mathbb N}\in L^2([0,b])$ and a graphon Hawkes process $N$ on $[0,T]$ starting on an empty history, with a conditional intensity density as defined in (\ref{conditional int dens}) satisfying Assumptions \ref{ass approx}--\ref{irrepresentable cond}. Let $\lambda_\theta$ be defined in (\ref{lambda_c def}), an suppose that we estimate $\theta$ by $\hat\theta$ as given in (\ref{lasso LS contrast}). Let $d$ be as in Assumption \ref{ass approx}, and suppose that $d>2$. Select $M$, $H$, and $T$  such that
\begin{itemize}
\item $M^2H\log(T)/\sqrt T=o(1)$;
\item $MH=o(T^{2/5})$;
\item $\sqrt T/(M^d\log(T))=\mathcal O(1)$;
\item $M\to\infty$ as $T\to\infty$.
\end{itemize}
Then  we can find constants $c, q_1,q_2,q_3>0$ and a (large) $T^*$ such that if we select $\eta:=c\log(T)/\sqrt T$, then
\begin{equation}
\mathcal E=\hat{\mathcal E}
\label{thm recovery eq}
\end{equation}
holds for any $T\geq T^*$, with probability at least 
$1-q_1T\exp(-q_2T^{1/5})-q_3/T$.
\end{theorem}

\begin{remark}
    In our spatiotemporal point-process setting, we do not impose a uniform upper bound on the conditional intensity; consequently, the martingale increments are not uniformly bounded. We therefore require the localization argument of Step~3C (which parallels that of Theorem~\ref{consistency thm}): we truncate on events where the number of points in each interaction window is logarithmically bounded, control the complement using exponential moment bounds for a dominating Hawkes process, and apply an exponential inequality to the resulting stopped martingale.
\end{remark}

\section{Simulation study}\label{simulation study section}
In this section, we investigate the finite-sample performance and robustness of our  estimation procedure along three dimensions. 
We first consider four unperturbed models, chosen to span a range of sparsity levels and interaction structures within the specified finite-dimensional model class.
We then consider three perturbed versions of the first model, introducing progressively different forms of misspecification. 
Finally, we vary the observation horizon $T$ to assess the finite-sample convergence underlying the consistency results.

We consider an unperturbed model with $M=8$ basis functions for the spatial structure on $\mathscr X=[0,1]$ and $H=3$ basis functions for the temporal structure, with an observation horizon of $T=500\,000$. A relatively long observation horizon is natural in this setting, as our goal is to estimate an infinite-dimensional self-exciting model from event times and their spatial marks. 

To capture genuinely continuous spatial dynamics, we use orthogonal polynomials as the spatial basis functions. Specifically, we take the first $M=8$ Legendre polynomials on $\mathscr X=[0,1]$:
\begin{align}
\varphi_{n+1}(x)
&=\sqrt{2n+1}\,P_n(2x-1) \notag\\
&=(-1)^n\sum_{k=0}^{n}
\binom{n}{k}\binom{n+k}{k}(-x)^k,
\qquad x\in[0,1],\quad n\in\mathbb N_0,
\label{Haar wavelets}
\end{align}
where $(P_n)_{n\in\mathbb N_0}$ denotes the sequence of classical Legendre polynomials on $[-1,1]$.

Furthermore, the $H=3$ temporal functions that we consider are 
\begin{equation}
\chi_m(t)=\exp(-mt),\quad m\in[3].\label{exp basis functions}
\end{equation}
Even though the exponential functions in Equation (\ref{exp basis functions}) are not orthogonal, our method remains applicable; see Remark \ref{nonorthogonal temporal functions}.

Observe that in our setting we have $MH=24<190\approx T^{2/5}$, whereas $M^2H\log(T)/\sqrt T\approx 3.6$. Hence, the assumptions of Theorem \ref{consistency thm} and Theorem \ref{recovery thm} are only partially fulfilled. We therefore operate in a genuine finite-sample regime rather than in a near-asymptotic setting.

In this setup, we will perform simulations for seven different graphons, of which the first four are spanned by the thus described model, and of which the last three are perturbations of the first graphon that we will describe. 
For each of the models we consider, we set 
\begin{align}
\alpha=\begin{bmatrix}0.05&0&0&0&0&0&0&0\end{bmatrix}^\top;\quad\label{fixedalpha}
\gamma=\begin{bmatrix}1&0&0\end{bmatrix}^\top.
\end{align}
For each of the unperturbed graphons that are spanned by the model, we specify the matrix $\beta\in\mathbb R^8$ of coefficients below.  
We keep $\alpha$ and $\gamma$ fixed to isolate variation in the spatial interaction structure, for which sparsity and support recovery are of primary interest. 
The four choices of $\beta$ range from highly sparse to substantially denser structures, while varying the location and signs of the nonzero coefficients. 
They contain, respectively, $4$, $6$, $8$, and $28$ nonzero coefficients, thereby spanning interaction structures of markedly different complexity. 
This allows us to assess both estimation error and the ability of the $\ell^1$ penalty to recover supports across sparsity levels.

Finite sample performance is evaluated using $100$ simulations using  $T=500\,000$, initialized with seed \texttt{74973836}.
For each simulation, we generate a sample path, and estimate the coefficients using the estimation procedure proposed in this paper. As in Theorem \ref{consistency thm} and Theorem \ref{recovery thm}, we select $\eta\propto\log(T)/\sqrt T$; more specifically, we set $\eta=\log(T)/(4\sqrt T)$, since in our synthetic data we have events at approximately $1/4\ \mathrm{Hz}$.

 To assess consistency (Theorem \ref{consistency thm}), we report the mean and standard deviation of the estimated coefficients $\hat\alpha$, $\hat\beta$ and $\hat\gamma$. Additionally, we present a histogram of the Frobenius norm $\|\hat\beta-\beta\|_F$ between the estimated and true coefficients. For the perturbed cases where the graphon is given by $W_\beta$ plus a perturbation, we also compute $\|\hat\beta-\beta\|_F$, treating $\beta$ as the ``true'' coefficient. 

To evaluate sparsity structure recovery of the graphon, we define true positives (TP) as the correctly identified nonzero coefficients of $\beta$, false positives (FP) as coefficients that are zero in $\beta$ but detected as nonzero, and false negatives (FN) as nonzero coefficients of $\beta$ that are not correctly identified. Furthermore, we consider the false negative rate (FNR) and false positive rate (FPR), defined as FN divided by the number of nonzero coefficients in $\beta$, and FP divided by the number of zero coefficients in $\beta$, respectively. The $F_1$-score is defined by $F_1=2\mathrm{TP}/(2\mathrm{TP}+\mathrm{FN}+\mathrm{FP})$. The maximum possible $F_1$-score is $1$, which corresponds to perfect selection. This measure is widely used in machine learning to evaluate selection and classification accuracy; see, for example, \cite{Cai2024, Forman2003, Ho2012}. We classify a coefficient $\beta_{k,l}$ as nonzero if $|\beta_{k,l}|\geq 2\,\mathrm{sd}(\hat\beta_{k,l})$, where $\mathrm{sd}(\hat\beta_{k,l})$ denotes the sample standard deviation of $\hat\beta_{k,l}$ across the $100$ simulations. For the perturbed cases in which the graphon is $W_\beta$ plus a perturbation, coefficients are compared to those of the unperturbed component of the graphon as specified in (\ref{beta ex1}). 

In the following experiments, data are generated from a (perturbation of a) conditional intensity density corresponding to $\nu_\alpha$, $W_\beta$ and $h_\gamma$ as defined in (\ref{expansion basis finite}), using the parameters $\alpha$ and $\gamma$ given in (\ref{fixedalpha}), and the matrix $\beta$ taking the shape of one of the following four matrices:
\begin{align}
\beta^{(1)}&=\begin{bmatrix}
\frac{3}{4} & \frac{1}{4} & 0 & 0 & 0 & 0 & 0 & 0 \\
0 & 0 & 0 & 0 & 0 & 0 & 0 & 0 \\
0 & 0 & 0 & 0 & 0 & 0 & 0 & 0 \\
0 & 0 & 0 & 0 & 0 & 0 & 0 & 0 \\
0 & 0 & 0 & 0 & 0 & 0 & 0 & 0 \\
0 & 0 & \frac{3}{20} & 0 & 0 & 0 & 0 & 0 \\
0 & 0 & 0 & 0 & 0 & 0 & 0 & 0 \\
0 & 0 & -\frac{3}{20} & 0 & 0 & 0 & 0 & 0
\end{bmatrix};\quad\label{beta ex1}
\beta^{(2)}&&=\begin{bmatrix}
\frac{7}{10} & 0 & 0 & \frac{1}{4} & 0 & 0 & 0 & 0 \\
0 & 0 & 0 & 0 & 0 & 0 & 0 & 0 \\
0 & 0 & 0 & 0 & 0 & 0 & 0 & 0 \\
0 & 0 & 0 & 0 & 0 & 0 & 0 & 0 \\
0 & -\frac{1}{5} & 0 & 0 & 0 & 0 & 0 & 0 \\
0 & -\frac{1}{5} & 0 & 0 & 0 & 0 & 0 & 0 \\
0 & -\frac{1}{5} & 0 & 0 & 0 & 0 & 0 & 0 \\
0 & -\frac{1}{5} & 0 & 0 & 0 & 0 & 0 & 0
\end{bmatrix};\\
\beta^{(3)}&=\begin{bmatrix}
\frac{3}{4} & 0 & 0 & \frac{1}{6} & 0 & 0 & 0 & 0 \\
0 & 0 & 0 & 0 & 0 & 0 & 0 & 0 \\
0 & 0 & 0 & 0 & 0 & 0 & 0 & 0 \\
0 & 0 & 0 & 0 & 0 & 0 & 0 & 0 \\
0 & -\frac{1}{6} & 0 & 0 & \frac{1}{12} & 0 & 0 & 0 \\
0 & -\frac{1}{6} & 0 & 0 & 0 & 0 & 0 & 0 \\
0 & -\frac{1}{6} & 0 & 0 & 0 & 0 & \frac{1}{12} & 0 \\
0 & -\frac{1}{6} & 0 & 0 & 0 & 0 & 0 & 0
\end{bmatrix};\quad
\beta^{(4)}&&=\begin{bmatrix}
\frac{4}{5} & \frac{1}{12} & -\frac{1}{12} & \frac{1}{12} & 0 & 0 & 0 & 0 \\
0 & 0 & 0 & 0 & 0 & 0 & 0 & 0 \\
0 & 0 & 0 & 0 & 0 & 0 & 0 & 0 \\
0 & 0 & 0 & 0 & 0 & 0 & 0 & 0 \\
0 & -\frac{1}{12} & -\frac{1}{12} & 0 & -\frac{1}{12} & \frac{1}{12} & \frac{1}{12} & \frac{1}{12} \\
0 & -\frac{1}{12} & -\frac{1}{12} & 0 & -\frac{1}{12} & \frac{1}{12} & \frac{1}{12} & \frac{1}{12} \\
0 & -\frac{1}{12} & -\frac{1}{12} & 0 & \frac{1}{12} & -\frac{1}{12} & -\frac{1}{12} & -\frac{1}{12} \\
0 & -\frac{1}{12} & -\frac{1}{12} & 0 & \frac{1}{12} & -\frac{1}{12} & -\frac{1}{12} & -\frac{1}{12}
\end{bmatrix}.\nonumber
\end{align} 

\textbf{Unperturbed examples 1--4.}
In unperturbed example $i$, we consider $W_\beta$ with $\beta=\beta^{(i)}$ as defined in (\ref{beta ex1}), for $i\in[4]$. In these unperturbed examples, Assumption \ref{ass approx} holds trivially, while Assumption \ref{stability ass} is satisfied by our choice of the parameters $\alpha$, $\beta$, and $\gamma$. Assumption \ref{min eigenvalue ass} follows from the choice of basis functions, and Assumption \ref{ass PSD C0} holds because we consider a self-exciting process; see the discussion following that assumption.

\textbf{Perturbed examples.} To assess robustness to misspecification, we perturb $W_{\beta^{(1)}}$ in three different ways. The perturbations are chosen to probe distinct forms of approximation error: a component outside the fitted spatial basis, a high-frequency deterministic component, and a rough stochastic component. Thus, the experiments move from relatively structured misspecification to increasingly irregular departures from the fitted model class.

\begin{enumerate}
    \item  We first consider a structured misspecification outside the fitted spatial basis. More specifically, with $\varphi_{k}(x)$ as in (\ref{Haar wavelets}), we consider the dynamics  driven by 
\begin{equation}\label{W 1st perturbed}
W_{9,\ldots,16}(x,y):=W_{\beta^{(1)}}(x,y)+\varphi_9(x)\left(\sum_{k=9}^{16}\varphi_k(y)\right)/3. 
\end{equation}
This means that the true graphon is actually spanned by the truncated basis $\{\varphi_k\}_{k\in[16]}$, but not by the smaller basis $\{\varphi_k\}_{k\in[8]}$. This is a misspecification as the chosen basis is too small.
    \item We next consider a high-frequency deterministic misspecification, which is less well aligned with the low-dimensional spatial basis. More specifically, we consider dynamics  driven by 
\begin{equation}\label{W 2nd perturbed}
W_{\mathrm{Trig}}(x,y)=W_{\beta^{(1)}}(x,y)+2\sin(20x)\sin(30y)/3. 
\end{equation}
Since the perturbation consists of a high-frequency trigonometric function, one would hope that we can detect the signal spanned by $\{\varphi_k\}_{k\in[8]}$ despite this periodic ``noise''.
    \item Finally, we consider a rough stochastic perturbation, representing a substantially less regular departure from the fitted model class.  We consider a misspecification similar to the previous two example, but we take it a bit further by considering a Gaussian perturbation. More specifically, for each simulation, we generate standard Brownian motions $B_1$ and $B_2$ on $[0,1]$, and we consider dynamics driven by 
\begin{equation}\label{W 3rd perturbed}
W_{\mathrm{BM}}(x,y)=W_{\beta^{(1)}}(x,y)+B_1(x)B_2(y)/3.
\end{equation}
One would hope that we can detect the signal spanned by $\{\varphi_k\}_{k\in[8]}$ despite the perturbation by this noise.
\end{enumerate}

In Appendix~\hyperref[app C]{C}, we report the sample mean of estimated coefficients together with their standard deviations, computed across the $100$ simulations. For each of the unperturbed examples, and for the first two perturbed examples, the true coefficients are recovered with good accuracy, although some parameter shrinkage is observed. This shrinkage disappears in the limit $T\to\infty$, by Theorem \ref{consistency thm}, but is well known to arise in finite samples when using lasso-type procedures. In the case of Brownian perturbations, the coefficients remain recoverable on average; however, the standard errors are considerably larger than in the other examples, making the estimates less precise and less informative for individual simulations. In Table \ref{consistency table}, we report the average and standard deviation of the Frobenius norm $\|\hat\beta-\beta\|_F$ across $100$ simulated sample paths for the first unperturbed example, using $T=1\,000, 5\,000, 10\,000, 50\,000, 100\,000, 500\,000$. The systematic decrease of the Frobenius error with $T$ provides a direct finite-sample illustration of the consistency result: although the observation horizon is not large enough for the asymptotic approximation to be particularly sharp in the most complex example, the estimation error decreases steadily as the amount of data increases.

In Figure \ref{fig1} we present a histogram of the Frobenius norm $\|\hat\beta-\beta\|_F$ across the $100$ sample paths for each of the four unperturbed examples. The distance between the estimate and the true matrix of coefficients is small, but  positive. As previously noted, finite-sample parameter shrinkage leads to small but nonzero values of $\|\hat\beta-\beta\|_F$. The histogram appears approximately bell-shaped, indicating regular behavior of the estimator. Furthermore, in the fourth example, where $\beta$ contains many nonzero coefficients, the Frobenius norm $\|\hat\beta-\beta\|_F$ is larger than in the other examples, which is consistent with shrinkage affecting a greater number of coordinates.

\begin{table}[ht]
\centering
\begin{tabular}{c|c|c}
\hline
$T$ & $\mathrm{mean}(\|\hat\beta-\beta\|_F)$ & $\mathrm{sd}(\|\hat\beta-\beta\|_F)$ \\
\hline
$1\,000$ & .4221 & .1264\\
$5\,000$ & .3404 & .0690\\
$10\,000$ & .3202 & .0673\\
$50\,000$ & .2246 & .0280\\
$100\,000$ & .1766 & .0308\\
$500\,000$ & .0953 & .0207\\
\hline
\end{tabular}
\caption{For different time horizons $[0,T]$, we estimate $\beta$ by $\hat{\beta}$ using $100$ simulated sample paths of the first unperturbed example. We report the sample mean and standard deviation of the Frobenius norm $\|\hat{\beta} - \beta\|_F$ across these paths. The results indicate that the Frobenius norm decreases as $T$ increases.}
\label{consistency table}
\end{table}

\begin{figure}[ht]\centering
\begin{tabular}{cc}
\subfloat[\scriptsize First unperturbed graphon]{\includegraphics[width = 3in]{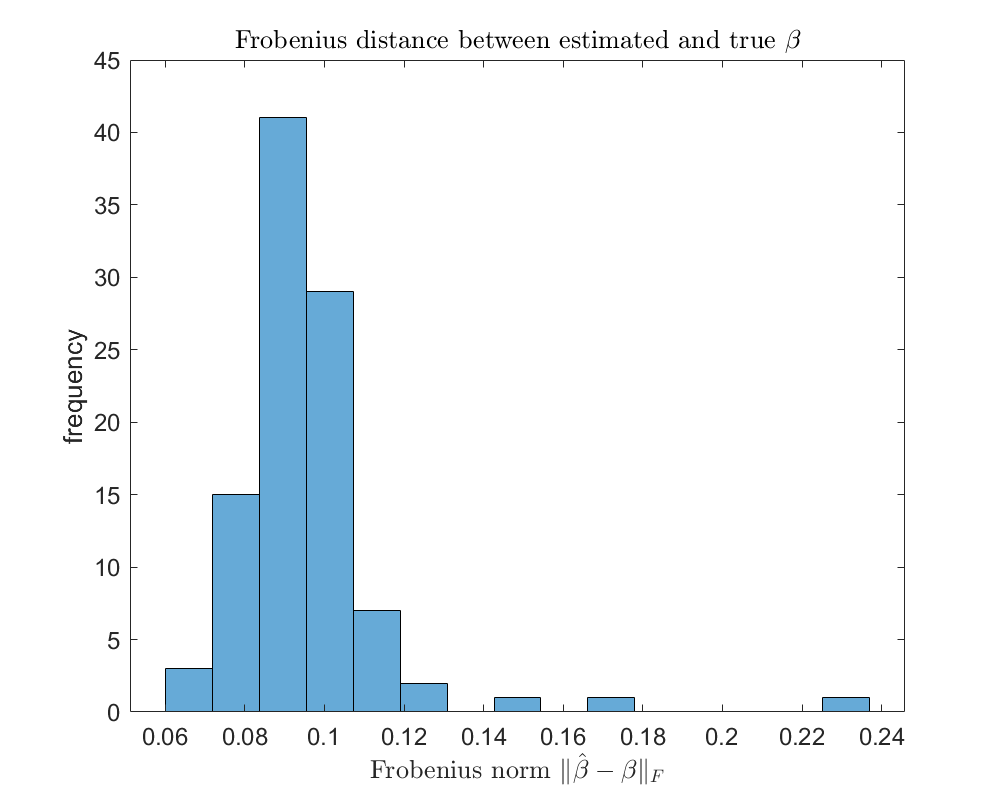}} &
\subfloat[\scriptsize Second unperturbed graphon]{\includegraphics[width = 3in]{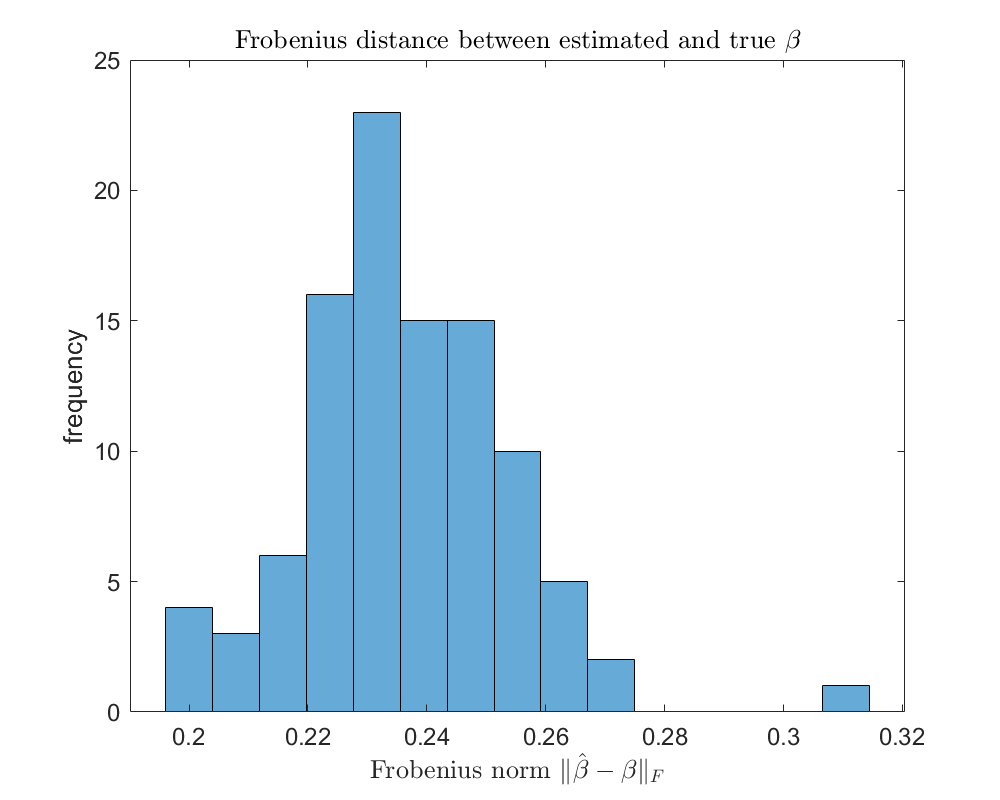}} \\
\subfloat[\scriptsize Third unperturbed graphon]{\includegraphics[width = 3in]{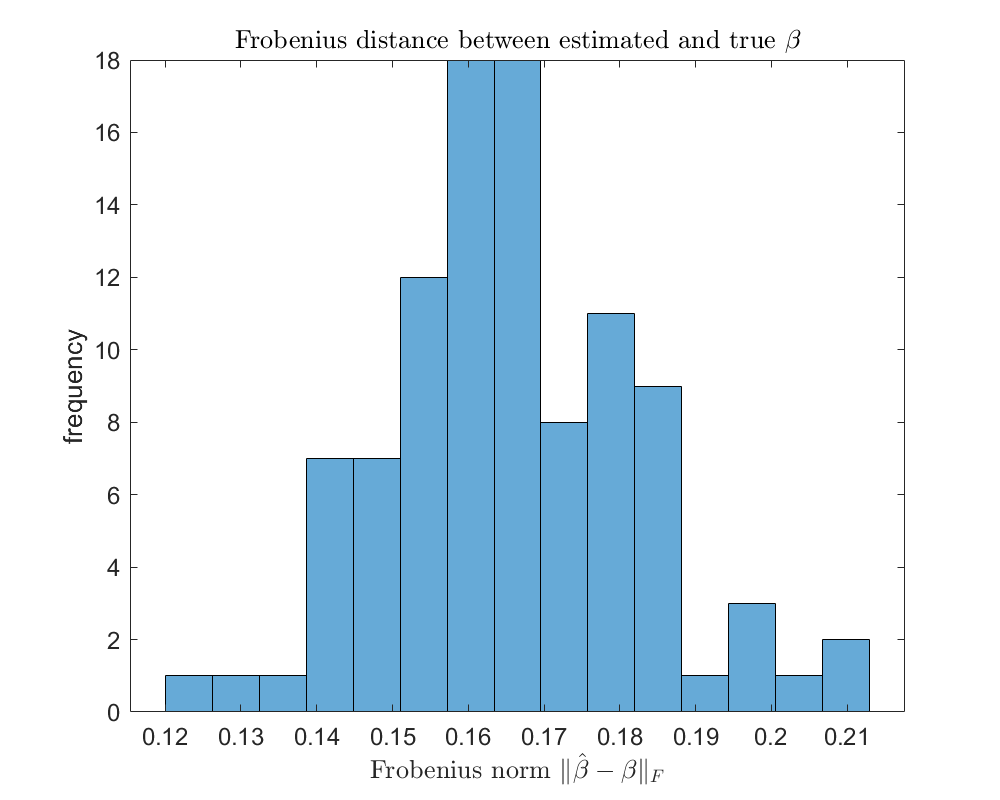}} &
\subfloat[\scriptsize Fourth unperturbed graphon]{\includegraphics[width = 3in]{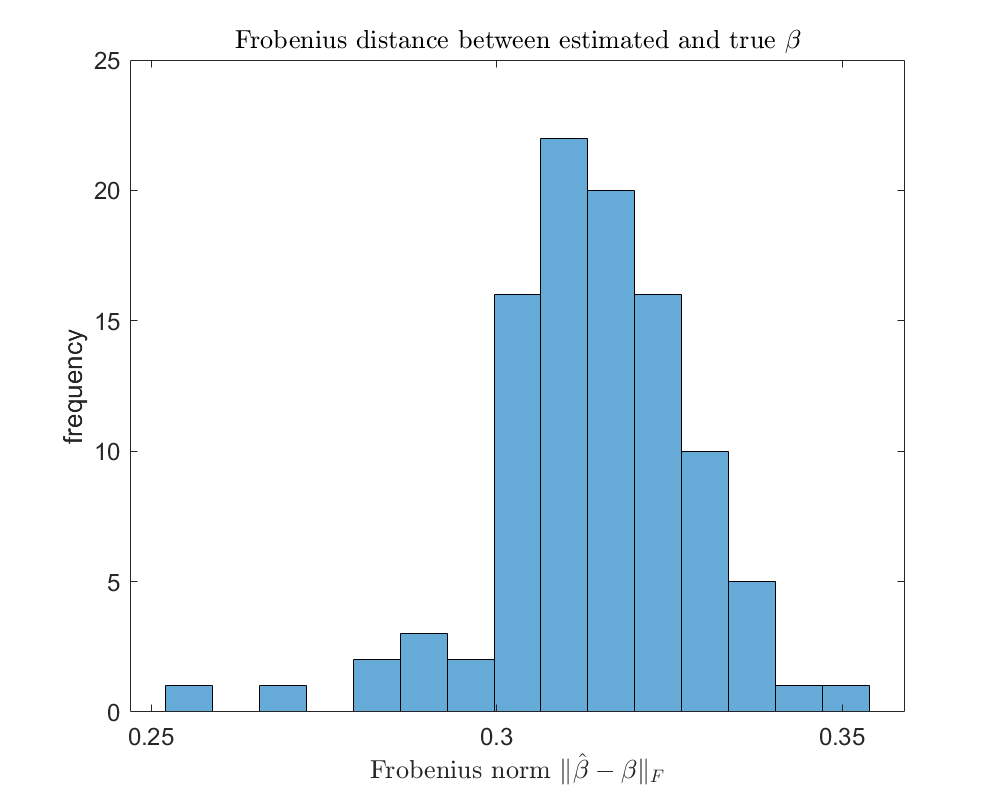}} 
\end{tabular}
\caption{\label{fig1} Using $T=500\,000$, for $100$ sample paths we estimate $\beta$ by $\hat\beta$. We present a histogram of the Frobenius norm $\|\hat\beta-\beta\|_F$ across these $100$ sample paths for each of the four unperturbed examples.}
\normalsize
\end{figure}

\begin{figure}[ht]
\centering
\begin{minipage}{0.48\textwidth}
\centering
\subfloat[\scriptsize First perturbed graphon]{\includegraphics[width=3in]{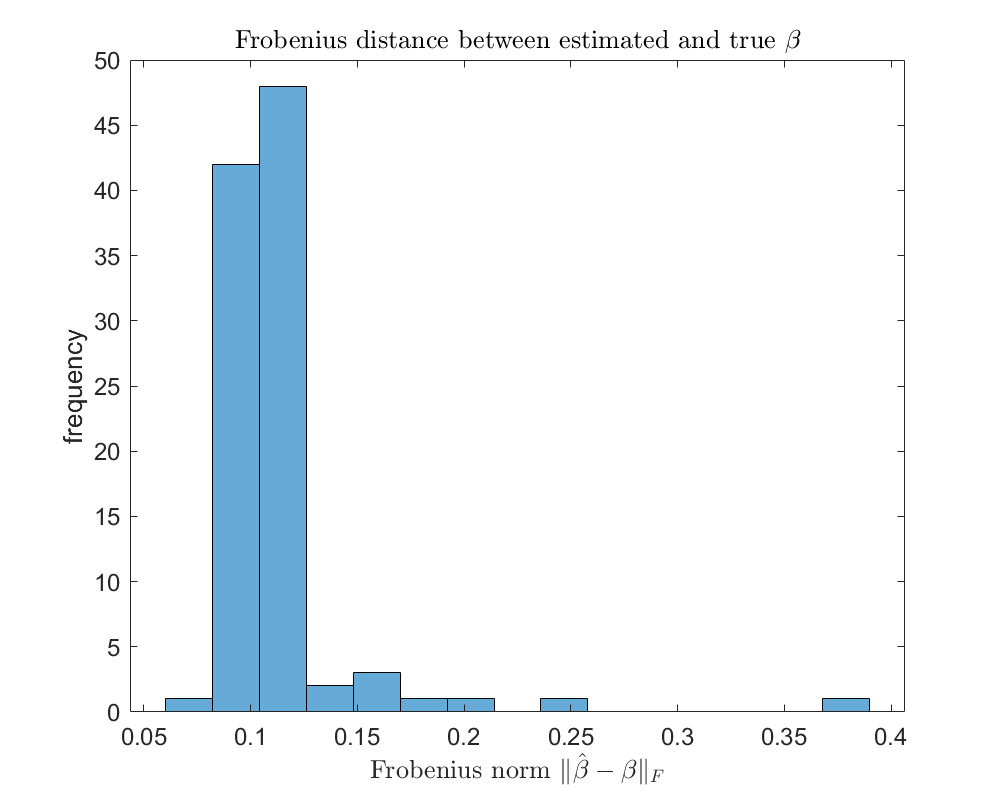}}
\end{minipage}
\hfill
\begin{minipage}{0.48\textwidth}
\centering
\subfloat[\scriptsize Second perturbed graphon]{\includegraphics[width=3in]{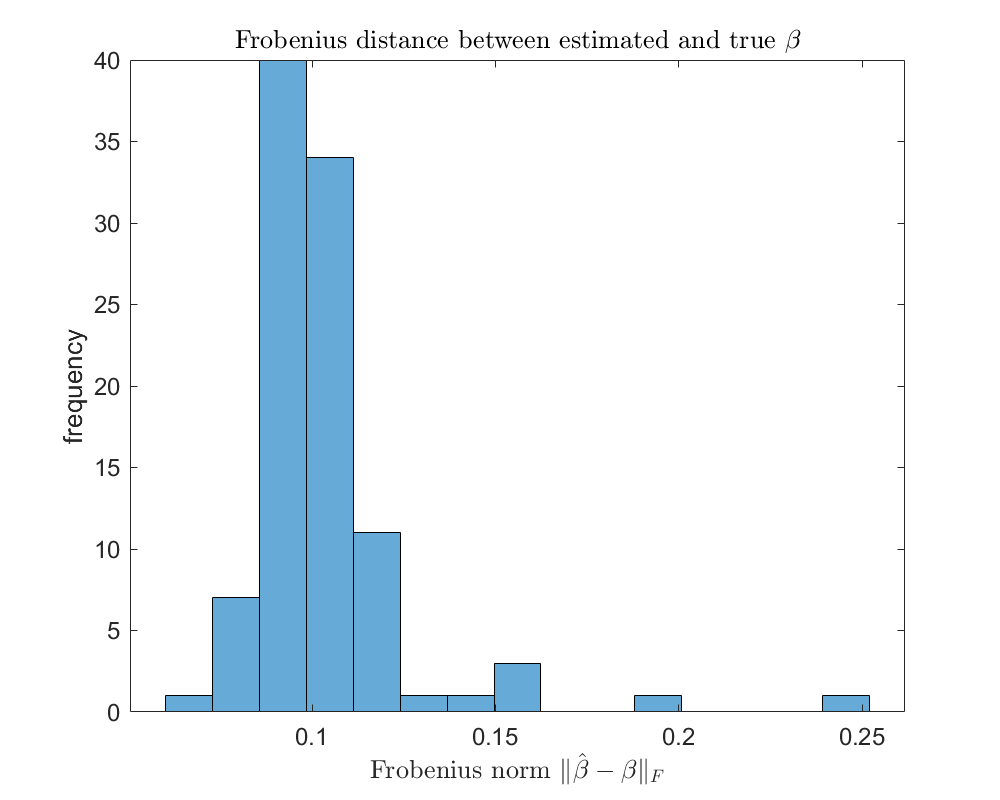}}
\end{minipage}
\vspace{0.8cm}
\begin{minipage}{0.48\textwidth}
\centering
\subfloat[\scriptsize Third perturbed graphon]{\includegraphics[width=3in]{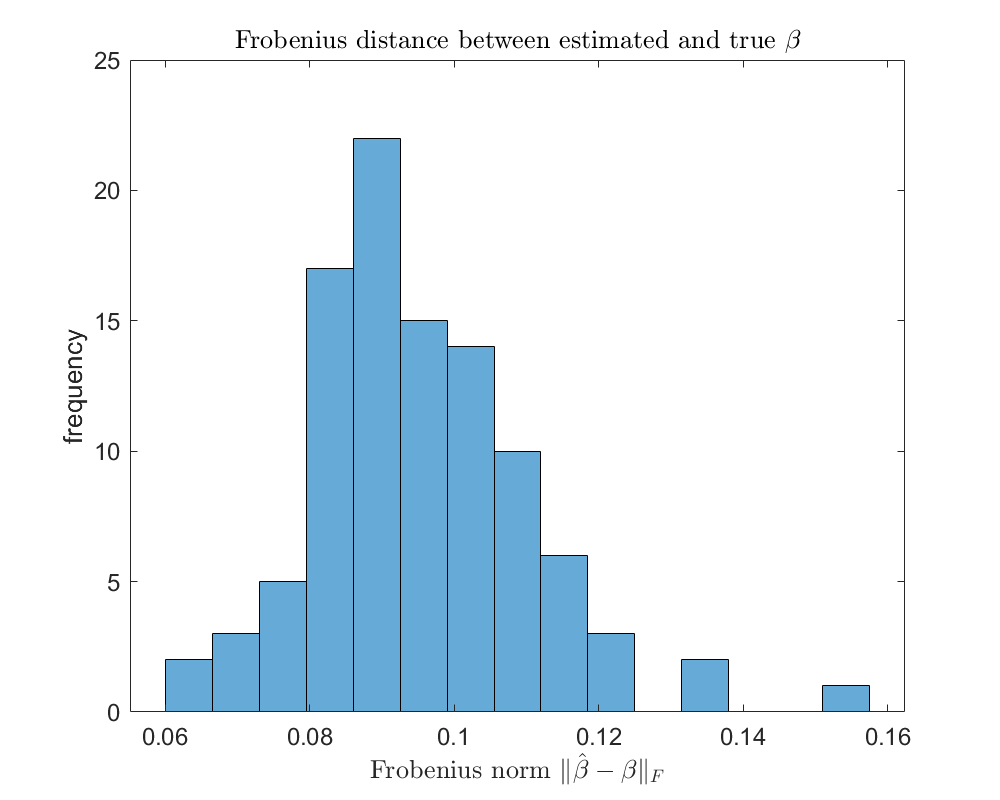}}
\end{minipage}
\caption{\label{fig2} Using $T=500\,000$, for $100$ sample paths we estimate $\beta$ by $\hat\beta$. We present a histogram of the Frobenius norm $\|\hat\beta-\beta\|_F$ across these $100$ sample paths for each of the three perturbed examples.}
\end{figure}

We present similar histograms for the three perturbed examples in Figure \ref{fig2}. For the first two examples, the Frobenius norms $\|\hat\beta-\beta\|_F$ appear to be slightly larger than in Figure \ref{fig1}(A), indicating that the perturbations introduce a modest loss in precision, but do not substantially deteriorate the overall estimation accuracy. For the third perturbed graphon, the Frobenius norms are of a similar magnitude as in the unperturbed case, suggesting that the impact of the perturbation is limited.
A natural explanation is that, in all three cases, the perturbations are approximately orthogonal to the chosen basis. Consequently, their contribution is only weakly represented in the coefficient matrix $\beta$, and do not cause a qualitative change in performance. This behavior is in line with the intuition underlying Assumption \ref{ass approx}, at least at a heuristic level. 

\textbf{Sparsity structure recovery of the graphon.}
In Table \ref{table1}, the mean false negative rate (FNR), false positive rate (FPR) and $F_1$-score over the $100$ simulations using $T=500\,000$ are reported. Between brackets, we also report the sample standard devations of those performance measures over the $100$ simulations.

\begin{table}[ht]
\centering
\begin{tabular}{l|l|ccc}
\hline
 & Model & FNR & FPR & $F_1$-score \\
\hline
\multirow{4}{*}{Unperturbed graphons}
 & $W_{\beta^{(1)}}$ & .0000 (.0000)  & .0250 (.0382)  & .8755 (.1477)  \\
 & $W_{\beta^{(2)}}$ & .0000 (.0000) & .0228 (.0351) & .9158 (.1019) \\
 & $W_{\beta^{(3)}}$ & .0425 (.0621)  & .0387 (.0983)  & .8977 (.1438)  \\
 & $W_{\beta^{(4)}}$ & .2779 (.0929)  & .0353 (.0981)  & .8169 (.0732)  \\
\hline
\multirow{3}{*}{Perturbed graphons}
 & $W_{9,\ldots,16}$ & .0000 (.0000)  & .0437 (.1141)  & .8670 (.2009)  \\
 & $W_{\mathrm{Trig}}$ & .0000 (.0000) & .0310 (.0535)  & .8511 (.1452)  \\
 & $W_{\mathrm{BM}}$ & .0000 (.0000) & .0363 (.0855)  & .8716 (.1903)  \\
\hline
\end{tabular}
\caption{Sample mean and standard deviation over the $100$ performed simulations for several performance measures for the unperturbed and perturbed examples described in this section, using $T=500\,000$.}
\label{table1}
\end{table}

We observe that false negatives are rare in the first three perturbed examples and in the first two unperturbed examples, indicating that the nonzero coefficients are typically identified with good accuracy. In the fourth (non-sparse) example, where the graphon contains $28$ nonzero coefficients out of $64$, some coefficients lie close to zero. Due to parameter shrinkage, these small coefficients are more difficult to detect, and a lasso-type procedure favouring sparsity may estimate small coefficients by zero, explaining a higher FNR. This should be interpreted as a finite-sample effect rather than a failure in terms of  support recovery.

False positives occur infrequently across all models. The $F_1$-score is high for the four unperturbed scenarios, indicating accurate variable selection. Although it decreases slightly for the models with perturbed graphons, it remains at a strong level. This suggests that our estimation procedure is robust, even when the true data-generating graphon only approximately belongs to the specified model class.

\section{Application to Allen Neuropixel data set}\label{data study section}
In this section, we apply our penalized least-squares estimation procedure for the graphon Hawkes process to real data. We analyze an electrophysiological experiment that originates from session \texttt{715093703} of the Neuropixels extracellular recording collection maintained by the Allen Institute for Brain Science
\cite{AllenNeuroData}. 
The dataset comprises time-stamped spike events and corresponding probe depths.  Let $\{(t_i, d_i)\}_{i=1}^{N}$ denote the set of spike events, where $t_i$ is the spike time (in seconds) and $d_i$ is the probe depth (in micrometers). 
The continuous spatial structure of the recording makes the graphon Hawkes framework particularly natural: unlike a multivariate Hawkes process, which requires assigning neurons to a finite collection of discrete dimensions, the graphon formulation allows neuronal interactions to vary continuously with spatial location. This avoids an arbitrary discretization of the neuronal population while providing a natural framework for capturing spatially structured network interactions.

The recording was obtained from a head-fixed laboratory mouse (\textit{Mus musculus}) in an awake state while the animal was presented with controlled visual stimuli on a monitor. Neuropixels probes were inserted into cortical and subcortical regions, enabling high-density measurements of extracellular voltage deflections associated with neuronal action potentials. These signals were subsequently spike-sorted to infer single-unit activity, yielding time-stamped spike events and depth annotations for each unit. The experimental paradigm is designed to study population-level responses to sensory input and to provide benchmark data for computational models of neural dynamics.


\begin{figure}[ht]\centering
\begin{tabular}{cc}
\subfloat[\scriptsize Normalized baseline intensity (Legendre polynomials)]{\includegraphics[width = 3in]{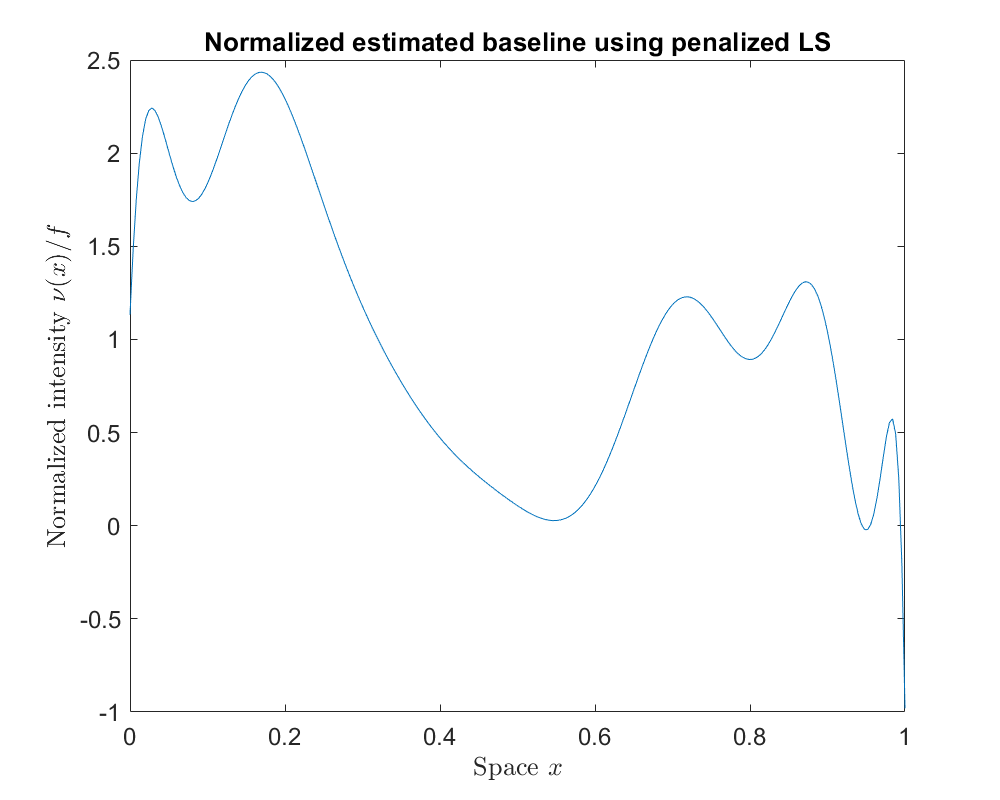}} &
\subfloat[\scriptsize Normalized graphon (Legendre polynomials)]{\includegraphics[width =3in]{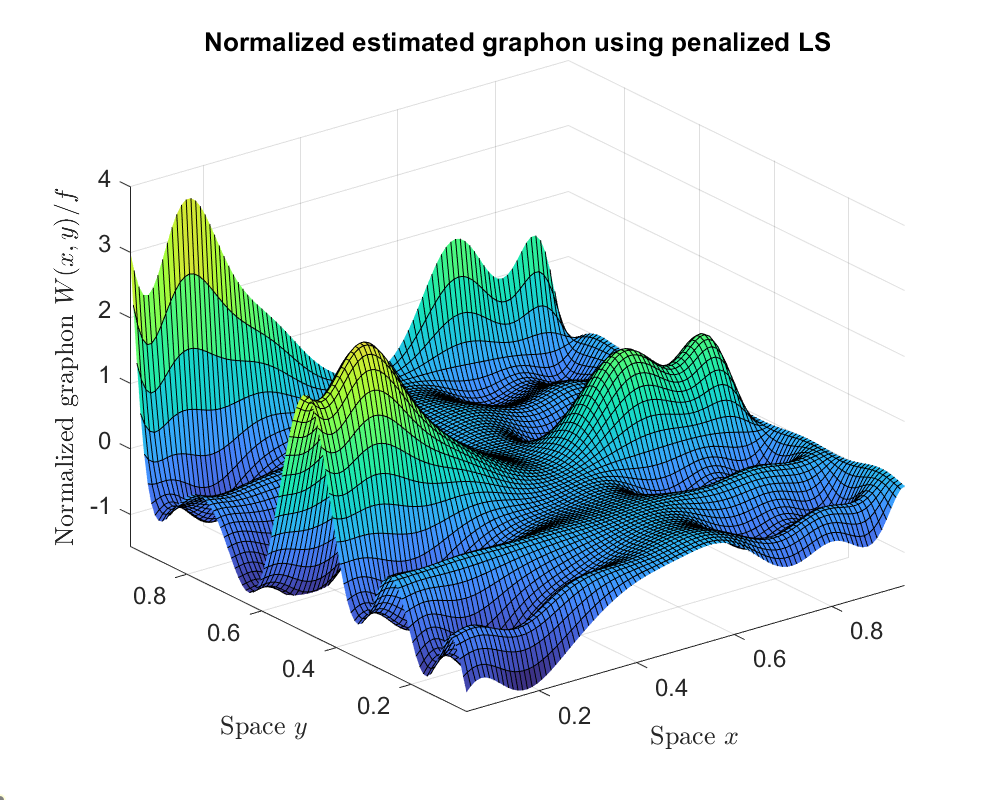}} \\
\subfloat[\scriptsize Normalized baseline intensity (Haar wavelets)]{\includegraphics[width = 3in]{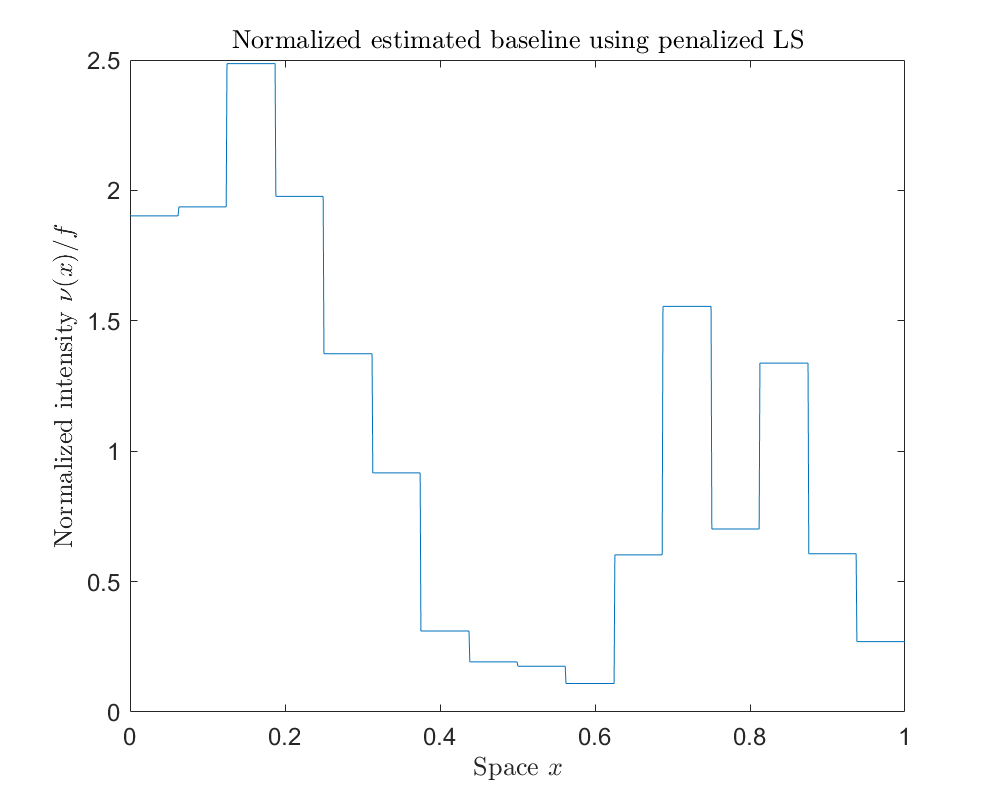}} &
\subfloat[\scriptsize Normalized graphon (Haar wavelets)]{\includegraphics[width = 3in]{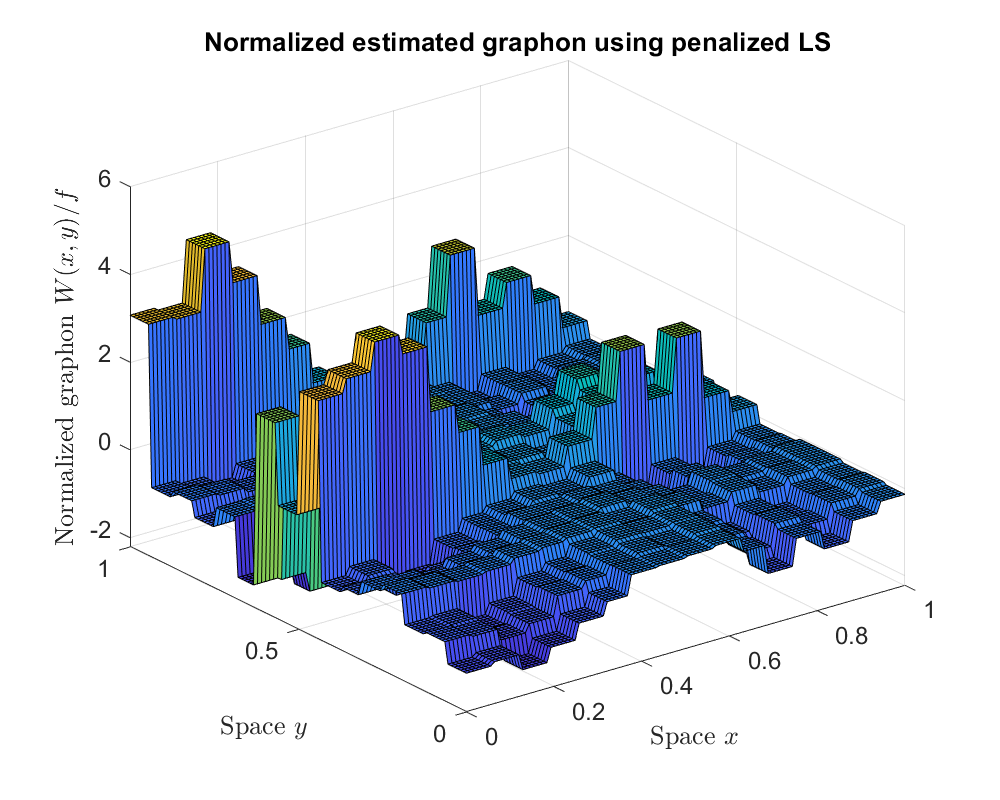}} 
\end{tabular}
\caption{\label{fig3} Comparison of penalized graphon Hawkes estimates for the Allen Neuropixels session \texttt{715093703}, based on spike events in the time interval $[5000,5250]$. Panels (A)–(B) display the estimated baseline intensity and graphon using a spatial basis of Legendre polynomials, whereas panels (C)–(D) use a spatial basis of Haar wavelets. All quantities are normalized by the empirical frequency $f$. The estimator based on Legendre polynomials produces smoother and more spatially regular estimates, while the results in panels (C)–(D) appear as discretized variants of those in panels (A)–(B). 
}
\normalsize
\end{figure}

\begin{figure}[ht]\centering
\begin{tabular}{cc}
\subfloat[\scriptsize Normalized baseline intensity (Legendre polynomials)]{\includegraphics[width = 3in]{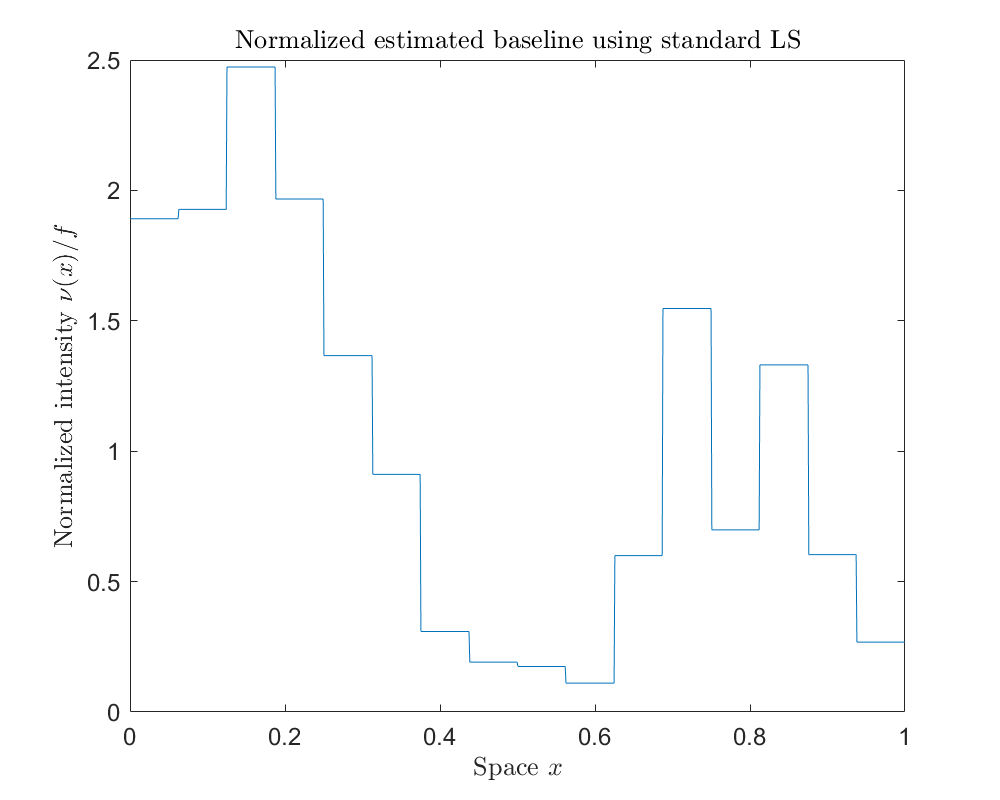}} &
\subfloat[\scriptsize Normalized graphon (Legendre polynomials)]{\includegraphics[width =3in]{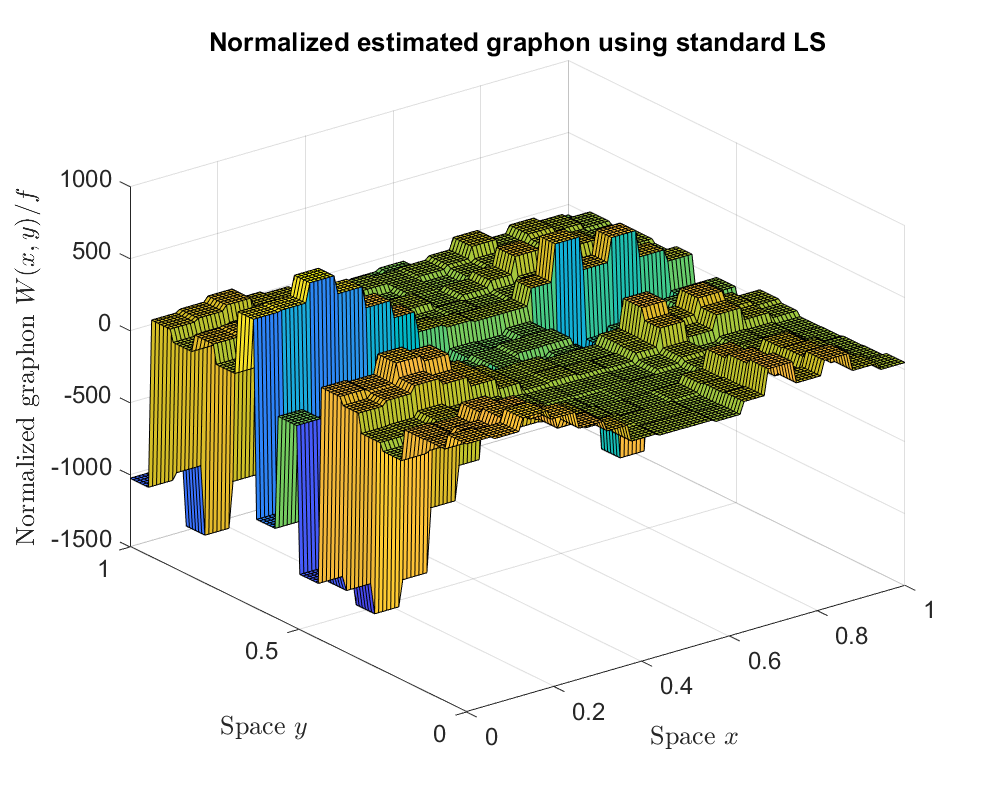}}
\end{tabular}
\caption{\label{fig4} Non-penalized least-squares estimates for the Allen Neuropixels session \texttt{715093703}, based on spike events in the interval $[5000,5250]$, are shown in Panels (A)–(B), where a spatial basis of Haar wavelets is used with $\eta=0$. All quantities are normalized by the empirical frequency $f$. In contrast to the penalized estimates in Figure \ref{fig3}(C)--(D), the non-penalized solution exhibits pronounced spatial irregularity, with large oscillations and substantial negative values of the graphon over parts of the domain. These findings indicate that, while the unregularized estimator can capture fine-scale variation, it is highly sensitive to sampling noise and produces estimates that are difficult to interpret in terms of underlying connectivity structure. 
}
\normalsize
\end{figure}

Fitting multivariate Hawkes processes to neuronal spike trains is by now standard practice in the analysis of neural population dynamics; see, for example, \cite{Bonnet2022, Bonnet2023, Cai2024, Chen2019, Donnet2020, Fallahi, Hansen, Wang2024, Zhang2016}. A spatiotemporal Hawkes process with a graphon-induced network structure was introduced in \cite{Baars2024} as a high-dimensional limit of multivariate Hawkes processes. In many neuroscience applications, the spatial domain is discretized so that each coordinate of a multivariate Hawkes process corresponds to a predefined spatial region or recording site. By contrast, the present dataset provides inherently continuous spatial information via probe depth. This motivates estimating the graphon Hawkes process directly to the non-discretized neuronal data, thereby avoiding an artificial partition of space. 

To control computational complexity and for illustrative purposes, we restrict the analysis to event times in the $250$-second interval $[5000,5250]$, comprising approximately $1.81$ million events. The spatial coordinate is observed on the discrete grid $\{20,40,\ldots,3280\}$, corresponding to probe-depth bins measured in micrometers. To obtain a continuous spatial coordinate, we apply a standard jittering procedure (cf.~\cite{Schoenberg2019}, Section 7) by adding independent $\mathrm{Uniform}(-10,10)$ perturbations to the recorded depths and linearly rescaling the resulting range $(10,3290)$ to the unit interval $\mathscr X=[0,1]$. This preserves the bin frequencies while avoiding artificial ties induced by the discretization and yields a continuous representation suitable for kernel-based estimation of the graphon connectivity function.

We fit the graphon Hawkes process to the data using a truncated spatial basis $(\varphi_k)_{k\in[M]}$ consisting of the first $M=16$ Legendre polynomials on $\mathscr X=[0,1]$, as given in (\ref{Haar wavelets}). 
The choice $M=16$ provides sufficient spatial resolution while keeping the resulting optimization problem computationally tractable; as a robustness check, we also considered $M=8$, for which the estimated patterns are qualitatively similar but somewhat less detailed. 
Let $f=1812778/250\approx7251.1\ \mathrm{Hz}$ denote the empirical event frequency in the observation window $[5000,5250]$. 
For the temporal component, we use $H=5$ exponential basis functions $t\mapsto\exp(-r_mft)$, with $r_1=1$, $r_2=2$, $r_3=5$, $r_4=1/2$, and $r_5=1/5$. 
This choice provides a range of decay rates, from relatively fast to relatively slow excitation, while keeping the temporal representation low-dimensional. 
For comparison, we also fit the model using a spatial basis consisting of Haar wavelets on $[0,1]$. The regularization parameter in the penalized specification is chosen to be $\eta=\log(T)/(4\sqrt{T})$, where $T=250$ is the length of the observation window. 
Estimates of the baseline intensity and graphon are normalized by the empirical event frequency $f$ to enable comparison of relative magnitudes across specifications. 


The baseline and graphon estimates are displayed in Figure \ref{fig3}. Remarkably, the Legendre- and Haar-based specifications yield consistent connectivity patterns, despite their substantially different spatial representations, indicating that the dominant features of the estimated graphon are not artifacts of the particular choice of basis. The Legendre basis represents these patterns as a smooth function of spatial location, whereas the Haar basis captures the same underlying structure through localized, piecewise-constant components. The qualitative agreement across the two specifications indicates that the estimated connectivity structure is robust to the choice of spatial basis. We note that the results were qualitatively similar when using $M=8$ Legendre polynomials instead of $M=16$, although the estimates were somewhat less detailed.

We interpret the estimated parameters of the graphon Hawkes process from a neurological perspective. 
The estimated baseline intensity $\nu(x)$ varies across $x\in[0,1]$, with high values for $x\in[0,1/4]$, near zero around $x=1/2$, and moderate values for $x\in[3/4,1]$. In the Hawkes framework, the baseline captures activity that is not explained by endogenous interactions, including spontaneous firing and stimulus-driven responses that are effectively independent of network activity. Accordingly, high $\nu(x)$ corresponds to neurons with substantial intrinsic or broadly responsive activity, whereas the near-zero baseline around $x=1/2$ indicates that little of the observed activity of these neurons is attributed to an exogenous baseline component by the fitted model.
This pattern suggests subpopulations with distinct firing characteristics, consistent with heterogeneity across cell types, cortical layers, or stimulus responsiveness.
Neurons around $x=1/2$ also exhibit weak incoming connectivity in the estimated graphon. Taken together, these estimates suggest that their activity is less strongly associated with recurrent network interactions. This interpretation is consistent with the broader perspective of low-dimensional and stimulus-related structure in neural population activity \cite{Archer2014}.

The estimated interaction kernel is approximately separable, $W(x,y)\approx W_X(x)W_Y(y)$, 
with $W_X(x)$ resembling $\nu(x)$. This factorization indicates that network interactions can be interpreted as the combination of target susceptibility and source influence. Populations with high $W_Y(y)$, near $y=1$ and $y=1/2$, act as stronger drivers of network activity. Such low-rank or separable structures are consistent with theoretical models in which population-level dynamics emerge from low-dimensional connectivity patterns \cite{Beiran2020, MastrogiuseppeOstojic2018}.



To assess the effect of penalization, we also implement the least-squares estimation procedure proposed in this paper using a spatial basis consisting of Haar wavelets, and using $\eta=0$, corresponding to the classical, non-penalized least-squares estimator. 
Figure \ref{fig4} displays the corresponding estimates of the baseline intensity and the graphon, each normalized through division by the frequency $f$. As can be seen from Figure \ref{fig4}, the non-penalized estimator produces erratic estimates with high spatial variability, reflected in large oscillations and substantial negative graphon values over parts of the domain. While this specification achieves a close empirical fit to the observed spike data, the resulting estimates are difficult to interpret from a neuroscientific perspective and suggest overfitting to sampling noise. The lack of regularization permits strong local variability, which undermines interpretability of the estimated connectivity pattern.

\section{Concluding remarks}\label{discussion section}
In this paper we studied statistical inference for linear graphon Hawkes processes observed on a continuous spatial domain. We proposed a penalized least-squares estimation procedure based on truncated orthogonal basis expansions for the baseline intensity, interaction graphon, and excitation kernel. The resulting estimator allows the number of parameters to grow with the observation horizon while promoting sparsity in the effective interaction structure. We established non-asymptotic error bounds for the estimated intensity, proved consistent recovery of the latent interaction network under suitable sparsity conditions, and illustrated the method through simulation experiments and an application to neuronal spike train data.

Several extensions and follow-up questions arise.
\begin{itemize}
\item
In this paper we considered graphon Hawkes processes starting from an empty history. The analysis can be extended in a straightforward way to the case where the process has a nonempty history on $(-b,0)$. Under Assumption~\ref{stability ass}(vi), the excitation kernel has bounded support, so that allowing a nonempty history on $(-b,0)$ effectively accommodates arbitrary past histories.
\item
The proof of Theorem~\ref{consistency thm} indicates that several variations of the proposed estimation procedure can be accommodated with only minor modifications. In particular, as noted in Remark~\ref{possible extensions remark}, the baseline coefficients $\alpha$ may be included in the regularization term without affecting the resulting probabilistic error bounds. Moreover, the $\ell^1$ penalty used in the present work can be replaced by other regularization schemes, such as ridge, elastic net, or standardized group lasso penalties \cite{Simon2013}. Since the key arguments rely primarily on the quadratic structure of the least-squares contrast and the bounds derived from the Gram matrix representation, extending the analysis to these alternative penalties is straightforward.
\item
The present work focuses on linear graphon Hawkes processes defined on a continuous spatial domain with a stationary baseline intensity. Our contribution lies in extending the framework to infinite-dimensional, spatiotemporal processes and removing the assumption that the conditional intensity is almost surely bounded. While frameworks such as \cite{Cai2024} can accommodate nonstationary and nonlinear intensities, we restricted attention to the linear and stationary setting to keep the scope manageable. Investigating estimation and recovery in more general nonlinear or nonstationary spatiotemporal Hawkes models is a natural direction for future work, but is beyond the scope of the present study.
\item
Another natural extension concerns statistical inference based on the estimated interaction structure. In particular, it would be interesting to develop formal hypothesis tests for the presence or absence of interactions, or for structural properties of the graphon $W$. Such procedures would rely on suitable limit theorems for the proposed estimator or for related score-type statistics. Developing these results would require a substantially more detailed asymptotic analysis and is therefore left for future work.
\item
Finally, while the present work focuses on theoretical guarantees and proof-of-concept applications, further work is needed to study the computational aspects of the proposed estimator in large-scale settings. In particular, developing scalable optimization algorithms and data-driven procedures for selecting the truncation levels and regularization parameters would be important for applying the methodology to very large datasets.

\end{itemize}

\section*{Acknowledgements}
The research of Justin Baars and Roger Laeven is funded in part by the Netherlands Organization for Scientific Research under an NWO VICI grant (2020--2027). The research of Michel Mandjes was funded in part by the NWO Gravitation project N{\sc etworks}, grant number 024.002.003. 
The authors report that there are no competing interests to declare.

\section*{Appendix A: Concentration inequalities and auxiliary technical results}\label{app A}

\subsection*{Concentration inequalities for first- and second-order functionals}
First, we establish establish auxiliary results to prove the concentration inequalities for functionals of the graphon Hawkes process, as given in Theorem \ref{concentration ineq}.

In the following coupling lemma, let $\bar\lambda$ be the solution to the Fredholm integral equation of the second kind $\bar\lambda=\nu +T_{\mathrm{hom}}(\bar\lambda)$, where   
$T_{\mathrm{hom}}$
is defined in \eqref{operator Thom marked general X}.
\begin{lemma}\label{couplingdeviationslemma}
Grant Assumption \ref{stability ass}. Let $N$ be a Hawkes process with intensity given in (\ref{conditional int dens}), starting from the empty history. 
For any positive time $z>0$, there exists a process $\tilde N\stackrel d= N$, such that $\tilde N$ is independent of the history of $N$ up to time $z$. 
Furthermore, for any $u\geq0$ and $n=\lfloor u/b\rfloor +1$ it holds that 
\begin{align}
\mathbb E\left|\tilde N(z+\mathrm du,\mathrm ds)-N(z+\mathrm du,\mathrm ds)\right|/(\mathrm du\,\mathrm ds)&\leq 2T_{\mathrm{hom}}^n(\bar\lambda)(s);\\
\mathbb E\,\Big|\tilde N(\mathrm dt,\mathrm ds_1)\tilde N(z+\mathrm du,\mathrm ds_2)&-N(\mathrm dt,\mathrm ds_1)N(z+\mathrm du,\mathrm ds_2)\Big|/(\mathrm du\,\mathrm dt\,\mathrm ds_1\,\mathrm ds_2)\nonumber\\
&\leq C \sum_{i=1}^{n+1}T_{\mathrm{hom}}^{i-1}(\mathrm{id})(s_1)\,T_{\mathrm{hom}}^{n-i+1}\,(\mathrm{\bar\lambda})(s_2),
\end{align}
where $C:=\max\{2(\overline\nu+C_W\|h\|_{L^\infty(\mathbb R_+)}),C_{\max}\}$, with $C_{\max}$ defined in (\ref{crude estimate}) and finite.
\end{lemma}
\begin{proof} We first prove the existence of the coupling $\tilde N$ of $N$. Then, we study the first-order deviations between $N$ and $\tilde N$. We finish by studying the second-order deviations.

\emph{Step 1}. We prove the existence of $\tilde N$.
Let $\breve N$ be a Poisson point process (PPP) on $\mathbb R^2\times\mathscr X$ with measure $\mathrm{Leb}\times \mathrm{Leb}\times \mu$, independent of the PPP $\bar N$ that is used for the canonical construction of $N$, see \cite{Baars2024}, Section 3, and the proof of \cite{Baars2024}, Theorem 1. Note that by $\mu(\mathscr X)<\infty$, we need only a single process $\bar N$ to construct $N$. Similarly to the Picard scheme given from \cite{Baars2024}, we construct $\tilde N^{(n)}$ with corresponding conditional intensity density $\tilde \lambda^{(n)}$. First let $\tilde N^{(0)}$ be the empty process consisting of no events. Then, for $n\in\mathbb N_0$, construct $\tilde N^{(n)}$ as 
\begin{align}
\tilde\lambda^{(n+1)}(t,x)&:=\nu(x)+\int_{(0,t)\times\mathscr X}W(x,y)h(t-s)\, \tilde N^{(n)}(\mathrm ds,\mathrm dy);\nonumber\\
\label{picard coupled}\tilde N^{(n+1)}(\mathrm dt,\mathrm dx)&:=\begin{cases}
  \breve N\left(\left[0,\tilde\lambda^{(n+1)}(t,x)\right]\times \mathrm dt\times \mathrm dx\right) & \text{if } t \leq z; \\
  \bar N\left(\left[0,\tilde\lambda^{(n+1)}(t,x)\right]\times \mathrm dt\times \mathrm dx\right) & \text{if } t > z.
\end{cases}
\end{align}
Note that one could write $$\tilde N^{(n+1)}(\mathrm dt,\mathrm dx)=\tilde N'\left(\left[0,\tilde\lambda^{(n+1)}(t,x)\right]\times \mathrm dt\times \mathrm dx\right),$$ where $$\tilde N'(\mathrm dt,\mathrm dx):=\mathbf1\{t\leq z\}\breve N(\mathrm dt,\mathrm dx)+\mathbf1\{t>z\}\bar N(\mathrm dt,\mathrm dx)$$ is a PPP on $\mathbb R^2\times\mathscr X$ with measure $\mathrm{Leb}\times \mathrm{Leb}\times \mu$. Therefore, the representation given in (\ref{picard coupled}) is equivalent to the standard Picard construction, whence $\tilde N\stackrel d= N$. 

In order to show that $\tilde N$ is independent of the history of $N$ up to time $z$, note that $N$ up to time $z$ is constructed using only $\bar N$ up to time $z$. Hence it suffices to show that $\tilde N$ is independent of $\mathcal H_z$, where $\mathcal H_t:=\sigma(\bar N(s):s\leq t)$. To this end, note that $\tilde N'$ is independent of $\mathcal H_z$, hence $\tilde N^{(1)}$ is independent of $\mathcal H_z$. Now suppose that $\tilde N^{(n)}$ is independent of $\mathcal H_z$. As $\tilde\lambda^{(n+1)}(t,x)$ is a function of $\tilde N^{(n)}$, it follows that $\tilde\lambda^{(n+1)}(t,x)$ is independent of $\mathcal H_z$. As $\tilde N^{(n+1)}(\mathrm dt,\mathrm dx)$ is a function of $\tilde\lambda^{(n+1)}(t,x)$ and of $\tilde N'$, which are both independent of $\mathcal H_z$, it follows that $\tilde N^{(n)}$ is independent of $\mathcal H_z$. By induction, $\tilde N^{(n)}$ is independent of $\mathcal H_z$ for each $n\in\mathbb N$. Therefore, the same holds for their limit $\tilde N$.

\emph{Step 2}. We study the first-order deviations between $N$ and $\tilde N$. We aim to show that for any $u>(n-1)b$, $n\in\mathbb N$, it holds that 
\begin{equation}\label{first order deviation IH}
\mathbb E\left|\tilde N(z+\mathrm du,\mathrm ds)-N(z+\mathrm du,\mathrm ds)\right|/(\mathrm du\,\mathrm ds)\leq 2T_{\mathrm{hom}}^n(\bar\lambda)(s),
\end{equation}
where $T_{\mathrm{hom}}$ is defined in (\ref{operator Thom marked general X}), and where $\mathrm{id}\in L^1(\mathscr X,\mu)$ is the identity operator. 

We prove the claim by induction. First consider $n=1$. We have 
\begin{align}
&\mathbb E\left|\tilde N(z+\mathrm du,\mathrm ds)-N(z+\mathrm du,\mathrm ds)\right|/(\mathrm du\,\mathrm ds)\nonumber\\
&=\mathbb E\left|\tilde N'\left(\left[0,\tilde\lambda(z+u,s)\right]\times \mathrm dt\times \mathrm ds\right)-\bar N\left(\left[0,\lambda(z+u,s)\right]\times \mathrm dt\times \mathrm ds\right)\right|/(\mathrm du\,\mathrm ds)\nonumber\\
&=\mathbb E\left|\tilde\lambda(z+u,s)-\lambda(z+u,s)\right|\nonumber\\
&\leq \int_{[0,b]\times\mathscr X}W(s,y)h(t)\, \mathbb E\left|\mathrm d\tilde N(z+u-t,y)-N(z+u-t,y)\right|\nonumber\\
&\leq2\int_{[0,b]\times\mathscr X}W(s,y)h(t)\bar\lambda(y)\, \mathrm dy\, \mathrm dt\leq2\|h\|_{L^1(\mathbb R_+)}\int_{\mathscr X}W(s,y)\bar\lambda(y) \, \mathrm dy\nonumber\\
&=2T_{\mathrm{hom}}(\mathrm{\bar\lambda})(s).\label{ineq 1st order dev}
\end{align}
In the second inequality, we use the triangle inequality and $$\mathbb E\left|\mathrm dN(z+u-t,y)\right|/(\mathrm dt\,\mathrm dy)\leq\mathbb E\left|\mathrm dN_{\mathrm{stat}}(z+u-t,y)\right|/(\mathrm dt\,\mathrm dy)=\bar\lambda(y),$$ by coupling our process starting in the empty history to the process $N_{\mathrm{stat}}$ starting in stationarity. 

Now suppose that (\ref{first order deviation IH}) holds true for some $n-1$, $n\geq2$, so for $u>(n-2)b$. The first inequality from (\ref{ineq 1st order dev}) now reads 
\begin{align}
&\mathbb E\left|\tilde N(z+\mathrm du,\mathrm ds)-N(z+\mathrm du,\mathrm ds)\right|/(\mathrm du\,\mathrm ds)\nonumber\\
&\leq \int_{[0,b]\times\mathscr X}W(s,y)h(t)\, \mathbb E\left|\mathrm d\tilde N(z+u-t,y)-N(z+u-t,y)\right|\nonumber.
\end{align}
Then, for $u>(n-1)b$, we may apply the induction hypothesis for $n-1$, since $t\in[0,b]$, so $u-t>(n-2)b$. This gives 
\begin{align}
\mathbb E\left|\tilde N(z+\mathrm du,\mathrm ds)-N(z+\mathrm du,\mathrm ds)\right|/(\mathrm du\,\mathrm ds)&\leq\nonumber2\int_{[0,b]\times\mathscr X}W(s,y)h(t)T_{\mathrm{hom}}^{n-1}(\mathrm{\bar\lambda})(y)\, \mathrm dy\, \mathrm dt\\
&=2T_{\mathrm{hom}}(T_{\mathrm{hom}}^{n-1}(\mathrm{\bar\lambda}))(s)=2T_{\mathrm{hom}}^n(\bar\lambda)(s),
\end{align}
providing the induction step. This proves the bound for the first-order devations between $\tilde N$ and $N$.

\emph{Step 3}. We study the second-order deviations between $N$ and $\tilde N$. 

\emph{Step 3A}. We start by giving a crude bound on the second-order deviations, which we will use in Step 3B. For $z$ as given in the theorem, and $\mathbb P_{\mathrm{stat}}$ the law of the process $N_{\mathrm{stat}}$ starting in stationarity, we have 
\begin{align}
&\mathbb E\left|\tilde N(\mathrm dt,\mathrm ds_1)\tilde N(z+\mathrm du,\mathrm ds_2)-N(\mathrm dt,\mathrm ds_1)N(z+\mathrm du,\mathrm ds_2)\right|/(\mathrm du\,\mathrm dt\,\mathrm ds_1\,\mathrm ds_2)\nonumber\\
&\leq2\mathbb E\left|N(\mathrm dt,\mathrm ds_1)N(z+\mathrm du,\mathrm ds_2)\right|/(\mathrm du\,\mathrm dt\,\mathrm ds_1\,\mathrm ds_2)\nonumber\\
&=2\mathbb E\left|\bar N([0,\lambda(t,s_1)]\times\mathrm dt\times\mathrm ds_1)\bar N([0,\lambda(z+u,s_2)]\times\mathrm d(z+u)\times\mathrm ds_2)\right|/(\mathrm du\,\mathrm dt\,\mathrm ds_1\,\mathrm ds_2)\nonumber\\
&=2\mathbb E[\lambda(t,s_1)\lambda(u+z,s_2)]\nonumber\\
&\leq2\max_{x\in\{s_1,s_2\}}\mathbb E_{\mathrm{stat}}[\lambda(t,x)^2]
\nonumber\\
&\leq2\max_{x\in\{s_1,s_2\}}\bigg\{\nu(x)^2+2\nu(x)\int_{(0,t)\times\mathscr X}W(x,y)h(t-s)\,\mathbb E_{\mathrm{stat}}[N(\mathrm ds,\mathrm dy)]\nonumber\\
&+\int_{(0,t)\times\mathscr X}\int_{(0,t)\times\mathscr X}W(x,y_1)W(x,y_2)h(t-s_1)h(t-s_2)\,\mathbb E_{\mathrm{stat}}[N(\mathrm ds_1,\mathrm dy_1)]\,\mathbb E_{\mathrm{stat}}[N(\mathrm ds_2,\mathrm dy_2)]\nonumber\\
&+\int_{(0,t)\times\mathscr X}W(x,y)^2h(t-s)^2\,\mathbb E_{\mathrm{stat}}[N(\mathrm ds,\mathrm dy)]
\bigg\}\nonumber\\
&\leq 2\max_{x\in\{s_1,s_2\}}\left\{(\nu(x)+T_{\mathrm{hom}}(\mathrm{\bar\lambda})(x))^2+\frac{\|h\|_{L^2(\mathbb R_+)}^2}{\|h\|_{L^1(\mathbb R_+)}}C_WT_{\mathrm{hom}}(\mathrm{\bar\lambda})(x)\right\}=:C_{\max}.\label{crude estimate}
\end{align}
The first inequality follows by the triangle inequality and since $\tilde N\stackrel{d}=N$; the second inequality follows by Cauchy-Schwarz and by coupling our process starting in the empty history to the process $N_{\mathrm{stat}}$ starting in stationarity; the third equality follows by separating the on- and off-diagonal contributions. Recalling that $\bar\lambda$ satisfies $\bar\lambda=\nu +T_{\mathrm{hom}}(\bar\lambda)$, and since $\rho(T_{\mathrm{hom}})<1$, the operator $(\mathrm{id}-T_{\mathrm{hom}})^{-1}\in B(L^1(\mathscr X))$ is bounded. Next, because $\|\nu\|_{L^1(\mathscr X)}<\infty$, it holds that $\|T_{\mathrm{hom}}(\mathrm{\bar\lambda})\|_{L^1(\mathscr X)}<\infty$. 
From $\sup_{x\in\mathscr X}|T_{\mathrm{hom}}(\mathrm{id})(x)|\leq\varrho_{T_{\mathrm{hom}}}$, it follows that $\|T_{\mathrm{hom}}\|_{L^\infty(\mathscr X)\to L^\infty(\mathscr X)}<1$, hence $\|\bar\lambda\|_{L^\infty(\mathscr X)}<\infty$. From this, it follows that the right-hand side of (\ref{crude estimate}) is bounded; i.e., $C_{\max}<\infty$.

\emph{Step 3B}. We look for a more refined bound on the second order deviations. We aim to show that for any $u>(n-1)b$, $n\in\mathbb N$, it holds that 
\begin{align}\nonumber
&\mathbb E\left|\tilde N(\mathrm dt,\mathrm ds_1)\tilde N(z+\mathrm du,\mathrm ds_2)-N(\mathrm dt,\mathrm ds_1)N(z+\mathrm du,\mathrm ds_2)\right|/(\mathrm du\,\mathrm dt\,\mathrm ds_1\,\mathrm ds_2)\\
&\leq C \sum_{i=1}^{n+1}T_{\mathrm{hom}}^{i-1}(\mathrm{id})(s_1)T_{\mathrm{hom}}^{n-i+1}(\mathrm{\bar\lambda})(s_2).\label{IH 2nd order dev}
\end{align}
We again use induction. First consider $n=1$. Let $u>0$ and $t\geq z$. We have 
\begin{align}
&\mathbb E\left|\tilde N(\mathrm dt,\mathrm ds_1)\tilde N(z+\mathrm du,\mathrm ds_2)-N(\mathrm dt,\mathrm ds_1)N(z+\mathrm du,\mathrm ds_2)\right|/(\mathrm du\,\mathrm dt\,\mathrm ds_1\,\mathrm ds_2)\nonumber\\
&=\mathbb E\left|\tilde N'([0,\lambda(t,s_1)]\times\mathrm dt\times\mathrm ds_1)\tilde N(z+\mathrm du,\mathrm ds_2)-\bar N([0,\tilde\lambda(t,s_1)]\times\mathrm dt\times\mathrm ds_1)N(z+\mathrm du,\mathrm ds_2)\right|/(\mathrm du\,\mathrm dt\,\mathrm ds_1\,\mathrm ds_2)\nonumber\\
&=\mathbb E\left|\tilde N(z+\mathrm du,\mathrm ds_2)\tilde \lambda(t,s_1)-N(z+\mathrm du,\mathrm ds_2)\lambda(t,s_1)\right|/(\mathrm du\,\mathrm ds_2)\nonumber\\
&\leq\overline\nu\ \mathbb E\left|\tilde N(z+\mathrm du,\mathrm ds_2)-N(z+\mathrm du,\mathrm ds_2)\right|/(\mathrm du\mathrm ds_2)\nonumber\\
&+\mathbb E\left|\int_{(0,t)\times \mathscr X}W(s_1,l)h(t-t')\left[\tilde N(\mathrm dt',\mathrm dl)\tilde N(z+\mathrm du,\mathrm ds_2)-N(\mathrm dt',\mathrm dl) N(z+\mathrm du,\mathrm ds_2)\right]\right|/(\mathrm du\,\mathrm ds_2)\nonumber\\
&\leq\overline\nu\ \mathbb E\left|\tilde N(z+\mathrm du,\mathrm ds_2)-N(z+\mathrm du,\mathrm ds_2)\right|/(\mathrm du\,\mathrm ds_2)\nonumber\\
&+\left|W(s_1,s_2)h(t-u-z)\right|\mathbb E\left|\tilde N(z+\mathrm du,\mathrm ds_2)- N(z+\mathrm du,\mathrm ds_2)\right|/(\mathrm du\,\mathrm ds_2)\nonumber\\
&+\int_{(0,t)\times \mathscr X}\left|W(s_1,l)h(t-t')\right|(1-\mathbf1\{l=s_2,t=u+z\})\nonumber\\
&\quad\times\mathbb E\left|\tilde N(\mathrm dt',\mathrm dl)\tilde N(z+\mathrm du,\mathrm ds_2)-N(\mathrm dt',\mathrm dl) N(z+\mathrm du,\mathrm ds_2)\right|/(\mathrm du\,\mathrm ds_2)\nonumber\\
&\leq2(\overline\nu+C_W\|h\|_{L^\infty(\mathbb R_+)})T_{\mathrm{hom}}(\mathrm{\bar\lambda})(s_2)+C_{\max}T_{\mathrm{hom}}(\mathrm{id})(s_1)\nonumber\\
&\leq C\left(T_{\mathrm{hom}}(\mathrm{id})(s_1)+T_{\mathrm{hom}}(\mathrm{\bar\lambda})(s_2)\right),\label{2nd order induction start}
\end{align}
where $C:=\max\{2(\overline\nu+C_W\|h\|_{L^\infty(\mathbb R_+)}),C_{\max}\}$.
The second equality uses that $\mathrm d\bar N(t)=\mathrm d\tilde N'(t)$ for $t\geq z$; note that we decompose into on- and off-diagonal contributions in the second inequality. The second to last inequality uses the crude bound from Step 3A.

Now suppose that (\ref{IH 2nd order dev}) holds true for some $n-1$, $n\geq2$, so for $u>(n-2)b$. Then for $n$, so $u>(n-1)b$, we have, using the inequality given in  (\ref{2nd order induction start})
\begin{align}
&\mathbb E\left|\tilde N(\mathrm dt,\mathrm ds_1)\tilde N(z+\mathrm du,\mathrm ds_2)-N(\mathrm dt,\mathrm ds_1)N(z+\mathrm du,\mathrm ds_2)\right|/(\mathrm du\,\mathrm dt\,\mathrm ds_1\,\mathrm ds_2)\nonumber\\
&\leq(\overline\nu+C_W\|h\|_{L^\infty(\mathbb R_+)})\mathbb E\left|\tilde N(z+\mathrm du,\mathrm ds_2)-N(z+\mathrm du,\mathrm ds_2)\right|/(\mathrm du\,\mathrm ds_2)\nonumber\\
&+\int_{(t-b,t)\times \mathscr X}\left|W(s_1,l)h(t-t')\right|(1-\mathbf1\{l=s_2,t=u+z\})\nonumber\\
&\quad\times\mathbb E\left|\tilde N(\mathrm dt',\mathrm dl)\tilde N(z+\mathrm du,\mathrm ds_2)-N(\mathrm dt',\mathrm dl) N(z+\mathrm du,\mathrm ds_2)\right|/(\mathrm du\,\mathrm ds_2)\nonumber\\
&\leq C T_{\mathrm{hom}}^n(\bar\lambda)(s_2)+C\int_{(t-b,t)\times \mathscr X}\left|W(s_1,l)h(t-t')\right|\sum_{i=1}^{n}T_{\mathrm{hom}}^{i-1}(\mathrm{id})(l)T_{\mathrm{hom}}^{n-i}(\mathrm{\bar\lambda})(s_2)\,\mathrm dt'\,\mathrm dl\nonumber\\
&\leq C\left(T_{\mathrm{hom}}^n(\bar\lambda)(s_2)+\sum_{i=1}^{n}T_{\mathrm{hom}}^{n-i}(\mathrm{\bar\lambda})(s_2)T_{\mathrm{hom}}(T_{\mathrm{hom}}^{i-1}(\mathrm{id}))(s_1)\right)\nonumber\\
&=C \sum_{i=1}^{n+1}T_{\mathrm{hom}}^{i-1}(\mathrm{id})(s_1)T_{\mathrm{hom}}^{n-i+1}(\mathrm{\bar\lambda})(s_2),\label{2nd order induction result}
\end{align}
where the second inequality follows by the induction hypothesis, which is applicable since $h$ has finite memory of length at most $b$.
This gives the induction step, hence we proved (\ref{IH 2nd order dev}).
\end{proof}

For later use, we summarize the findings of Step 3A of the proof of Lemma \ref{couplingdeviationslemma} in a separate result.
\begin{lemma}\label{step3Alemma}
Grant Assumption \ref{stability ass}. Let $N$ be a Hawkes process with intensity given in (\ref{conditional int dens}), starting from the empty history. 
Consider the first- and second-order mean intensities $\bar\lambda^{(1)}(t,x)$ and $\bar\lambda^{(2)}(t_1,x_1,t_2,x_2)$ defined in (\ref{1st order mean intensity})--(\ref{2nd order mean intensity}). It holds that we can find constants $c_i>0$ such that $\|\bar\lambda^{(i)}\|_{L^\infty}\leq c_i$ for $i=1,2$. Furthermore, it holds that $\|\bar\lambda\|_{L^\infty(\mathscr X)}<\infty$.
\end{lemma}

\begin{lemma}\label{expdecaylemma}
Grant Assumption \ref{stability ass}. Let $N$ be a Hawkes process with intensity given in (\ref{conditional int dens}), starting from the empty history. 
For any given positive time $z>0$, let $\tilde N\stackrel d= N$ be the coupling process constructed in Lemma~\ref{couplingdeviationslemma}.
Then it holds for $u\geq0$ 
\begin{align}
&\mathbb E\left|\tilde N(z+\mathrm du,\mathrm ds)-N(z+\mathrm du,\mathrm ds)\right|/(\mathrm du\,\mathrm ds)\leq c_1\exp(-c_2u);\label{exp 1st order}\\
&\mathbb E\left|\tilde N(\mathrm dt,\mathrm ds_1)\tilde N(z+\mathrm du,\mathrm ds_2)-N(\mathrm dt,\mathrm ds_1)N(z+\mathrm du,\mathrm ds_2)\right|/(\mathrm du\,\mathrm dt\,\mathrm ds_1\,\mathrm ds_2)\leq c_3\exp(-c_4u),\label{exp 2nd order}
\end{align}
where $c_i$, $i\in[4]$, are constants depending on $b,\varrho_{T_{\mathrm{hom}}}, h, W, \nu$.

\end{lemma}
\begin{proof}
We have seen in the proof of Lemma \ref{couplingdeviationslemma} that $\|\bar\lambda\|_{L^\infty(\mathscr X)}<\infty$. Hence for $n=\lfloor \frac ub+1\rfloor$ it holds that 
\begin{equation}
\mathbb E\left|\tilde N(z+\mathrm du,\mathrm ds)-N(z+\mathrm du,\mathrm ds)\right|/(\mathrm du\,\mathrm ds)\leq  2\|\bar\lambda\|_{L^\infty(\mathscr X)}T_{\mathrm{hom}}^n(\mathrm{id})(s)\leq 2\|\bar\lambda\|_{L^\infty(\mathscr X)}\varrho_{T_{\mathrm{hom}}}^n.
\end{equation}
Set $c_1=2\|\bar\lambda\|_{L^\infty(\mathscr X)}$ and $c_2=-\log(\varrho_{T_{\mathrm{hom}}})/b$, then (\ref{exp 1st order}) follows.

Similarly, for $n=\lfloor \frac ub+1\rfloor$ and $C$ as defined after (\ref{2nd order induction start}) it holds that 
\begin{align}
&\mathbb E\left|\tilde N(\mathrm dt,\mathrm ds_1)\tilde N(z+\mathrm du,\mathrm ds_2)-N(\mathrm dt,\mathrm ds_1)N(z+\mathrm du,\mathrm ds_2)\right|/(\mathrm du\,\mathrm dt\,\mathrm ds_1\,\mathrm ds_2)\\
&\leq C \sum_{i=1}^{n+1}T_{\mathrm{hom}}^{i-1}(\mathrm{id})(s_1)T_{\mathrm{hom}}^{n-i+1}(\mathrm{\bar\lambda})(s_2)\leq 2C\|\bar\lambda\|_{L^\infty(\mathscr X)}(n+1)\varrho_{T_{\mathrm{hom}}}^n\leq 2C\|\bar\lambda\|_{L^\infty(\mathscr X)}(n+1)\exp(-c_2nb).\nonumber
\end{align}
We can find $n_0$ such that for $n\geq n_0$, $$(n+1)\exp(-c_2nb)\leq\exp(-c_2nb/2).$$ Hence for $\lfloor \frac ub+1\rfloor\geq n_0$, 
\begin{equation}
\mathbb E\left|\tilde N(\mathrm dt,\mathrm ds_1)\tilde N(z+\mathrm du,\mathrm ds_2)-N(\mathrm dt,\mathrm ds_1)N(z+\mathrm du,\mathrm ds_2)\right|/(\mathrm du\,\mathrm dt\,\mathrm ds_1\,\mathrm ds_2)\leq 2C\|\bar\lambda\|_{L^\infty(\mathscr X)}\exp(-c_2nb/2).
\end{equation}
Therefore, setting $c_3=2C\|\bar\lambda\|_{L^\infty(\mathscr X)}(n_0+1)$ and $c_4=c_2/2$, (\ref{exp 2nd order}) can be seen to hold for any $u\geq0$.
\end{proof}

\phantom{.}\newline
{\it Proof of Theorem~\ref{concentration ineq}}.
We divide the proof according to the order of the functional. For the first-order functional, we partition the time interval into blocks of length $2\epsilon$ and write $y_1$ as the empirical mean of the resulting block variables $(y_{1,i})$. We then verify the two conditions of \cite{Merlevede2011}, Theorem~1: an exponential tail bound for $y_{1,i}$ and an exponential bound on its temporal dependence coefficient. The latter is obtained by coupling the process with an independent copy. Theorem~1 then yields (\ref{1st order bernstein}).
The second-order case follows the same construction, with block variables $(y_{2,i})$. Here, the tail bound has exponent $1/2$, reflecting the product structure of the second-order functional. The coupling argument again gives exponential decay of the temporal dependence coefficient. Applying \cite{Merlevede2011}, Theorem~1 with this tail exponent yields (\ref{2nd order bernstein}).
The overall argument follows the Bernstein-type strategy of \cite{Cai2024, Chen2019}, but is adapted here to spatiotemporal point processes rather than finite-dimensional multivariate processes.

\emph{First-order functionals}. For a small constant $\epsilon>0$ chosen such that $\frac T{2\epsilon}\in\mathbb N$, we consider the partial empirical averages 
\begin{equation}\label{y_i def}
y_{1,i}:=\frac1{2\epsilon}\int_{[2\epsilon(i-1),2\epsilon i]\times\mathscr X}f_1(t,x)\,N(\mathrm dt,\mathrm dx),\quad 1\leq i\leq \frac T{2\epsilon}.
\end{equation}
Note that $y_1$ can be seen as the sample average over $(y_{1,i})_{1\leq i\leq \frac T{2\epsilon}}$, i.e., 
\begin{equation}\label{sample mean eq}
y_1=\overline{y_{1,\cdot}}=\frac{2\epsilon}T\sum_{i=1}^{T/(2\epsilon)}y_{1,i}.
\end{equation}
To quantify the temporal dependence in the sequence $(y_{1,i})_{1\leq i\leq T/(2\epsilon)}$, we use the $\tau$-dependence coefficient of \cite{DedeckerPrieur2004}. For a gap $l\in\mathbb N$, this coefficient is given by
\begin{equation}\label{temporal dependence}
\tau_y(l)=\sup_{u\in\{0,1,\ldots,T/(2\epsilon)-l\}}\tau\left(\mathcal H_u^y,y_{1,u+l}\right),
\end{equation}
where $\mathcal H_u^y:=\sigma(y_{1,v}:1\leq v\leq u)$. For a random variable $X$ and a $\sigma$-algebra $\mathcal M$, the coefficient $\tau(\mathcal M,X)$ is defined as
$$
\tau(\mathcal M,X):=\mathbb E\left[\sup_{f\in L^\infty:\|f\|_{L^\infty}\leq1}\left|\int f(x)\,\mathbb P_{X|\mathcal M}(\mathrm dx)-\int f(x)\,\mathbb P_X(\mathrm dx)\right|\right].
$$
Moreover, \cite{DedeckerPrieur2004}, Lemma~3, provides the bound
\begin{equation}\label{lemma3DP04}
\tau(\mathcal M,X)\leq\mathbb E|X-Y|,
\end{equation}
where $Y\stackrel{d}{=}X$ is independent of $\mathcal M$.

In this proof, we establish that $(y_{1,i})_{1\leq i\leq T/(2\epsilon)}$ forms a weakly dependent sequence, in the sense that the two conditions in Equations (2.6)–(2.8) of \cite{Merlevede2011} are satisfied. The desired Bernstein-type inequality then follows directly from Theorem 1 of \cite{Merlevede2011}.

First, we prove Equation (2.7) in \cite{Merlevede2011}: we show, for constants $c_1,c_2>0$,
\begin{equation}
\sup_{1\leq i\leq T/(2\epsilon)}\mathbb P(|y_{1,i}|\geq x)\leq c_1\exp(c_2x), \quad x>0.\label{1st order cond1}
\end{equation}
To prove (\ref{1st order cond1}), note that by boundedness of $f_1$ it holds that $$|y_{1,i}|\leq\frac{\|f_1\|_{L^\infty}}{2\epsilon}N([2\epsilon(i-1),2\epsilon i]\times\mathscr X).$$ 
Therefore, we look for an exponential bound on the tail of $N([2\epsilon(i-1),2\epsilon i]\times\mathscr X)$. As $N$ starts from the empty history, it holds that 
\begin{equation}
\mathbb P\left(N([2\epsilon(i-1),2\epsilon i]\times\mathscr X)\geq n \right)\leq\mathbb P_{\mathrm{stat}}\left(N([2\epsilon(i-1),2\epsilon i]\times\mathscr X)\geq n \right),\label{exp bound stationary}
\end{equation}
where
$\mathbb P_{\mathrm{stat}}$ denotes the law of the process $N_{\mathrm{stat}}$ starting in stationarity.
By stationarity, the right-hand side of (\ref{exp bound stationary}) is independent of $i$, hence it suffices to obtain a bound of the form 
\begin{equation}
\mathbb P_{\mathrm{stat}}\left(N([0,2\epsilon ]\times\mathscr X)\geq n \right)\leq c_1\exp(-c_2 n), \quad x>0.\label{1st order exp bound general form}
\end{equation}
Since $\overline\nu<\infty$, our process consists of a stream of immigrants that can be bounded by a stream of immigrants of rate $\overline\nu\mu(\mathscr X)$ with uniform spatial coordinates. In turn, for each of those immigrants, since $\sup_{x\in\mathscr X}|T_{\mathrm{hom}}(\mathrm{id})(x)|\leq\varrho_{T_{\mathrm{hom}}}$, the size of the total offspring for each spatially uniformly distributed immigrant can be bounded in expectation by $({1-\varrho_{T_{\mathrm{hom}}}})^{-1}$. 
This essentially reduces the dimension of the system to $1$ (see, e.g., \cite{Baars2024}, Remark 1 and Section 4.2). 
Existence of exponential moments for the bounding process (and therefore, for $N$) follows by standard results on existence of exponential moments for Hawkes processes or for Galton-Watson branching processes, e.g., \cite{Leblanc2025b,Leblanc2025,KarimLaevenMandjes2025,AL26}. 
Hence, there exists $\theta^*>0$ such that for $0<\theta<\theta^*$
$$
\mathbb E_{\mathrm{stat}}\left[\exp(\theta N([0,2\epsilon ]\times\mathscr X))\right]<\infty. 
$$
By a standard Chernoff bound, (\ref{1st order exp bound general form}) follows.

Second, we prove Equation (2.6) in \cite{Merlevede2011}. 
We show, for $\tau_y(l)$ defined in (\ref{temporal dependence}), for constants $c_3,c_4>0$,
\begin{equation}\label{1st order cond2}
\tau_y(l)\leq c_3\exp(-c_4 l), \quad l\in\mathbb N.
\end{equation}
As in Lemma \ref{couplingdeviationslemma}, let $\tilde N\stackrel d= N$ be the process independent of $\mathcal H_{2\epsilon z}$, the $\sigma$-algebra generated by $N$ up to time $2\epsilon z$. We construct the counterpart $\tilde y_{1,i}$ of $y_{1,i}$ through $\tilde N$:
\begin{equation}\label{tilde y def}
\tilde y_{1,i}=\frac1{2\epsilon}\int_{[2\epsilon(i-1),2\epsilon i]\times\mathscr X}f_1(t,x)\,\tilde N(\mathrm dt,\mathrm dx),\quad 1\leq i\leq \frac T{2\epsilon}.
\end{equation}
Since  $\tilde N\stackrel d= N$, it follows that $(\tilde y_{1,i})_{1\leq i\leq \frac T{2\epsilon}} \stackrel d= (y_{1,i})_{1\leq i\leq \frac T{2\epsilon}}$. Furthermore, since $\tilde N$ is independent of $\mathcal H_{2\epsilon z}$, it follows that $(\tilde y_{1,i})_{1\leq i\leq \frac T{2\epsilon}}$ is independent of $( y_{1,i})_{1\leq i\leq z}$. Furthermore, by Lemma \ref{expdecaylemma} it holds that 
\begin{align}
\mathbb E\left|\tilde y_{1,z+l}-y_{1,z+l}\right|&=\frac1{2\epsilon}\mathbb E\left|\int_{[2\epsilon(z+l-1),2\epsilon (z+l)]\times\mathscr X}f_1(t,x)\left(\tilde N(\mathrm dt,\mathrm dx)-N(\mathrm dt,\mathrm dx)\right)\right|\nonumber\\
&\leq\frac1{2\epsilon}\int_{[2\epsilon(z+l-1),2\epsilon (z+l)]\times\mathscr X}\left|f_1(t,x)\right|\mathbb E\left|\tilde N(\mathrm dt,\mathrm dx)-N(\mathrm dt,\mathrm dx)\right|\nonumber\\
&\leq\frac1{2\epsilon}2\epsilon\mu(\mathscr X)\|f_1\|_{L^\infty}c_5\exp(-c_6(l-1)),
\end{align}
giving an exponential bound of the form (\ref{1st order cond2}).

Using (\ref{1st order cond1}) and (\ref{1st order cond2}), we may apply \cite{Merlevede2011}, Theorem 1, with $\gamma_1=\gamma_2=1$, hence $\gamma=\frac12$. This gives, for $\delta\in(0,1)$, a bound of the form
\begin{align}\label{Merlevede 1st order bound}
\mathbb P\left(\left|\sum_{i=1}^{T/(2\epsilon)}(y_{1,i}-\mathbb E[y_{1,i}]\right|\geq c_7\frac T{2\epsilon}T^{-\delta}\right)\leq q_1 T\exp(-q_2 T^{(1-\delta)/2})+q_3\exp(-q_4T^{1-2\delta}).
\end{align}
For $\delta=\frac25$, using (\ref{sample mean eq}), we obtain the bound (\ref{1st order bernstein}).

\emph{Second-order functionals}. 
As for the first-order statistic $y_1$, take a small constant $\epsilon>0$ chosen such that $\frac T{2\epsilon}\in\mathbb N$, and consider the partial empirical averages 
\begin{equation}\label{y_2i def}
y_{2,i}:=\frac1{2\epsilon}\int_{[2\epsilon(i-1),2\epsilon i]\times\mathscr X}\int_{(0,T)\times\mathscr X}f_2(t-s,x,y)\,N(\mathrm dt,\mathrm dx)\,N(\mathrm ds,\mathrm dy),\quad 1\leq i\leq \frac T{2\epsilon}.
\end{equation}
Note that $y_2$ can be seen as the sample average over $(y_{2,i})_{1\leq i\leq \frac T{2\epsilon}}$, i.e., 
\begin{equation}\label{sample mean eq2}
y_2=\overline{y_{2,\cdot}}=\frac{2\epsilon}T\sum_{i=1}^{T/(2\epsilon)}y_{2,i}.
\end{equation}
In this part of the proof, let $\mathcal H_u^y:=\sigma(y_{1,v}:1\leq v\leq u)$, now measuring the temporal dependence by 
\begin{equation}\label{temporal dependence2}
\tau_y(l)=\sup_{u\in\{0,1,\ldots,T/(2\epsilon)-l\}}\tau\left(\mathcal H_u^y,y_{2,u+l}\right).
\end{equation}
As before, we show that $(y_{2,i})_{1\leq i\leq \frac T{2\epsilon}}$ is a weakly dependent in order to apply \cite{Merlevede2011}, Theorem 1. In order to do this, we need to check the two conditions given in Equations (2.6)-(2.8) in \cite{Merlevede2011}.

First, we prove Equation (2.7) in \cite{Merlevede2011}: we show, for constants $c_8,c_9>0$,
\begin{equation}
\sup_{1\leq i\leq T/(2\epsilon)}\mathbb P(|y_{2,i}|\geq x)\leq c_8\exp(-c_9x^{1/2}), \quad x>0.\label{2nd order cond1}
\end{equation}
We have assumed that $f_2(t,x,y)=0$ for $t\notin B$, for some bounded set $B$. By translating $f_2$, if necessary, we may assume that $B=[0,\kappa]$, for some $\kappa>0$. Then it holds that $$|y_{2,i}|\leq\frac{\|f_2\|_{L^\infty}}{2\epsilon}N([2\epsilon(i-1),2\epsilon i]\times\mathscr X)N([2\epsilon(i-1),2\epsilon i+\kappa]\times\mathscr X).$$ 
As in the proof for the first-order probabilities, $N([2\epsilon(i-1),2\epsilon i+\delta]\times\mathscr X)$, $\delta=0,\kappa$, can be shown to have exponential tails of order $1$. From this, it follows that $y_{2,i}$ has exponential tails of order $\frac12$, i.e., (\ref{2nd order cond1}) holds. Indeed, suppose that $Z_1,Z_2$ are random variables with $\mathbb P(|Z_i|\geq x)\leq c_{1,i}\exp(-c_{2,i}x)$, for $x>0$ and for positive constants $c_{i,j}$, $i,j=1,2$. Then 
$$
\mathbb P\left(|Z_1Z_2|\geq x\right)\leq \mathbb P\left(|Z_1|\geq x^{1/2}\right)+\mathbb P\left(|Z_1|\geq x^{1/2}\right)\leq (c_{1,1}+ c_{1,2})\exp\left(-(c_{2,1}\wedge c_{2,2})x^{1/2}\right).
$$

Second, we prove Equation (2.6) in \cite{Merlevede2011}. We show, for $\tau_y(l)$ defined in (\ref{temporal dependence2}), for constants $c_{10},c_{11}>0$,
\begin{equation}\label{2nd order cond2}
\tau_y(l)\leq c_{10}\exp(-c_{11} l), \quad l\in\mathbb N.
\end{equation}
As in the proof for the first-order functionals, take $\tilde N\stackrel d= N$ with $\tilde N$ independent of $\mathcal H_{2\epsilon z}$. We construct the counterpart $\tilde y_{2,i}$ of $y_{2,i}$ through $\tilde N$:
\begin{equation}\label{tilde y_2 def}
\tilde y_{2,i}=\frac1{2\epsilon}\int_{[2\epsilon(i-1),2\epsilon i]\times\mathscr X}\int_{(0,T)\times\mathscr X}f_2(t-s,x,y)\,\tilde N(\mathrm dt,\mathrm dx)\,\tilde N(\mathrm ds,\mathrm dy),\quad 1\leq i\leq \frac T{2\epsilon}.
\end{equation}
Since  $\tilde N\stackrel d= N$, it follows that $(\tilde y_{2,i})_{1\leq i\leq \frac T{2\epsilon}} \stackrel d= (y_{2,i})_{1\leq i\leq \frac T{2\epsilon}}$. Furthermore, since $\tilde N$ is independent of $\mathcal H_{2\epsilon z}$, it follows that $(\tilde y_{2,i})_{1\leq i\leq \frac T{2\epsilon}}$ is independent of $( y_{2,i})_{1\leq i\leq z}$. Furthermore, by Lemma \ref{expdecaylemma} it holds that 
\begin{align}
&\mathbb E\left|\tilde y_{2,z+l}-y_{2,z+l}\right|\nonumber
\\&=\frac1{2\epsilon}\mathbb E\left|\int_{[2\epsilon(l-1),2\epsilon l]\times\mathscr X}\int_{(0,T)\times\mathscr X}f_2(t-s,x,y)\left(\tilde N(\mathrm dt,\mathrm dx)\tilde N(\mathrm ds,\mathrm dy)- N(\mathrm dt,\mathrm dx) N(\mathrm ds,\mathrm dy)\right)\right|\nonumber\\
&\leq\frac1{2\epsilon}\int_{[2\epsilon(l-1),2\epsilon l]\times\mathscr X}\int_{[t-\kappa,t]\times\mathscr X}\left|f_2(t-s,x,y)\right|\left|\tilde N(\mathrm dt,\mathrm dx)\tilde N(\mathrm ds,\mathrm dy)- N(\mathrm dt,\mathrm dx) N(\mathrm ds,\mathrm dy)\right|\nonumber\\
&\leq\frac{\|f_2\|_{L^\infty}}{2\epsilon}\mu(\mathscr X)^2\kappa c_{12}\exp(-c_{13}(2\epsilon(l-1)-\kappa))
\end{align}
giving an exponential bound of the form (\ref{2nd order cond2}).

Using (\ref{2nd order cond1}) and (\ref{2nd order cond2}), we may apply \cite{Merlevede2011}, Theorem 1, with $\gamma_1=1$, $\gamma_2=\frac12$, hence $\gamma=\frac13$. This gives, for $\delta\in(0,1)$, a bound of the form
\begin{align}\label{Merlevede 1st order bound2}
\mathbb P\left(\left|\sum_{i=1}^{T/(2\epsilon)}(y_{2,i}-\mathbb E[y_{2,i}]\right|\geq c_{14}\frac T{2\epsilon}T^{-\delta}\right)\leq q_5 T\exp(-q_6 T^{(1-\delta)/3})+q_7\exp(-q_8T^{1-2\delta}).
\end{align}
For $\delta=\frac25$, using (\ref{sample mean eq2}), we obtain the bound (\ref{2nd order bernstein}).
\qed

\subsection*{Probabilistic lower bound for smallest restricted eigenvalue of $G$}
A lower bound on the minimum (restricted) eigenvalue of the corresponding Gram matrix is a standard requirement in high-dimensional sparse estimation, ensuring that the design remains sufficiently non-degenerate along sparse directions; see, e.g., \cite{Buhlmann2011}, Chapter 6, and \cite{Cai2024}, Lemma S11.

We now establish an \emph{unrestricted} eigenvalue condition on $G$, where $G$ is the random Gram matrix defined in (\ref{G random matrix}). 
In deriving a lower bound for this eigenvalue, we impose some structure on the orthonormal sequence $(\varphi_l)_{l\in\mathbb N}$ through Assumption \ref{min eigenvalue ass}.

\begin{lemma}\label{lemma min eigenvalue}
Grant Assumption \ref{stability ass}-\ref{ass PSD C0}. Suppose that $MH=o(T^{2/5})$. Consider the random Gram matrix $G$ defined in (\ref{G random matrix}). Then we can find $\Sigma_{\min}>0$ and constants $q_1,q_2>0$ such that when $T$ is sufficiently large, with probability at least $1-q_1T\exp(-q_2T^{1/5})$ it holds that 
$\sigma_{\min}(G)\geq\Sigma_{\min}$.
\end{lemma}
\begin{proof}
The proof consists of two steps. In the first step, we bound the minimum restricted eigenvalue $\Delta^\top \mathbb E[G]\Delta$ of the population matrix $\mathbb E[G]$; in the second step, we bound the difference between $\Delta^\top G\Delta$ and $\Delta^\top \mathbb E[G]\Delta$. 

In this proof, we employ the Rayleigh quotient characterization of the minimal eigenvalue,
$$\sigma_{\min}(G)=\min_{\Delta\in\mathbb R^{1+MH}\setminus\{0\}}\frac{\Delta^\top G\Delta}{\|\Delta\|_{\ell^2}^2}.$$ This implies that if we can bound $\Delta^\top G\Delta\geq\Sigma_{\min}\|\Delta\|_{\ell^2}^2$ for any $\Delta\in\mathbb R^{1+MH}$, then $\sigma_{\min}(G)\geq\Sigma_{\min}$.

\emph{Step 1}. 
 Write $\Delta=\begin{bmatrix}\Delta_0\\\Delta^{(MH)}\end{bmatrix}\in\mathbb R^{1+MH}$, where $\Delta_0\in\mathbb R$ denotes the first element of $\Delta$, and where the remaining $MH$ elements are captured by the vector $\Delta^{(MH)}\in\mathbb R^{MH}$. For $l\in[M],m\in[H]$, let $$\varpi_{l,m}(t):=\int_{(0,t)\times\mathscr X}\varphi_l(y)\chi_m(t-s)\bar\lambda^{(1)}(s,y)\, \mathrm ds\,\mathrm dy.$$ 
 Using the definitions \eqref{1st order mean intensity}--\eqref{covariance def}, note that the second moment measure admits the decomposition
\begin{align}\label{second moment decomposition}
\mathbb E\left[N(\mathrm ds,\mathrm dy)N(\mathrm ds',\mathrm dy')\right]
&=
\left(\bar\lambda^{(1)}(s,y)\bar\lambda^{(1)}(s',y')
+C^{(0)}(s,y,s',y')\right)\, \mathrm ds\, \mathrm dy \, \mathrm ds'\,\mathrm dy'\nonumber\\
&\quad+\mathbf{1}_{{(s,y)=(s',y')}}\bar\lambda^{(1)}(s,y)\, \mathrm ds\, \mathrm dy.
\end{align}
 By expanding $$
\frac1{T}\int_{(0,T)}\int_{(0,t)\times\mathscr X}\int_{(0,t)\times\mathscr X}\varphi_l(y)\chi_m(t-s)\varphi_{l'}(y')\chi_{m'}(t-s')\,\mathbb E\left[N(\mathrm ds,\mathrm dy)N(\mathrm ds',\mathrm dy')\right]\mathrm dt,
$$ 
we obtain 
\begin{equation}\label{decomposition}
\Delta^\top\mathbb E[G]\Delta
=
\sum_{i=1}^5 B_i,
\end{equation}
with $B_1
:= \Delta_0^2$ and
\begin{align}
B_2
&:= \frac1{T}\int_{(0,T)}
2\Delta_0
\sum_{\substack{l\in[M]\\m\in[H]}}
\Delta_{l,m}\varpi_{l,m}(t)\,\mathrm dt, \\
B_3
&:= \sum_{\substack{l,l'\in[M]\\m,m'\in[H]}}
\Delta_{l,m}
\frac1{T}\int_{(0,T)}
\varpi_{l,m}(t)\varpi_{l',m'}(t)\,\mathrm dt\, 
\Delta_{l',m'}, \\
B_4
&:= \sum_{\substack{l,l'\in[M]\\m,m'\in[H]}}
\Delta_{l,m}
\frac1{T}\int_{(0,T)}
\int_{(0,t)\times\mathscr X}
\int_{(0,t)\times\mathscr X}\nonumber\\
&\qquad\qquad\qquad\varphi_l(y)\,\chi_m(t-s)\,\varphi_{l'}(y')\,\chi_{m'}(t-s')\,C^{(0)}(s,y,s',y')
\,\mathrm ds\,\mathrm dy\,\mathrm ds'\,\mathrm dy'\,\mathrm dt\, 
\Delta_{l',m'}, \\
B_5
&:= \sum_{\substack{l,l'\in[M]\\m,m'\in[H]}}
\Delta_{l,m}
\frac1{T}\int_{(0,T)}
\int_{(0,t)\times\mathscr X}
\varphi_l(y)\varphi_{l'}(y)
\chi_m(t-s)\chi_{m'}(t-s)
\bar\lambda^{(1)}(s,y)
\,\mathrm ds\,\mathrm dy\,\mathrm dt\, 
\Delta_{l',m'}.
\end{align}
Here, Assumption \ref{ass PSD C0} implies that that $B_4\geq0$. Next, we prove that $B_5\geq\Gamma_{\min}\|\Delta^{(M)}\|_{\ell^2}^2$, where $\Gamma_{\min}:=\mathfrak h_2^2\underline\nu/2$, with $\mathfrak h_2$ as in Assumption \ref{min eigenvalue ass}(iii).
This follows from 
\begin{align}
B_5&=\frac1{T}\int_{(0,T)}\int_{(0,t)\times\mathscr X}\xi_\Delta(t-s,y)^2\bar\lambda^{(1)}(s,y)\, \mathrm ds\, \mathrm dy\, \mathrm dt\nonumber\\
&\geq \frac1{T}\int_{(0,T)}\int_{(0,t)\times\mathscr X}\xi_\Delta(t-s,y)^2\underline\nu\, \mathrm ds\, \mathrm dy\, \mathrm dt\nonumber\\
&\geq\frac12\underline\nu\sum_{\substack{l,l'\in[M]\\m,m'\in[H]}}\Delta_{l,m}\Delta_{l',m'}\varphi_l(y)\varphi_{l'}(y)\, \mathrm dy\int_{(0,b)}\chi_m(t)\chi_{m'}(t)\, \mathrm dt\nonumber\\
&\geq\frac12\mathfrak h_2^2\,\underline\nu \|\Delta^{(MH)}\|_{\ell^2}^2=\Gamma_{\min}\|\Delta^{(MH)}\|_{\ell^2}^2,
\end{align}
where $\xi_\Delta(t,y):=\sum_{l\in[M],m\in[H]}\Delta_{l,m}\varphi_l(y)\chi_m(t)$. The second inequality holds for $T\geq 2b$. In the final inequality, we use that $\chi_m$ is a.e.\ positive for all $m\in\mathbb N$.

Next, for $\varsigma>0$ we can bound $B_2\leq\varsigma B_1+\varsigma^{-1}B_3$, since for any $x,y\in\mathbb R$, $|2xy|\leq\varsigma x^2+\varsigma^{-1}y^2$. This implies that 
\begin{equation}\label{before B3 bound}
\Delta^\top\mathbb E[G]\Delta\geq(1-\varsigma)\Delta_0^2+(1-\varsigma^{-1})B_3+\Gamma_{\min}\|\Delta^{(MH)}\|_{\ell^2}^2.
\end{equation}
To bound (\ref{before B3 bound}), we find an upper bound on $|B_3|$. 
Write $B_3$ as
\begin{equation}\label{B3 1}B_3=\frac1{T}\int_{(0,T)}\left(\int_{((t-b)\vee0,t)\times\mathscr X}\xi_{\Delta}(t-s,y)\bar\lambda^{(1)}(s,y)\, \mathrm ds\, \mathrm dy\right)^2\mathrm dt.\end{equation}
Let $I_t:=\int_{((t-b)\vee0,t)\times\mathscr X}\bar\lambda^{(1)}(s,y)\,\mathrm dy<\infty$, so that $\bar\lambda^{(1)}(s,y)/I_t$ becomes a probability measure on $(0,t)\times\mathscr X$. Note that by coupling the process starting in the empty history to the process $N_{\mathrm{stat}}$ starting in stationarity $$\bar\lambda^{(1)}(s,y)=\mathbb E\left|\mathrm dN(s,y)\right|/(\mathrm ds\, \mathrm dy)\leq\mathbb E\left|\mathrm dN_{\mathrm{stat}}(s,y)\right|/(\mathrm ds\, \mathrm dy)=\bar\lambda(y).$$ Therefore, we can bound $\sup_{t\in[0,T]}I_t\leq b\|\bar\lambda\|_{L^\infty(\mathscr X)}\mu(\mathscr X)$, where $\|\bar\lambda\|_{L^\infty(\mathscr X)}<\infty$ by Lemma \ref{step3Alemma} from Appendix~\hyperref[app A]{A}, while $\mu(\mathscr X)<\infty$ by assumption. 
We apply Jensen's inequality to the integrand of (\ref{B3 1}) to obtain 
\begin{align}
&\left(\int_{(0,t)\times\mathscr X}\xi_\Delta(t-s,y)\bar\lambda^{(1)}(s,y)\, \mathrm ds\, \mathrm dy\right)^2
\leq\int_{(0,t)\times\mathscr X}(\xi_\Delta(t-s,y))^2\bar\lambda^{(1)}(s,y)/I_t\, \mathrm ds\, \mathrm dy\, I_t^2\nonumber\\
&\leq\int_{(0,t)\times\mathscr X}(\xi_\Delta(t-s,y))^2\, \mathrm ds\, \mathrm dy\, \|\bar\lambda\|_{L^\infty(\mathscr X)}I_t
\leq \|\Delta^{(HM)}\|_{\ell^2}^2\mathfrak h_2^2\|\bar\lambda\|_{L^\infty(\mathscr X)}^2\mu(\mathscr X),\label{B3 2}
\end{align}
where we use Cauchy-Schwarz, orthonormality of $(\varphi_l)_{l\in[M]}$, and orthogonality of $(\chi_m)_{m\in\mathbb N}$.
Now define $C:=\mathfrak h_2^2\|\bar\lambda\|_{L^\infty(\mathscr X)}^2\mu(\mathscr X)$. This allows us to write $B_3\leq C\|\Delta^{(MH)}\|_{\ell^2}^2$. Hence, for $\varsigma\in(0,1)$, so for $(1-\varsigma^{-1})<0$, this gives 
\begin{equation}\label{after B3 bound}
\Delta^\top\mathbb E[G]\Delta\geq(1-\varsigma)\Delta_0^2+(1-\varsigma^{-1})C\|\Delta^{(MH)}\|_{\ell^2}^2+\Gamma_{\min}\|\Delta^{(MH)}\|_{\ell^2}^2.
\end{equation}
In particular, for $\varsigma:=\frac{2C}{2C+\Gamma_{\min}}$, the bound (\ref{after B3 bound}) reads 
\begin{equation}\label{after B3 final bound}
\Delta^\top\mathbb E[G]\Delta\geq\frac{\Gamma_{\min}}{2C+\Gamma_{\min}}\Delta_0^2-\frac{\Gamma_{\min}}{2C}C\|\Delta^{(M)}\|_{\ell^2}^2+\Gamma_{\min}\|\Delta^{(M)}\|_{\ell^2}^2\geq2\Sigma_{\min}\|\Delta\|_{\ell^2}^2,
\end{equation}
where $\Sigma_{\min}:=\frac12\min\{\frac{\Gamma_{\min}}{2C+\Gamma_{\min}},\frac{\Gamma_{\min}}2\}>0$.

\emph{Step 2}. We proceed by identifying an upper bound on $\Delta^\top(G-\mathbb E[G])\Delta$. Since $G_{0,0}$ is deterministic, we expand 
\begin{equation}\label{diff G E[G]}
\Delta^\top(G-\mathbb E[G])\Delta=(\Delta^{(MH)})^\top(G^{(MH)}-\mathbb E[G^{(MH)}])\Delta^{(MH)}+2\Delta_0\sum_{\substack{l\in[M]\\m\in[H]}}(G_{0,(l,m)}-\mathbb E[G_{0,(l,m)}])\Delta_{l,m}.
\end{equation}
We first bound the contribution of the first term of (\ref{diff G E[G]}). Let $\otimes$ denote the tensor product. Next, let $\varphi^{(M)}(y)=(\varphi_1(y),\ldots,\varphi_M(y))^\top$ and $\chi^{(H)}(t)=(\chi_1(t),\ldots,\chi_H(t))^\top$, and define
\begin{align}
f_1(s,y)&:=(\Delta^{(MH)})^\top\left(\left(\int_s^T\chi^{(H)}(t-s)(\chi^{(H)}(t-s))^\top\,\mathrm dt\right)\otimes\left(\varphi^{(M)}(y)\left(\varphi^{(M)}(y)\right)^\top\right)\right) \Delta^{(MH)};  \\
f_2(s,s',y,y')&:=(\Delta^{(MH)})^\top\bigg(\left(\int_{\max\{0,s'-s\}}^{b+\min\{0,s'-s\}}h(t)^{(H)}(\chi^{(H)}(t+s-s'))^\top\,\mathrm dt\right)\nonumber\\
&\otimes\left(\varphi^{(M)}(y)\left(\varphi^{(M)}(y)\right)^\top \right)\bigg) \,\Delta^{(MH)}.
\end{align}
By a change of variables and using  Theorem \ref{concentration ineq},
\begin{align}
&(\Delta^{(MH)})^\top(G^{(MH)}-\mathbb E[G^{(MH)}])\Delta^{(MH)}\nonumber\\
&=\frac1T\int_{(0,T)\times\mathscr X}f_1(s,y)\left\{N(\mathrm ds,\mathrm dy)-\mathbb E[N(\mathrm ds,\mathrm dy)]\right\}\nonumber\\
&+\frac1T\int_{(0,T)\times\mathscr X}\int_{(0,T)\times\mathscr X}f_2(s,s',y,y')\mathbf1\{s\neq s'\}\left\{N(\mathrm ds,\mathrm dy)N(\mathrm ds',\mathrm dy')-\mathbb E[N(\mathrm ds,\mathrm dy)N(\mathrm ds',\mathrm dy')]\right\}\nonumber\\
&\leq \|f_1\|_{L^\infty}c_{1}T^{-2/5}+\|f_2\|_{L^\infty}c_{2}T^{-2/5},\label{T -2/5 bound QF}
\end{align}
with probability at least $1- q_1T\exp(-q_2T^{1/5})-q_3T\exp(-q_4T^{1/5})$, where the constants $c_i,q_j>0$ come from Theorem \ref{concentration ineq}. By further bounding $\|f_i\|_{L^\infty}\leq \mathfrak h_2^2\|\Delta^{(MH)}\|_{\ell^2}^2\max_{i\in[M]}\|\varphi_i\|_{L^\infty(\mathscr X)}^2$ we conclude that we can find a constant $c_3$ such that with probability at least $1- q_1T\exp(-q_2T^{1/5})-q_3T\exp(-q_4T^{1/5})$
\begin{equation} (\Delta^{(MH)})^\top (G^{(MH)}-\mathbb E[G^{(MH)}])\Delta^{(MH)}\leq c_3MT^{-2/5}\|\Delta^{(MH)}\|_{\ell^2}^2=o(\|\Delta\|_{\ell^2}^2),\label{bound T1/3}\end{equation} 
using $M=o(T^{2/5})$.  Here we employ Assumption \ref{min eigenvalue ass}, which guarantees  that we can find a constant $c_4$ such that $\max_{l\in[M]}\|\varphi_l\|_{L^\infty(\mathscr X)}^2\leq c_4 M$.

We now bound the contribution of the second term of (\ref{diff G E[G]}). Let \begin{equation}f_3(s,y):=2\Delta_0\sum_{\substack{l\in[M]\\m\in[H]}}\varphi_l(y)\int_{(s,T)}\chi_m(t-s)\,\mathrm dt\,\Delta_{l,m}.\end{equation} Again, by a change of variables and using  Theorem \ref{concentration ineq}, with probability at least
$1- q_5T\exp(-q_6T^{1/5})$,
\begin{align}
2\Delta_0\sum_{\substack{l\in[M]\\m\in[H]}}(G_{0,(l,m)}-\mathbb E[G_{0,(l,m)}])\Delta_{l,m}&=\frac1T\int_{(0,T)\times\mathscr X}f_3(s,y)\left\{N(\mathrm ds,\mathrm dy)-\mathbb E[N(\mathrm ds,\mathrm dy)]\right\}\nonumber\\
&\leq \|f_3\|_{L^\infty}c_5T^{-2/5},
\end{align}
where, using $2xy\leq x^2+y^2$ and setting $c_6:=1+c_4\mathfrak h_1^2$,
\begin{align}
\|f_3\|_{L^\infty}&\leq\sum_{\substack{l\in[M]\\m\in[H]}}\left(\Delta_0^2+\Delta_{l,m}^2\varphi_l(y)^2\left(\int_{(0,T)}\chi_m(t-s)\, \mathrm ds\right)^2\right)\nonumber\\
&\leq MH\Delta_0^2+\max_{l\in[M]}\|\varphi_l\|_{L^\infty(\mathscr X)}^2\mathfrak h_1^2\|\Delta^{(MH)}\|_{\ell^2}^2\leq c_6MH\|\Delta\|_{\ell^2}^2.
\end{align}
Using $MH=o(T^{2/5})$, we can find an event of probability at least $1- q_5T\exp(-q_6T^{1/5})$ such that \begin{equation}2\Delta_0\sum_{\substack{l\in[M]\\m\in[H]}}(G_{0,(l,m)}-\mathbb E[G_{0,(l,m)}])\Delta_{l,m}=o(\|\Delta\|_{\ell^2}^2).\label{bound T1/3 2}\end{equation}

Combining (\ref{diff G E[G]}), (\ref{bound T1/3}) and (\ref{bound T1/3 2}), we can find an event of probability at least 
$$
1- q_1T\exp(-q_2T^{1/5})-q_3T\exp(-q_4T^{1/5})-q_5T\exp(-q_6T^{1/5})
$$
such 
 that $\Delta^\top(G-\mathbb E[G])\Delta\leq\Sigma_{\min}\|\Delta\|_{\ell^2}^2$. For such $T$, it holds that 
\begin{equation}
\Delta^\top G\Delta \geq \Delta^\top \mathbb E[G]\Delta  - \Delta^\top(G-\mathbb E[G])\Delta\geq \Sigma_{\min}\|\Delta\|_{\ell^2}^2.
\end{equation}
Now the result follows from the Rayleigh quotient characterization of the minimal eigenvalue.
\end{proof}

For later reference, we summarize the result of Equations  (\ref{T -2/5 bound QF})--(\ref{bound T1/3}) in a lemma.

\begin{lemma}\label{119120 lemma}
Grant Assumption \ref{stability ass}-\ref{ass PSD C0}. Suppose that $MH=o(T^{2/5})$. Consider the random Gram matrix $G$ defined in (\ref{G random matrix}). For any $\Delta\in\mathbb R^{1+MH}$, write $\Delta=[\Delta_0\ \Delta^{(MH)}]^\top$, where $\Delta_0\in\mathbb R,\Delta^{(MH)}\in\mathbb R^{MH}$. Then it holds that  \begin{equation} (\Delta^{(MH)})^\top (G^{(MH)}-\mathbb E[G^{(MH)}])\Delta^{(MH)}\leq c_3MT^{-2/5}\|\Delta^{(MH)}\|_{\ell^2}^2=o(\|\Delta\|_{\ell^2}^2).\label{bound T1/3 B}\end{equation} 
\end{lemma}

\section*{Appendix B: Relegated proofs of Sections~\ref{setup section}--\ref{sparsistency section}}\label{app B}

{\it Proof of Lemma~\ref{integrated diff and SDP}}.
Set $\pi:=\beta\otimes\gamma$, so $\Delta\pi:=\Delta (\beta\otimes\gamma)=\tilde\beta\otimes\tilde\gamma-\hat\beta\otimes\hat\gamma$.  Then $\Delta\theta^\otimes_k=(\Delta\alpha_k,\Delta\pi_{k,\cdot,\cdot})$, where $\pi_{k,l,m}:=\beta_{k,l}\gamma_m$.
We have 
\begin{align}
\lambda_{\tilde \theta}(t,x)-\lambda_{\hat \theta}(t,x)&=\sum_{k\in[M]}\Delta\alpha_k\varphi_k(x)+\int_{(0,t)\times\mathscr X}\sum_{\substack{k,l\in[M]\\m\in[H]}}\Delta \pi_{k,l,m}\varphi_k(x)\varphi_l(y)\chi_m(t-s)\,N(\mathrm ds,\mathrm dy)\nonumber\\
&=\sum_{k\in[M]}\varphi_k(x)\left(\Delta\alpha_k+\int_{(0,t)\times\mathscr X}\sum_{\substack{l\in[M]\\m\in[H]}}\Delta \pi_{k,l,m}\varphi_l(y)\chi_m(t-s)\,N(\mathrm ds,\mathrm dy)\right).
\end{align}
Since the sequence $(\varphi_k)_{k\in\mathbb N}$ is orthonormal, we have
\begin{align}
&\frac1{T}\int_{(0,T)\times\mathscr X}(\lambda_{\tilde \theta}(t,x)-\lambda_{\hat \theta}(t,x))^2\, \mathrm dt\,\mathrm dx\nonumber\\
&=\sum_{k\in[M]}\frac1T\int_{(0,T)}\left(\Delta\alpha_k+\int_{(0,t)\times\mathscr X}\sum_{\substack{l\in[M]\\m\in[H]}}\Delta \pi_{k,l,m}\varphi_l(y)\chi_m(t-s)\,N(\mathrm ds,\mathrm dy)\right)^2\mathrm dt=\sum_{k\in[M]}\left(\Delta \theta_k^\otimes\right)^\top G\Delta\theta^\otimes_k.\label{semidefinite product}
\end{align}
 By the Rayleigh quotient characterization of the minimal eigenvalue,
\begin{equation}
\sum_{k\in[M]}\left(\Delta \theta_k^\otimes\right)^\top G\Delta\theta^\otimes_k\geq\sum_{k\in[M]}\sigma_{\min}(G)\|\Delta\theta_k^\otimes\|_2^2\geq\sigma_{\min}(G)\|\Delta\theta^\otimes\|_2^2\geq\sigma_{\min}(G)\frac{\|\Delta\theta^\otimes\|_{\ell^1}^2}{\mathscr S},\label{Frob norm bound}
\end{equation}
which finishes the proof.\qed\newline

{\it Proof of Theorem~\ref{consistency thm}}.
The proof proceeds by comparing the estimator $\hat\theta$ with the approximating parameter $\tilde\theta$ from Assumption \ref{ass approx} through their penalized least-squares criteria. In Step 1, we decompose the resulting excess risk into an approximation term, a martingale term, and a penalty term. In Step 2, the martingale term is controlled via a truncation argument: on a high-probability event, local event counts are logarithmically bounded, which yields uniform bounds on the jumps and predictable quadratic variation of the associated martingale; an exponential martingale inequality then gives the required deviation bound. In Step 3, the penalty term is controlled through the empirical Gram matrix and its minimum eigenvalue. Combining these bounds reduces the analysis to a deterministic comparison inequality, which we resolve in Step 4 by considering the three possible relative magnitudes of the estimation, approximation, and estimator--approximant errors. In each case, this yields the stated rate.

\emph{Step 1: Comparing least-squares estimator to ``good'' approximating parameter}. 
Since $\hat \theta$ is a minimizer of the optimization problem in (\ref{lasso LS contrast}), and since $\tilde \theta$ is eligible as well, it holds that 
\begin{align}
&\frac1T\left(-\int_{(0,T)\times\mathscr X}\lambda_{\hat \theta}(t,x)\,N(\mathrm dt,\mathrm dx)+\frac12\int_{(0,T)\times\mathscr X}\lambda_{\hat \theta}(t,x)^2\, \mathrm dt\, \mathrm dx\right)+\eta\|\hat \beta\otimes\hat\gamma\|_{\ell^1}\nonumber\\
&\leq\frac1T\left(-\int_{(0,T)\times\mathscr X}\lambda_{\tilde \theta}(t,x)\,N(\mathrm dt,\mathrm dx)+\frac12\int_{(0,T)\times\mathscr X}\lambda_{\tilde \theta}(t,x)^2\, \mathrm dt\, \mathrm dx\right)+\eta\|\tilde \beta\otimes\tilde\gamma\|_{\ell^1}.\label{optimal solution ineq}
\end{align}
Adding 
$$
\frac1{2T}\int_{(0,T)\times\mathscr X}\lambda(t,x)^2\, \mathrm dt\, \mathrm dx-\frac1T\int_{(0,T)\times\mathscr X}\lambda(t,x)\lambda_{\hat \theta}(t,x)\, \mathrm dt\, \mathrm dx
$$
to both sides of (\ref{optimal solution ineq}) and rewriting gives us, for any minimizer $\hat\theta$ of \eqref{lasso LS contrast},
\begin{align}
\frac1{2T}\int_{(0,T)\times\mathscr X}(\lambda_{\hat \theta}(t,x)-\lambda(t,x))^2\, \mathrm dt\, \mathrm dx
&\leq I_1+I_2+I_3,\label{I123eq}\\
\end{align}
where
\begin{align}
I_1
&:=\frac1{2T}\int_{(0,T)\times\mathscr X}(\lambda_{\tilde \theta}(t,x)-\lambda(t,x))^2\, \mathrm dt\, \mathrm dx,\\
I_2
&:=\frac1T\int_{(0,T)\times\mathscr X}(\lambda_{\tilde \theta}(t,x)-\lambda_{\hat \theta}(t,x))
\bigl(\lambda(t,x)\, \mathrm dt\, \mathrm dx-N(\mathrm dt,\mathrm dx)\bigr),\\
I_3
&:=\eta\bigl(\|\tilde \beta\otimes\tilde\gamma\|_{\ell^1}-\|\hat \beta\otimes\hat\gamma\|_{\ell^1}\bigr).
\end{align}

In this proof, we can immediately bound $I_1$, using Assumption  \ref{ass approx}. In Steps 2 and 3, we bound both $I_2$ and $I_3$ in terms of $$\frac1{T}\int_{(0,T)\times\mathscr X}(\lambda_{\tilde \theta}(t,x)-\lambda_{\hat \theta}(t,x))^2\,\mathrm dt\,\mathrm dx,$$ Next, in Step 4, we rewrite the resulting bound in three (partially overlapping) cases. Case (i) specified in Step 4 allows for a clever way of rewriting \eqref{I123eq}. Case (ii) makes the bounding straightforward, by Assumption  \ref{ass approx}. In Case (iii) we derive an inequality that can be treated similarly as (\ref{I123eq}) in Case (i). 

\emph{Step 2: Probabilistic bound for $I_2$}. 
Consider the unbounded martingale
\begin{equation}\label{martingale Y_T}
Y_T := \int_{(0,T)\times\mathscr X} (\lambda_{\hat \theta}(t,x) - \lambda_{\tilde \theta}(t,x))(N(\mathrm dt,\mathrm dx) - \lambda(t,x)\, \mathrm dt\, \mathrm dx).
\end{equation}
Its predictable quadratic variation satisfies 
\begin{equation}\label{martingale Y_T QV}
\langle Y,Y \rangle_T= \int_{(0,T)\times\mathscr X} (\lambda_{\hat \theta}(t,x) - \lambda_{\tilde \theta}(t,x))^2 \lambda(t,x)\,  \mathrm dt\, \mathrm dx. 
\end{equation}
We use the martingale structure of $Y_T$ to bound it using an exponential inequality for \emph{bounded} martingales. Since $Y_T$ is unbounded, we apply a truncation argument.


\emph{Step 2A: Truncation}. 
For $s<u$ and $B>0$ consider the event
\begin{equation}\label{A event eq}
A_{s,u;B}:=\left\{\omega:\sup_{t\in[s,u]}N([t-b,t]\times\mathscr X)\leq B \right\},
\end{equation}
where $b$ is the constant from Assumption \ref{stability ass}(vi).

We will apply  \cite{vandeGeer1995}, Lemma 2.1, which says that for a martingale $(X_t)_{t\geq0}$ having predictable quadratic variation $\langle X,X\rangle_t$ that if $|\Delta X_t|\leq K$ for all $t>0$, for some $K\in[0,\infty)$, then for each $a,b>0$ it holds that
\begin{equation}\label{lemma2.1 ineq}
\mathbb P(X_T\geq a,\langle X,X\rangle_T\leq b^2)\leq\exp\left(-\frac{a^2}{2(aK+b^2)}\right).
\end{equation}
We construct a martingale from $Y_t$ that satisfies the conditions of \cite{vandeGeer1995}, Lemma 2.1. To this end, consider the stopping time
\begin{equation}\label{stopping time}
\tau=\inf\left\{t\geq 0\ \Big|\ \sup_{u \in [0,t]} N\left([u-b,u]\times\mathscr X\right) > B\right\} \wedge T.
\end{equation}
Then, consider the stopped martingale $X_t=Y_{t\wedge\tau}$. 
Note that  $\omega \in A_{0,T;B}$ implies that $X_t=Y_t$ for all $t\in[0,T]$. We bound $I_2$ in probability by conditioning on $A_{0,T;B}$: for $L>0$,
\begin{align}
\mathbb P(I_2 > L)&=\mathbb P(Y_T> TL)\leq\mathbb P(Y_T> TL\ |\  A_{0,T;B})+\mathbb P( A_{0,T;B}^{\mathsf{c}})=\mathbb P(X_T> TL\ |\  A_{0,T;B})+\mathbb P( A_{0,T;B}^{\mathsf{c}}).\label{I_2 prob bound}
\end{align}

We first bound $\mathbb P( A_{0,T;B}^{\mathsf{c}})$
Note that by Assumption \ref{stability ass}(ii), (iii), (iv), $t\mapsto\int_{\mathscr X}\lambda(t,x)\mathrm dx$ can be bounded by
\begin{equation}\label{intensity uni Hawkes}
\lambda_{1D}(t)=\overline\nu+\sum_{t_i<t}\frac{\varrho_{T_{\mathrm{hom}}}}{\mu(\mathscr X)\|h\|_{L^1(\mathbb R_+)}}h(t-t_i),
\end{equation}
which corresponds to the intensity of a spatiotemporal Hawkes process $N_{1D}$ with i.i.d.\ uniform spatial coordinates. By a straightforward coupling using the standard Picard iteration construction of the graphon Hawkes process (cf. the proof of Lemma \ref{couplingdeviationslemma}), $$N((s,t]\times\mathscr X)\leq N_{1D}((s,t]\times\mathscr X).$$ Since $N_{1D}$ has i.i.d. uniform spatial coordinates, it can be treated as a univariate Hawkes process $\tilde N_{1D}$ with intensity 
\begin{equation}\label{intensity uni Hawkes1d}
\tilde\lambda_{1D}(t)=\overline\nu\mu(\mathscr X)+\sum_{t_i<t}\frac{\varrho_{T_{\mathrm{hom}}}}{\|h\|_{L^1(\mathbb R_+)}}h(t-t_i),
\end{equation} where to each particle, we attach a uniform coordinate in $\mathscr X$. 
By the existence of exponential moments for non-Markovian Hawkes processes \cite{Leblanc2025b} and by the Chernoff bound, we can find $q_1,q_2>0$ such that
\begin{align}
&\mathbb P\left(\sup_{t\in[0,T]}N((t-b,t]\times \mathscr X)>B\right)\leq \sum_{i=0}^{\lceil T/b\rceil}\mathbb P\left(N((i-1)b,ib]\times\mathscr X)\geq B/2\right)\nonumber\\
&\leq\sum_{i=0}^{\lceil T/b\rceil}\mathbb P\left(\tilde N_{1D}((i-1)b,ib])\geq B/2\right)
\label{LDP uni Hawkes}\leq\frac{2T}{b}q_1\exp(-q_2 B). 
\end{align}
From now on, we set $B:=\frac2{q_2}\log(T)$. It follows that 
\begin{equation}\label{P3 bound}
\mathbb P(A_{0,T;B}^{\mathsf{c}})\leq \frac{2q_1}bT\exp(-2\log(T))=\frac{2q_1}{bT}.
\end{equation}

\emph{Step 2B: Jumps of the martingale}. 
We move on to the first term in the right-hand side of (\ref{I_2 prob bound}).
Note that $|\Delta X_t|\leq |\Delta Y_t|$ for all $t\geq0$. 
We bound $$|\Delta Y_t|\leq\sup_{t\in[0,T]}\sup_{x\in\mathscr X}|\lambda_{\hat \theta}(t,x) - \lambda_{\tilde \theta}(t,x)|.$$ 
Denote the difference in coefficients by $\Delta\alpha=\tilde\alpha-\hat\alpha$ and $\Delta (\beta\otimes\gamma)=\tilde\beta\otimes\tilde\gamma-\hat\beta\otimes\hat\gamma$.
Using Assumption \ref{stability ass}(vi), we know that $\sup_{x\in\mathscr X}\left|\lambda_{\tilde \theta}(t,x)-\lambda_{\hat \theta}(t,x)\right|$ is only influenced by events for $N$ that occur in the time interval $[t-b,t]$. Furthermore, using Assumption  \ref{min eigenvalue ass}, for each $t$, the contribution to $\sup_{x\in\mathscr X}\left|\lambda_{\tilde \theta}(t,x)-\lambda_{\hat \theta}(t,x)\right|$ of each of those events can be bounded by $c_1\mathfrak h_\infty\|\Delta (\beta\otimes\gamma)\|_{\ell^1} M$.
On $A_{0,T;B}$, we know that the number of events in $[t-b,t]$ is bounded by $B$. Therefore, on $A_{0,T;B}$,  
\begin{equation}|\Delta X_t|\leq\sqrt {c_1}\|\Delta\alpha\|_{\ell^1}\sqrt M+ c_1\mathfrak h_\infty\|\Delta (\beta\otimes\gamma)\|_{\ell^1}MB\leq c_2(\|\Delta\alpha\|_{\ell^1}+\|\Delta (\beta\otimes\gamma)\|_{\ell^1})MB,
\end{equation}
where $c_2:=\max\{\sqrt{c_1},c_1\mathfrak h_\infty\}$. 
By Lemmas \ref{integrated diff and SDP} and  \ref{lemma min eigenvalue} from Appendix~\hyperref[app A]{A},  it holds with probability at least $1-q_3T\exp(-q_4T^{1/5})$ that
\begin{align}
|\Delta X_t|&\leq c_2MB\left(\frac{\mathscr S}{\sigma_{\min}(G)}\frac1{T}\int_{(0,T)\times\mathscr X}(\lambda_{\tilde \theta}(t,x)-\lambda_{\hat \theta}(t,x))^2\, \mathrm dt\, \mathrm dx\right)^{1/2}\nonumber\\
&\leq c_3M\sqrt{\mathscr S}\,\frac{\log(T)}{\sqrt T}\left(\int_{(0,T)\times\mathscr X} (\lambda_{\hat \theta}(t,x) - \lambda_{\tilde \theta}(t,x))^2 \, \mathrm dt\, \mathrm dx\right)^{1/2},
\end{align}
 where $\mathscr S\leq2 M^2H$ is the number of nonzero elements of $\Delta \alpha$ and $\Delta (\beta\otimes\gamma)$, where $c_3:=2\sqrt 2c_2/(q_2\sqrt{\Sigma_{\min}})$, and where $\Sigma_{\min}$ is specified in  Lemma \ref{lemma min eigenvalue} from Appendix~\hyperref[app A]{A}.
 
\emph{Step 2C: Quadratic variation of the martingale}. 
Next, we bound $\langle X,X\rangle_T$ on $A_{0,T;B}$. 
To this end, we bound $\sup_{t\in[0,T]}\sup_{x\in\mathscr X}\lambda(t,x)$ for $\omega\in A_{0,T;B}$. By Assumption \ref{stability ass}(vi), we know that $\lambda(t,x)$ is only influenced by events in $[t-b,t]$. On $A_{0,T;B}$, we know that the number of events in $[t-b,t]$ is bounded by $B$. Letting $c_4:=\overline\nu+C_W\|h\|_{L^\infty(\mathbb R_+)}<\infty$ by Assumption \ref{stability ass}, we infer that $\sup_{t\in[0,T]}\sup_{x\in\mathscr X}\lambda(t,x)\leq c_4 B$. Therefore, on $A_{0,T;B}$
\begin{equation}\label{quadratic var bound}
\langle X,X\rangle_T\leq c_5 \log(T)\int_{(0,T)\times\mathscr X} (\lambda_{\hat \theta}(t,x) - \lambda_{\tilde \theta}(t,x))^2\, \mathrm dt\, \mathrm dx,
\end{equation}
where $ c_5=2c_4/q_2$.

\emph{Step 2D: Exponential inequality for the martingale}. 
In (\ref{I_2 prob bound}), set $$L:=2\sqrt{c_5}\frac{\log(T)}{T}\left(\int_{(0,T)\times\mathscr X} (\lambda_{\hat \theta}(t,x) - \lambda_{\tilde \theta}(t,x))^2 \, \mathrm dt\, \mathrm dx\right)^{1/2}.$$ Then by (\ref{lemma2.1 ineq}), we obtain
\begin{align}
\mathbb P(X_T> TL\ |\  A_{0,T;B})\leq\exp\left(-\frac{4c_5\log(T)^2}{2(2 c_3\sqrt{c_5}\sqrt{\mathscr S}M\frac{\log(T)}{\sqrt T}\log(T)+c_5 \log(T))}\right).
\end{align}
Using $\sqrt{\mathscr S} M\log(T)/\sqrt T=o(1)$, which is guaranteed since $\mathscr S\leq 2M^2H$ (being the number of nonzero entries of $\Delta\alpha$ and $\Delta (\beta\otimes\gamma)$; see (\ref{semidefinite product})--(\ref{Frob norm bound}) below), select $T$ sufficiently large such that $2 c_3\sqrt{\mathscr S}M\log(T)/\sqrt T\leq \sqrt{c_5}$. For such $T$, it holds that 
\begin{equation}\label{P1 bound}
\mathbb P(X_T> TL\ |\  A_{0,T;B})\leq\frac1T.
\end{equation}
Combining (\ref{I_2 prob bound}), (\ref{P3 bound}) and (\ref{P1 bound}), with $q_5:=\frac{2q_1}b+1$ and $c_6:=2\sqrt{c_5}$ it follows that with probability at least
$1-q_3T\exp(-q_4T^{1/5})-\frac{q_5}T$,
\begin{equation}\label{I2 final bound}
I_2\leq c_6\frac{\log(T)}{T}\left(\int_{(0,T)\times\mathscr X} (\lambda_{\hat \theta}(t,x) - \lambda_{\tilde \theta}(t,x))^2\,  \mathrm dt\, \mathrm dx\right)^{1/2}.
\end{equation}

\emph{Step 3: Probabilistic bound for $I_3$}. 
Next, we bound $I_3$ by bounding $\|\Delta (\beta\otimes\gamma)\|_{\ell^1}$ in terms of 
\begin{equation}\label{bound i3}
\frac1{T}\int_{(0,T)\times\mathscr X}(\lambda_{\tilde \theta}(t,x)-\lambda_{\hat \theta}(t,x))^2\mathrm dt\mathrm dx,
\end{equation}
We rewrite (\ref{bound i3}) in terms of a certain quadratic form, for which we need a bounded eigenvalue condition. 
By Lemma \ref{integrated diff and SDP} and using the reverse triangle inequality, we can bound $I_3$ by 
\begin{equation}\eta\left(\frac{\mathscr S}{\sigma_{\min}(G)}\frac1{T}\int_{(0,T)\times\mathscr X}(\lambda_{\tilde \theta}(t,x)-\lambda_{\hat \theta}(t,x))^2\, \mathrm dt\, \mathrm dx\right)^{1/2}.\label{eq I3 bound}
\end{equation}
To make use of this bound, we use Lemma \ref{lemma min eigenvalue} from Appendix~\hyperref[app A]{A}, giving a  probabilistic lower bound on $\sigma_{\min}(G)$. 
Note that in the bounding of $I_2$, we already applied Lemma \ref{lemma min eigenvalue} to condition on the event $\{\omega:\sigma_{\min}(G)\geq\Sigma_{\min}\}$, which has probability at least $1-q_3T\exp(-q_4T^{1/5})$.

\emph{Step 4: Combining the bounds}. 
By (\ref{I2 final bound}) and (\ref{eq I3 bound}), on an event of probability at least $1-q_3T\exp(-q_4T^{1/5})-\frac{q_5}T$, 
(\ref{I123eq}) reduces to 
\begin{align}
&\frac1{2T}\int_{(0,T)\times\mathscr X}(\lambda_{\hat \theta}(t,x)-\lambda(t,x))^2\, \mathrm dt\, \mathrm dx\nonumber\\
&\leq\frac1{2T}\int_{(0,T)\times\mathscr X}(\lambda_{\tilde \theta}(t,x)-\lambda(t,x))^2\, \mathrm dt\, \mathrm dx\nonumber\\
&+\left(c_6\frac{\log(T)}{\sqrt T}+\frac1{\sqrt{\Sigma_{\min}}}\eta\sqrt{\mathscr S}\right)\left(\frac1{T}\int_{(0,T)\times\mathscr X}(\lambda_{\tilde \theta}(t,x)-\lambda_{\hat \theta}(t,x))^2\, \mathrm dt\, \mathrm dx\right)^{1/2}.\label{after I3 bound}
\end{align}
Set $\eta:=c_7\sqrt{\Sigma_{\min}}\log(T)/\sqrt T$ for some constant $c_7>0$. Then, with $c_8=4(c_6+c_7)$, (\ref{after I3 bound}) can be rewritten to  
\begin{align}
&\frac1{T}\int_{(0,T)\times\mathscr X}(\lambda_{\hat \theta}(t,x)-\lambda(t,x))^2\, \mathrm dt\, \mathrm dx\nonumber\\
&\leq\frac1{T}\int_{(0,T)\times\mathscr X}(\lambda_{\tilde \theta}(t,x)-\lambda(t,x))^2\, \mathrm dt\, \mathrm dx+\frac{c_8}2\frac{\log(T)\sqrt {\mathscr S}}{\sqrt T}\left(\frac1{T}\int_{(0,T)\times\mathscr X}(\lambda_{\tilde \theta}(t,x)-\lambda_{\hat \theta}(t,x))^2\, \mathrm dt\, \mathrm dx\right)^{1/2}.\label{after I3 bound2}
\end{align}

We now consider three cases:
\begin{enumerate}[label=(\roman*)]
\item $\displaystyle\int_{(0,T)\times\mathscr X}(\lambda_{\hat \theta}(t,x)-\lambda(t,x))^2\, \mathrm dt\, \mathrm dx\geq\int_{(0,T)\times\mathscr X}(\lambda_{\tilde \theta}(t,x)-\lambda(t,x))^2\, \mathrm dt\, \mathrm dx\quad$ and \\$\displaystyle\int_{(0,T)\times\mathscr X}(\lambda_{\hat \theta}(t,x)-\lambda(t,x))^2\, \mathrm dt\, \mathrm dx\geq\int_{(0,T)\times\mathscr X}(\lambda_{\tilde \theta}(t,x)-\lambda_{\hat \theta}(t,x))^2\, \mathrm dt\, \mathrm dx$;
\item $\displaystyle\int_{(0,T)\times\mathscr X}(\lambda_{\hat \theta}(t,x)-\lambda(t,x))^2\, \mathrm dt\, \mathrm dx<\int_{(0,T)\times\mathscr X}(\lambda_{\tilde \theta}(t,x)-\lambda(t,x))^2\, \mathrm dt\, \mathrm dx$;
\item $\displaystyle\int_{(0,T)\times\mathscr X}(\lambda_{\hat \theta}(t,x)-\lambda(t,x))^2\, \mathrm dt\, \mathrm dx<\int_{(0,T)\times\mathscr X}(\lambda_{\tilde \theta}(t,x)-\lambda_{\hat \theta}(t,x))^2\, \mathrm dt\, \mathrm dx$.
\end{enumerate}
For those three cases, we will bound $\frac1{T}\int_{(0,T)\times\mathscr X}(\lambda_{\hat \theta}(t,x)-\lambda(t,x))^2\, \mathrm dt\, \mathrm dx$.

\emph{Step 4A: Case (i)}. 
We can rewrite (\ref{after I3 bound2}) to 
\begin{align}
&\frac1{2T}\int_{(0,T)\times\mathscr X}(\lambda_{\hat \theta}(t,x)-\lambda(t,x))^2\, \mathrm dt\, \mathrm dx\nonumber\\
&\leq\frac1{2T}\left(\int_{(0,T)\times\mathscr X}(\lambda_{\tilde \theta}(t,x)-\lambda(t,x))^2\, \mathrm dt\, \mathrm dx\int_{(0,T)\times\mathscr X}(\lambda_{\hat \theta}(t,x)-\lambda(t,x))^2\, \mathrm dt\, \mathrm dx\right)^{1/2}\nonumber\\
&+\frac{c_8}2\frac{\log(T)\sqrt{\mathscr S}}{\sqrt T}\left(\frac1{T}\int_{(0,T)\times\mathscr X}(\lambda_{\hat \theta}(t,x)-\lambda(t,x))^2\, \mathrm dt\, \mathrm dx\right)^{1/2}.
\end{align}
We divide both sides by $\left(\frac1{T}\int_{(0,T)\times\mathscr X}(\lambda_{\hat \theta}(t,x)-\lambda(t,x))^2\, \mathrm dt\, \mathrm dx\right)^{1/2}$, apply Assumption \ref{ass approx}, and square both sides of the inequality, using that $(x+y)^2\leq2x^2+2y^2$. This gives 
\begin{equation}
\frac1{T}\int_{(0,T)\times\mathscr X}(\lambda_{\hat \theta}(t,x)-\lambda(t,x))^2\, \mathrm dt\, \mathrm dx\leq 2c_9 M^{-2d}H^{-d}+2c_8^2\frac{\log(T)^2\mathscr S}{T}
\label{after I3 bound3}
\end{equation}
on an event of probability at least 
$1-q_3T\exp(-q_4T^{1/5})-\frac{q_5}T-q_6T\exp(-q_7T^{1/5})$, 
where the constants $q_6,q_7$ and $c_9$ come from Assumption \ref{ass approx}. 

\emph{Step 4B: Case (ii)}. With different constants, we directly have (\ref{after I3 bound3}) with probability at least $1-q_8T\exp(q_9T^{1/5})$, by Assumption \ref{ass approx}.

\emph{Step 4C: Case (iii)}. We start from (\ref{optimal solution ineq}), but now add $$
\frac1{2T}\int_{(0,T)\times\mathscr X}\lambda_{\tilde \theta}(t,x)^2\, \mathrm dt\, \mathrm dx-\frac1T\int_{(0,T)\times\mathscr X}\lambda_{\tilde \theta}(t,x)\lambda_{\hat \theta}(t,x)\, \mathrm dt\, \mathrm dx
$$ 
to both sides of the inequality, which gives us 
\begin{align}
\frac1{2T}\int_{(0,T)\times\mathscr X}(\lambda_{\hat \theta}(t,x)-\lambda_{\tilde \theta}(t,x))^2\, \mathrm dt\, \mathrm dx&\leq\underbrace{\frac1T\int_{(0,T)\times\mathscr X}(\lambda_{\tilde \theta}(t,x)-\lambda_{\hat \theta}(t,x))(\lambda_{\tilde \theta}(t,x)-\lambda(t,x))\, \mathrm dt\, \mathrm dx}_{=:J_1}\nonumber\\
&+\underbrace{\frac1{T}\int_{(0,T)\times\mathscr X}(\lambda_{\tilde \theta}(t,x)-\lambda_{\hat \theta}(t,x))(\lambda(t,x)\, \mathrm dt\, \mathrm dx-N(\mathrm dt,\mathrm dx))}_{=:J_2}\nonumber\\
&+\underbrace{\eta(\|\tilde \beta\otimes\tilde\gamma\|_{\ell^1}-\|\hat \beta\otimes\hat\gamma\|_{\ell^1})}_{=:J_3}.\label{I123eq2}
\end{align}
By Cauchy-Schwarz, $J_1$ can be bounded by 
$$
\left(\frac1T\int_{(0,T)\times\mathscr X}(\lambda_{\tilde \theta}(t,x)-\lambda_{\hat \theta}(t,x))^2\, \mathrm dt\, \mathrm dx\right)^{1/2}\left(\frac1T\int_{(0,T)\times\mathscr X}(\lambda_{\tilde \theta}(t,x)-\lambda(t,x))^2\, \mathrm dt\, \mathrm dx\right)^{1/2},
$$ 
after which we may bound the second factor using Assumption \ref{ass approx}. Next, note that $J_2=I_2$, for which we apply the same bounding as before. Furthermore, $\|\Delta (\beta\otimes\gamma)\|_{\ell^1}$ may be bounded as before. 
After applying those bounds, and after  dividing both sides of the remaining equation by \begin{equation}\label{int hat tilde}
\left(\frac1T\int_{(0,T)\times\mathscr X}(\lambda_{\tilde \theta}(t,x)-\lambda_{\hat \theta}(t,x))^2\, \mathrm dt\, \mathrm dx\right)^{1/2},
\end{equation}
we obtain an inequality for (\ref{int hat tilde}) of the same shape as (\ref{after I3 bound2}). A similar argument as in Case (i) gives that (for different constants) on an event of probability at least 
$1-q_{10}T\exp(-q_{11}T^{1/5})-\frac{q_{12}}T-q_{13}T\exp(-q_{14}T^{1/5})$, 
\begin{equation}
\frac1{T}\int_{(0,T)\times\mathscr X}(\lambda_{\hat \theta}(t,x)-\lambda_{\tilde \theta}(t,x))^2\, \mathrm dt\, \mathrm dx\leq c_{10} M^{-2d}H^{-d}+c_{11}\frac{\log(T)^2\mathscr S}{T}.
\label{case 3 bound}
\end{equation}
The theorem follows since we have \begin{equation*}\int_{(0,T)\times\mathscr X}(\lambda_{\hat \theta}(t,x)-\lambda(t,x))^2\, \mathrm dt\, \mathrm dx<\int_{(0,T)\times\mathscr X}(\lambda_{\tilde \theta}(t,x)-\lambda_{\hat \theta}(t,x))^2\, \mathrm dt\, \mathrm dx
\end{equation*}
under Case (iii).
\qed\newline

{\it Proof of Theorem~\ref{recovery thm}}.
We begin by constructing a solution to a restricted version of (\ref{lasso LS contrast}) that satisfies the corresponding KKT conditions. By construction, this solution ensures that $\mathcal E \supset \hat{\mathcal E}$. Next, we establish $\mathcal E \subset \hat{\mathcal E}$ by showing that $\Delta\beta$ is sufficiently small and invoking Assumption \ref{minimal signal}. Consequently, $\mathcal E = \hat{\mathcal E}$. Finally, we verify the strict dual feasibility of the constructed solution, thereby confirming that the restricted solution also solves the original problem, i.e., that we have successfully ``witnessed'' the solution.

\emph{Step 1}. We construct an estimator $\hat\theta=(\hat\alpha,\hat\beta,\hat\gamma)\in\mathbb R^{M}\times\mathbb R^{M\times M}\times\mathbb R^H$ with $\|\hat\gamma\|_{\ell^1}=1$ satisfying the KKT conditions for optimizing (\ref{lasso LS contrast}). As in (\ref{expansion basis finite}), let $h_\gamma=\sum_{m\in[H]}\gamma_m\chi_m$. By taking (weak) derivatives with respect to $\beta_{k,l}$ in (\ref{lasso LS contrast}), we are looking for a primal-dual pair $(\beta_{k,l},g_{k,l})_{k,l\in[M]}$ satisfying
\begin{align}
&-\frac1T\int_{(0,T)\times\mathscr X}\int_{(0,t)\times\mathscr X}\varphi_k(x)\varphi_l(y)h_{\hat\gamma}(t-s)\,N(\mathrm ds,\mathrm dy)\,N(\mathrm dt,\mathrm dx)\nonumber\\
&+\frac1T\int_{(0,T)\times\mathscr X}\int_{(0,t)\times\mathscr X}\varphi_k(x)\varphi_l(y)h_{\hat\gamma}(t-s)\,N(\mathrm ds,\mathrm dy)\,\lambda_{\hat\theta}(t,x)\,\mathrm dt\,\mathrm dx+\eta \hat g_{k,l}=0,\label{dual equation}
\end{align}
where \begin{align}
\hat g_{k,l}&= \mathrm{sign}(\hat\beta_{k,l}),&&\text{if }\hat\beta_{k,l}\neq0;\label{g nonzero}\\
|\hat g_{k,l}|&\leq1,&&\text{if }\hat\beta_{k,l}=0.\label{g zero}
\end{align}
Note that by  Assumption \ref{ass approx}, $\tilde\beta_{k,l}=0$ for $(k,l)\in\mathcal E^{\mathsf{c}}$. First, we construct a primal solution. Only for $(k,l)\in\mathcal E$, we allow $\hat\beta_{k,l}$ to be nonzero: i.e., we 
compute $\hat \beta_{|\mathcal E}$, the restriction of $\hat\beta$ to $\mathcal E$, as 
\begin{align}
&(\hat\alpha,(\hat\beta_{|\mathcal E},\hat\beta_{|\mathcal E^{\mathsf{c}}}),\hat\gamma)=\hat \theta\label{lasso LS contrast PDW}\\
&\in\argmin_{\substack{\theta\in\mathbb R^{M}\times(\mathbb R^{\mathcal E}\times\{0\}^{\mathcal E^{\mathsf{c}}})\times\mathbb R^H\\\|\gamma\|_{\ell^1}=1}}\left\{\frac1T\left(-\int_{(0,T)\times\mathscr X}\lambda_\theta(t,x)\,N(\mathrm dt,\mathrm dx)+\frac12\int_{(0,T)\times\mathscr X}\lambda_\theta(t,x)^2\,\mathrm dt\, \mathrm dx\right)+\eta\|\beta\otimes\gamma\|_{\ell^1}\right\}.\nonumber
\end{align} 

Next, we construct the corresponding dual solution $\hat g$: if $\hat\beta_{k,l}\neq0$, define $\hat g_{k,l}$ through (\ref{g nonzero}); if $\hat\beta_{k,l}=0$, define $\hat g_{k,l}$ by solving (\ref{dual equation}).

The primal-dual pair solves the restriction (\ref{lasso LS contrast PDW}) of (\ref{lasso LS contrast}). By construction, we have  $\mathcal E\supset\hat{\mathcal E}$. Hence, it remains to show that $\mathcal E\subset\hat{\mathcal E}$ and that $\hat\theta$ solves the original problem (\ref{lasso LS contrast}) as well.

\emph{Step 2}. We show that $\mathcal E\subset\hat{\mathcal E}$.
Let $\Delta (\beta\otimes\gamma)=\tilde\beta\otimes\tilde\gamma-\hat\beta\otimes\hat\gamma$.
By Assumption \ref{ass approx} and by Lemma \ref{integrated diff and SDP}, we can find $c_1,c_2,q_1,q_2,q_3>0$ such that with probability at least  $1-q_1T\exp(-q_2T^{1/5})-q_3/T$ it holds that 
\begin{align}
\|\Delta (\beta\otimes\gamma)\|_2^2&\leq\frac1{\sigma_{\min}(G)}\sum_{k\in[M]}\left(\Delta \theta^\otimes_{k}\right)^\top G\Delta \theta^\otimes_{k}\leq\frac1{\Sigma_{\min}}\frac1{T}\int_{(0,T)\times\mathscr X}(\lambda_{\tilde \theta}(t,x)-\lambda_{\hat \theta}(t,x))^2\,\mathrm dt\,\mathrm dx\nonumber\\
&\leq\frac2{\Sigma_{\min}}\left(\frac1{T}\int_{(0,T)\times\mathscr X}(\lambda_{\hat \theta}(t,x)-\lambda_{\theta}(t,x))^2\,\mathrm dt\,\mathrm dx+\frac1{T}\int_{(0,T)\times\mathscr X}(\lambda_{\tilde \theta}(t,x)-\lambda_{\theta}(t,x))^2\,\mathrm dt\,\mathrm dx\right)\nonumber\\
&\leq c_1M^{-2d}+c_2\frac{\log(T)^2M^2H}{T},\label{delta beta bound}
\end{align}
using $(x+y)^2\leq2x^2+2y^2$. The right-hand side of (\ref{delta beta bound}) tends to $0$ as $T\to\infty$. Hence, for $T$ sufficiently large, $\max_{k,l\in[M],m\in[H]}|\Delta(\beta\otimes\gamma)_{k,l,m}|\leq\beta_{\min}\gamma_{\max}/2$. Note that $\min_{(k,l)\in\mathcal E}|\tilde\beta_{k,l}|\geq\beta_{\min}$ and $\max_{m\in[H]}|\tilde\gamma_m|\geq\gamma_{\max}$ by Assumption \ref{minimal signal}. 
From this, we infer that for all $(k,l)\in\mathcal E$, there exists $m\in[H]$ such that $|\hat\beta_{k,l}\hat\gamma_m|\geq\beta_{\min}\gamma_{\max}/2>0$. In particular, $\hat\beta_{k,l}\neq0$ for all $(k,l)\in\mathcal E$. In other words, $\mathcal E\subset\hat{\mathcal E}$ on an event of probability at least  $1-q_1T\exp(-q_2T^{1/5})-q_3/T$.

\emph{Step 3}.
By construction, $\mathcal E\supset\hat{\mathcal E}$, hence $\mathcal E=\hat{\mathcal E}$ on an event of probability at least  $1-q_1T\exp(-q_2T^{1/5})-q_3/T$. It remains to show that the solution $\hat\theta$ to the restricted problem (\ref{lasso LS contrast PDW})  solves the original problem (\ref{lasso LS contrast}) as well. In order to show this, we verify the \emph{strict} dual feasibility: $|\hat g_{k,l}|<1$ for all $(k,l)\in\mathcal E^{\mathsf{c}}$.  To prove this condition, we expand (\ref{dual equation}) into
\begin{align}
&-\eta \hat g_{k,l}\label{dual nonactive}=\underbrace{\frac1T\int_{(0,T)\times\mathscr X}\varphi_k(x)\Psi_{l,\hat\gamma}(t)(\lambda_{\hat\theta}(t,x)-\lambda_{\tilde\theta}(t,x))\,\mathrm dt\,\mathrm dx}_{=:D_1}\\
&+\underbrace{\frac1T\int_{(0,T)\times\mathscr X}\varphi_k(x)\Psi_{l,\hat\gamma}(t)(\lambda_{\tilde\theta}(t,x)-\lambda(t,x))\,\mathrm dt\,\mathrm dx}_{=:D_2}+\underbrace{\frac1T\int_{(0,T)\times\mathscr X}\varphi_k(x)\Psi_{l,\hat\gamma}(t)(\lambda(t,x)\,\mathrm dt\,\mathrm dx-N(\mathrm dt,\mathrm dx))}_{=:D_3},\nonumber
\end{align}
where $\Psi_{l,\gamma}(t):=\int_{(0,t)\times\mathscr X}\varphi_l(y)h_\gamma(t-s)\,N(\mathrm ds,\mathrm dy)=\sum_{m\in[H]}\gamma_m\Psi_{l,m}$, where $\Psi_{l,m}$ is defined in (\ref{psi def}).
We show that $|\hat g_{k,l}|<1$ for $T$ sufficiently large by showing that $|D_i/\eta|\leq1/3$ for $i=1,2,3$.  

\emph{Step 3A}. We first consider $D_1$. As in the proof of Lemma \ref{integrated diff and SDP}, 
set $\pi:=\beta\otimes\gamma$, so $\Delta\pi:=\Delta (\beta\otimes\gamma)=\tilde\beta\otimes\tilde\gamma-\hat\beta\otimes\hat\gamma$.  
Then $\Delta\theta^\otimes_k=(\Delta\alpha_k,\Delta\pi_{k,\cdot,\cdot})$, where $\pi_{k,l,m}:=\beta_{k,l}\gamma_m$.
Note that we can write 
\begin{align}
\lambda_{\hat\theta}(t,x)-\lambda_{\tilde\theta}(t,x)&=\sum_{k'\in[M]}\Delta\alpha_{k'}\varphi_{k'}(x)+\sum_{\substack{k',l'\in[M]\\m'\in[H]}}\Delta\pi_{k',l',m'}\varphi_{k'}(x)\Psi_{l',m'}(t)\nonumber\\
&=\sum_{k'\in[M]}\Delta\alpha_{k'}\varphi_{k'}(x)+\sum_{\substack{(k',l')\in\mathcal E\\m'\in[H]}}\Delta\pi_{k',l',m'}\varphi_{k'}(x)\Psi_{l',m'}(t)\label{diff lambda phi psi}
\end{align}
since $\hat\beta_{k',l'}=\tilde\beta_{k',l'}=0$ for $(k',l')\in\mathcal E$. Plug (\ref{diff lambda phi psi}) into the expression for $D_1$, use $\langle\varphi_k,\varphi_k\rangle=1$, and apply Cauchy-Schwarz. This gives 
\begin{align}
D_1&=\sum_{m\in[H]}\hat\gamma_m\left(\Delta\alpha_kG_{(l,m),0}+\sum_{\substack{l'\in\mathcal E_k\\m'\in[H]}}\Delta\pi_{k,l',m'}G_{(l,m),(l',m')}\right)\nonumber\\
&\leq\sum_{m\in[H]}\hat\gamma_m\|\Delta\theta^\otimes_k\|_2\left(G_{(l,m),0}^2+\sum_{\substack{l'\in\mathcal E_k\\m'\in[H]}}G_{(l,m),(l',m')}^2\right)^{1/2}. 
\end{align}
We may bound $\|\Delta\theta^\otimes_k\|_2$ as in the proof of Theorem \ref{consistency thm}: in equation (\ref{Frob norm bound}), the bound can be improved by replacing $\mathscr S\leq M^2H$ in Lemma \ref{integrated diff and SDP}   by $\mathcal SH$. As in (\ref{delta beta bound}), we can find $c_3,c_4,q_4,q_5,q_6>0$ such that
\begin{equation}
\|\Delta\theta^\otimes_k\|_2\leq\left(c_3M^{-2d}+c_4\frac{\log(T)^2\mathcal SH}{T}\right)^{1/2}\leq\sqrt{c_3}M^{-d}+\sqrt{c_4}\frac{\log(T)\sqrt{\mathcal S}\sqrt H}{\sqrt T}
\end{equation}
on an event of probability at least
$1-q_4T\exp(-q_5T^{1/5})-q_6/T$. The second inequality follows from $\sqrt{x+y}\leq\sqrt x+\sqrt y$ for $x,y\geq0$.
Since $\|\hat\gamma\|_{\ell^1}$=1, 
we can bound 
\begin{equation}
D_1\leq \left(\sqrt{c_3}M^{-d}+\sqrt{c_4}\frac{\log(T)\sqrt{\mathcal S}\sqrt H}{\sqrt T}\right)\max_{\substack{l\in \mathcal E_k^{\mathsf{c}}\\ m\in[H]}}\left(G_{(l,m),0}^2+\sum_{l'\in \mathcal E_k}G_{(l,m),(l',m')}^2\right)^{1/2}.
\end{equation}
By an application of Assumption \ref{irrepresentable cond}, using $q_6$ from the previous bound, we can find $c_5,c_6,q_7,q_8>0$ such that \begin{equation}|D_1/\eta|\leq \left(c_5 M^{-d}/\sqrt{\mathcal S}+c_6\frac{\log(T)}{\sqrt T}\right)\Big/\eta\end{equation}
 on an event of probability at least $1-q_7T\exp(-q_8T^{1/5})-q_6/T$. 
Since we set $\sqrt T/(M^d\sqrt{\mathcal S}\log(T))=\mathcal O(1)$, it holds  for  sufficiently large $T$  that $|D_1/\eta|<1/3$ if we take $\eta:=c_7\log(T)/\sqrt T$ for $c_7>0$ sufficiently large. 

\emph{Step 3B}. We move on to $D_2$.
Let 
\begin{equation}
G_{(l,\hat\gamma),(l,\hat\gamma)}:=\sum_{m,m'\in[H]}\hat\gamma_m\hat\gamma_{m'}G_{(l,m),(l,m')}.
\end{equation}
 By Cauchy-Schwarz, Assumption \ref{ass approx} and since $\langle \varphi_k,\varphi_k\rangle=1$, it follows that 
\begin{equation}\label{first bound D2}
D_2\leq\sqrt{G_{(l,\hat\gamma),(l,\hat\gamma)}}\left(\frac1T\int_{(0,T)\times\mathscr X}(\lambda_{\tilde \theta}(t,x)-\lambda(t,x))^2\,\mathrm dt\,\mathrm dx\right)^{1/2}\leq c_8\sqrt{G_{(l,\hat\gamma),(l,\hat\gamma)}}M^{-d}
\end{equation}
with probability at least  $1-q_9T\exp(-q_{10}T^{1/5})$, for some constants $q_9,q_{10}>0$. We give a probabilistic upper on bound $G_{(l,m),(l,m')}$, indendent of $m,m'\in[H]$. First, we bound
\begin{align}
&\mathbb E[G_{(l,\hat\gamma),(l,\hat\gamma)}]\nonumber\\
&\leq\frac1T\int_{(0,T)\times\mathscr X}\int_{(0,t)\times\mathscr X}\int_{(0,t)\times\mathscr X}\varphi_l(y_1)\varphi_l(y_2)h_{\hat\gamma}(t-s_1)h_{\hat\gamma}(t-s_2)\,\mathbb E[N(\mathrm ds_1,\mathrm dy_1)N(\mathrm ds_2,\mathrm dy_2)]\,\mathrm dt\nonumber\\
&\leq c_9\frac1T\int_{(0,T)\times\mathscr X}\int_{(0,t)\times\mathscr X}\int_{(0,t)\times\mathscr X}\varphi_l(y_1)\varphi_l(y_2)h_{\hat\gamma}(t-s_1)h_{\hat\gamma}(t-s_2)\,\mathrm ds_1\,\mathrm dy_1\,\mathrm ds_2\,\mathrm dy_2\,\mathrm dt\nonumber\\
&\leq c_9 \mathfrak h_1^2\langle \varphi_l,1\rangle^2.\label{EG_ll upper}
\end{align}
where the second inequality holds for some constant $c_9>0$, by Lemma \ref{step3Alemma} from Appendix~\hyperref[app A]{A}, and where the third inequality uses that $\|\hat\gamma\|_{\ell^1}=1$. Note that $\mu(\mathscr X)<\infty$, hence $1\in L^2(\mathscr X)$. Hence, by Bessel's inequality we infer that $\langle \varphi_l,1\rangle$ is uniformly bounded over $l\in\mathbb N$. Hence, we can find a constant $c_{10}>0$ such that $\mathbb E[G_{(l,\hat\gamma),(l,\hat\gamma)}]\leq c_{10}$ for all $l\in\mathbb N$. 

Next, we apply Lemma \ref{119120 lemma} from Appendix~\hyperref[app A]{A} with $\Delta^{(MH)}=e_l^{(M)}\otimes\hat\gamma$, where $e_l^{(M)}$ is the $l$'th basis unit vector in $\mathbb R^{M}$. Note that since $|\hat\gamma_m|\leq1$ for each $m\in[H]$,  $\|\Delta^{(MH)}\|_{\ell^2}^2=\|\hat\gamma\|_2\leq\|\hat\gamma\|_1=1$. This implies that we can find constants $q_{11},q_{12}>0$ such that for $c_{11}=c_{10}/2$ and for $T$ sufficiently large, $ G_{(l,\hat\gamma),(l,\hat\gamma)}\leq c_{11}$ on an event of probability at least $1-q_{11}T\exp(-q_{12}T^{1/5})$. Hence, for $c_{12}:=c_8\sqrt{c_{11}}$, we can find $q_{13},q_{14}>0$ such that
\begin{equation}\label{second bound D2}
D_2\leq c_{12} M^{-d}
\end{equation}
on an event of probability at least $1-q_{13}T\exp(-q_{14}T^{1/5})$. Since $\sqrt T/(M^d\log(T))=\mathcal O(1)$ for  sufficiently large $T$ it holds that $|D_2/\eta|<1/3$ if we take $\eta:=c_7\log(T)/\sqrt T$ for  sufficiently large $c_7>0$. 

\emph{Step 3C}. Finally, to obtain a bound for $D_3$, we employ a truncation argument together with an exponential inequality for martingales, conditional on the event $A_{0,T;B}$ defined in (\ref{A event eq}). This approach parallels the one used to bound $I_2$ in the proof of Theorem \ref{consistency thm}, and is needed here as well since we work with unbounded martingales, while the exponential inequality assumes a bounded martingale. To avoid repeating the entire argument, we only sketch the estimates for the relevant quantities.

Take $B=2\log(T)/q_2^*$, as before, where $q_2^*>0$ is the constant defined in the proof of Theorem \ref{consistency thm}. Consider the martingale 
\begin{equation}\label{martingale Y2_T}
Y_T:=\int_{(0,T)\times\mathscr X}\varphi_k(x)\Psi_{l,\hat\gamma}(t)(\lambda(t,x)\,\mathrm dt\,\mathrm dx-N(\mathrm dt,\mathrm dx))
\end{equation}
having predictable quadratic variation
\begin{equation}\label{martingale Y2_T QV}
\langle Y,Y \rangle_T= \int_{(0,T)\times\mathscr X} (\varphi_k(x)\Psi_{l,\hat\gamma}(t))^2 \lambda(t,x)\, \mathrm dt \,\mathrm dx.
\end{equation}
Let $\tau$ be as in (\ref{stopping time}), and consider the stopped martingale $X_t=Y_{t\wedge\tau}$. As in the proof of Theorem \ref{consistency thm}, on $A_{0,T;B}$ there are at most $B$ events in any interval $[t-b,t]$, for $t\in[0,T]$. Assumptions \ref{stability ass} and  \ref{min eigenvalue ass} give that 
\begin{equation}|\Delta X_t|\leq c_{13}M\log(T)\end{equation} for some $c_{13}>0$, where we use that $\|\hat\gamma\|_{\ell^1}=1$, which holds by construction. Furthermore, in the bounding of $D_2$ we found that on an event of probability at least $1-q_{13}T\exp(-q_{14}T^{1/5})$ it holds that 
\begin{equation}\int_{(0,T)}\Psi_{l,\hat\gamma}(t)^2\,\mathrm dt=TG_{(l,\hat\gamma),(l,\hat\gamma)}\leq T c_{11},\end{equation} while $\langle\varphi_k,\varphi_k\rangle=1$. This gives the  bound  \begin{equation}\langle X,X \rangle_T\leq c_{14}T\log(T)\end{equation} for some $c_{14}>0$. We apply inequality (\ref{lemma2.1 ineq}) with $a=\sqrt{c_{14}}\sqrt T\log(T)$. Using (\ref{P3 bound}), this implies that we can find $q_{15},q_{16},q_{17}>0$ such that $D_3\leq \log T/\sqrt T$ on an event of probability at least $1-q_{15}T\exp(-q_{16}T^{1/5})-q_{17}/T$. Therefore, for  sufficiently large $T$ it holds that $|D_3/\eta|<1/3$ if we take $\eta:=c_7\log(T)/\sqrt T$ for $c_7>0$ sufficiently large. 

Combining the bounds for $D_1,D_2$ and $D_3$, it follows that $|\hat g_{k,l}|<1$ for $T$ sufficiently large, if we take $\eta:=c_7\log(T)/\sqrt T$ for $c_7>0$ sufficiently large. Since the bounds for $D_1,D_2$ and $D_3$ do not depend on $(k,l)\in\mathcal E$, the theorem follows.
\qed

\section*{Appendix C: Relegated simulation results of Section~\ref{simulation study section}}\label{app C}
In seven different setups, we sampled $100$ realization of a graphon Hawkes process on the time interval $[0,5\cdot10^5]$, to which we applied our estimation procedure. Here, we report mean and standard deviation (between brackets) of the estimated coefficients.

For the first unperturbed example, we obtain
\small
\begin{equation}
\begin{array}{c}\hat\beta\\(\mathrm{sd}(\hat\beta))\end{array} 
=
\renewcommand{\arraystretch}{1.3}
\setlength{\arraycolsep}{5pt}
\begin{pmatrix}
\substack{.7421\\[2pt](.0394)} &
\substack{.2044\\[2pt](.0146)} &
\substack{.0000\\[2pt](.0002)} &
\substack{-.0000\\[2pt](.0003)} &
\substack{.0002\\[2pt](.0011)} &
\substack{.0000\\[2pt](.0004)} &
\substack{-.0004\\[2pt](.0047)} &
\substack{-.0001\\[2pt](.0009)}
\\
\substack{-.0001\\[2pt](.0006)} &
\substack{-.0000\\[2pt](.0003)} &
\substack{-.0000\\[2pt](.0000)} &
\substack{-.0000\\[2pt](.0000)} &
\substack{-.0000\\[2pt](.0007)} &
\substack{.0001\\[2pt](.0014)} &
\substack{.0000\\[2pt](.0002)} &
\substack{-.0000\\[2pt](.0002)}
\\
\substack{.0000\\[2pt](.0006)} &
\substack{.0000\\[2pt](.0005)} &
\substack{-.0000\\[2pt](.0001)} &
\substack{.0000\\[2pt](.0002)} &
\substack{.0000\\[2pt](.0007)} &
\substack{-.0001\\[2pt](.0015)} &
\substack{-.0001\\[2pt](.0004)} &
\substack{-.0002\\[2pt](.0026)}
\\
\substack{-.0001\\[2pt](.0003)} &
\substack{.0001\\[2pt](.0013)} &
\substack{.0001\\[2pt](.0004)} &
\substack{-.0001\\[2pt](.0009)} &
\substack{-.0002\\[2pt](.0018)} &
\substack{-.0001\\[2pt](.0007)} &
\substack{-.0001\\[2pt](.0010)} &
\substack{-.0000\\[2pt](.0004)}
\\
\substack{.0001\\[2pt](.0006)} &
\substack{.0001\\[2pt](.0004)} &
\substack{-.0000\\[2pt](.0002)} &
\substack{-.0002\\[2pt](.0009)} &
\substack{.0001\\[2pt](.0009)} &
\substack{.0001\\[2pt](.0009)} &
\substack{.0001\\[2pt](.0016)} &
\substack{.0000\\[2pt](.0003)}
\\
\substack{.0000\\[2pt](.0002)} &
\substack{-.0000\\[2pt](.0004)} &
\substack{.0980\\[2pt](.0136)} &
\substack{.0000\\[2pt](.0003)} &
\substack{.0000\\[2pt](.0002)} &
\substack{.0000\\[2pt](.0001)} &
\substack{-.0001\\[2pt](.0012)} &
\substack{.0000\\[2pt](.0003)}
\\
\substack{.0002\\[2pt](.0010)} &
\substack{-.0000\\[2pt](.0000)} &
\substack{.0000\\[2pt](.0005)} &
\substack{-.0002\\[2pt](.0024)} &
\substack{-.0001\\[2pt](.0008)} &
\substack{.0001\\[2pt](.0005)} &
\substack{.0000\\[2pt](.0014)} &
\substack{.0001\\[2pt](.0007)}
\\
\substack{.0000\\[2pt](.0002)} &
\substack{-.0000\\[2pt](.0001)} &
\substack{-.1001\\[2pt](.0134)} &
\substack{.0000\\[2pt](.0002)} &
\substack{.0000\\[2pt](.0002)} &
\substack{.0001\\[2pt](.0009)} &
\substack{.0000\\[2pt](.0002)} &
\substack{.0000\\[2pt](.0004)}
\end{pmatrix}
\end{equation}
\normalsize

For the second unperturbed example, we obtain
\small
\begin{equation}
\begin{array}{c}\hat\beta\\(\mathrm{sd}(\hat\beta))\end{array} 
=
\renewcommand{\arraystretch}{1.3}
\setlength{\arraycolsep}{5pt}
\begin{pmatrix}
\substack{.6810\\[2pt](.0300)} &
\substack{-.0000\\[2pt](.0004)} &
\substack{-.0000\\[2pt](.0000)} &
\substack{.1949\\[2pt](.0143)} &
\substack{-.0000\\[2pt](.0003)} &
\substack{-.0001\\[2pt](.0011)} &
\substack{.0002\\[2pt](.0027)} &
\substack{-.0002\\[2pt](.0030)}
\\
\substack{-.0000\\[2pt](.0002)} &
\substack{-.0000\\[2pt](.0002)} &
\substack{-.0005\\[2pt](.0046)} &
\substack{.0000\\[2pt](.0022)} &
\substack{.0000\\[2pt](.0012)} &
\substack{.0000\\[2pt](.0009)} &
\substack{-.0000\\[2pt](.0002)} &
\substack{-.0000\\[2pt](.0002)}
\\
\substack{-.0000\\[2pt](.0004)} &
\substack{.0001\\[2pt](.0006)} &
\substack{-.0001\\[2pt](.0010)} &
\substack{.0001\\[2pt](.0005)} &
\substack{-.0003\\[2pt](.0026)} &
\substack{-.0001\\[2pt](.0011)} &
\substack{.0000\\[2pt](.0004)} &
\substack{.0001\\[2pt](.0009)}
\\
\substack{-.0000\\[2pt](.0001)} &
\substack{.0000\\[2pt](.0019)} &
\substack{.0000\\[2pt](.0001)} &
\substack{.0000\\[2pt](.0001)} &
\substack{-.0001\\[2pt](.0005)} &
\substack{-.0000\\[2pt](.0001)} &
\substack{-.0002\\[2pt](.0015)} &
\substack{.0000\\[2pt](.0001)}
\\
\substack{.0001\\[2pt](.0005)} &
\substack{-.1098\\[2pt](.0140)} &
\substack{-.0000\\[2pt](.0004)} &
\substack{.0002\\[2pt](.0016)} &
\substack{-.0001\\[2pt](.0009)} &
\substack{.0005\\[2pt](.0045)} &
\substack{-.0001\\[2pt](.0011)} &
\substack{.0001\\[2pt](.0007)}
\\
\substack{.0001\\[2pt](.0005)} &
\substack{-.1061\\[2pt](.0153)} &
\substack{-.0002\\[2pt](.0027)} &
\substack{.0000\\[2pt](.0002)} &
\substack{.0001\\[2pt](.0035)} &
\substack{.0004\\[2pt](.0023)} &
\substack{.0001\\[2pt](.0008)} &
\substack{.0000\\[2pt](.0003)}
\\
\substack{.0003\\[2pt](.0009)} &
\substack{-.0996\\[2pt](.0127)} &
\substack{-.0004\\[2pt](.0041)} &
\substack{.0002\\[2pt](.0010)} &
\substack{-.0004\\[2pt](.0020)} &
\substack{.0000\\[2pt](.0002)} &
\substack{.0001\\[2pt](.0009)} &
\substack{-.0008\\[2pt](.0060)}
\\
\substack{.0002\\[2pt](.0006)} &
\substack{-.1010\\[2pt](.0156)} &
\substack{.0000\\[2pt](.0002)} &
\substack{.0003\\[2pt](.0021)} &
\substack{-.0001\\[2pt](.0008)} &
\substack{.0002\\[2pt](.0014)} &
\substack{.0000\\[2pt](.0000)} &
\substack{-.0000\\[2pt](.0000)}
\end{pmatrix}
\end{equation}
\normalsize

For the third unperturbed example, we obtain
\small
\begin{equation}
\begin{array}{c}\hat\beta\\(\mathrm{sd}(\hat\beta))\end{array} 
=
\renewcommand{\arraystretch}{1.3}
\setlength{\arraycolsep}{5pt}
\begin{pmatrix}
\substack{.7290\\[2pt](.0143)} &
\substack{.0001\\[2pt](.0010)} &
\substack{-.0001\\[2pt](.0011)} &
\substack{.1231\\[2pt](.0132)} &
\substack{.0001\\[2pt](.0021)} &
\substack{.0002\\[2pt](.0010)} &
\substack{.0001\\[2pt](.0004)} &
\substack{.0001\\[2pt](.0017)}
\\
\substack{-.0001\\[2pt](.0008)} &
\substack{.0003\\[2pt](.0018)} &
\substack{.0001\\[2pt](.0007)} &
\substack{.0002\\[2pt](.0009)} &
\substack{-.0000\\[2pt](.0007)} &
\substack{.0001\\[2pt](.0007)} &
\substack{.0001\\[2pt](.0009)} &
\substack{.0000\\[2pt](.0005)}
\\
\substack{.0001\\[2pt](.0007)} &
\substack{-.0002\\[2pt](.0025)} &
\substack{-.0001\\[2pt](.0010)} &
\substack{.0002\\[2pt](.0022)} &
\substack{-.0001\\[2pt](.0014)} &
\substack{.0002\\[2pt](.0017)} &
\substack{-.0001\\[2pt](.0009)} &
\substack{.0000\\[2pt](.0005)}
\\
\substack{-.0000\\[2pt](.0004)} &
\substack{.0001\\[2pt](.0006)} &
\substack{-.0000\\[2pt](.0009)} &
\substack{-.0000\\[2pt](.0017)} &
\substack{-.0003\\[2pt](.0023)} &
\substack{.0000\\[2pt](.0010)} &
\substack{-.0001\\[2pt](.0007)} &
\substack{.0000\\[2pt](.0006)}
\\
\substack{.0002\\[2pt](.0007)} &
\substack{-.1027\\[2pt](.0118)} &
\substack{-.0002\\[2pt](.0018)} &
\substack{-.0000\\[2pt](.0006)} &
\substack{.0341\\[2pt](.0111)} &
\substack{.0008\\[2pt](.0062)} &
\substack{.0001\\[2pt](.0012)} &
\substack{.0002\\[2pt](.0012)}
\\
\substack{.0000\\[2pt](.0006)} &
\substack{-.0975\\[2pt](.0131)} &
\substack{-.0001\\[2pt](.0006)} &
\substack{.0000\\[2pt](.0006)} &
\substack{-.0001\\[2pt](.0008)} &
\substack{.0002\\[2pt](.0008)} &
\substack{.0001\\[2pt](.0005)} &
\substack{.0001\\[2pt](.0013)}
\\
\substack{.0007\\[2pt](.0014)} &
\substack{-.0984\\[2pt](.0150)} &
\substack{.0001\\[2pt](.0010)} &
\substack{-.0001\\[2pt](.0010)} &
\substack{-.0002\\[2pt](.0012)} &
\substack{.0003\\[2pt](.0018)} &
\substack{.0323\\[2pt](.0113)} &
\substack{.0001\\[2pt](.0008)}
\\
\substack{.0002\\[2pt](.0009)} &
\substack{-.0942\\[2pt](.0158)} &
\substack{.0001\\[2pt](.0008)} &
\substack{-.0001\\[2pt](.0010)} &
\substack{-.0001\\[2pt](.0009)} &
\substack{-.0001\\[2pt](.0014)} &
\substack{-.0000\\[2pt](.0010)} &
\substack{.0001\\[2pt](.0005)}
\end{pmatrix}
\end{equation}
\normalsize

For the fourth unperturbed example, we obtain
\small
\begin{equation}
\begin{array}{c}\hat\beta\\(\mathrm{sd}(\hat\beta))\end{array} 
=
\renewcommand{\arraystretch}{1.3}
\setlength{\arraycolsep}{5pt}
\begin{pmatrix}
\substack{.7508\\[2pt](.0383)} &
\substack{.0355\\[2pt](.0090)} &
\substack{-.0360\\[2pt](.0078)} &
\substack{.0347\\[2pt](.0097)} &
\substack{.0000\\[2pt](.0007)} &
\substack{-.0001\\[2pt](.0013)} &
\substack{-.0000\\[2pt](.0014)} &
\substack{-.0001\\[2pt](.0010)}
\\
\substack{.0001\\[2pt](.0012)} &
\substack{-.0001\\[2pt](.0015)} &
\substack{-.0004\\[2pt](.0036)} &
\substack{.0003\\[2pt](.0019)} &
\substack{.0004\\[2pt](.0043)} &
\substack{.0001\\[2pt](.0008)} &
\substack{-.0002\\[2pt](.0017)} &
\substack{-.0000\\[2pt](.0012)}
\\
\substack{.0000\\[2pt](.0004)} &
\substack{-.0001\\[2pt](.0007)} &
\substack{.0006\\[2pt](.0050)} &
\substack{-.0000\\[2pt](.0006)} &
\substack{-.0001\\[2pt](.0010)} &
\substack{.0002\\[2pt](.0017)} &
\substack{.0000\\[2pt](.0005)} &
\substack{.0002\\[2pt](.0012)}
\\
\substack{-.0000\\[2pt](.0004)} &
\substack{-.0001\\[2pt](.0011)} &
\substack{.0001\\[2pt](.0004)} &
\substack{-.0001\\[2pt](.0011)} &
\substack{.0003\\[2pt](.0030)} &
\substack{-.0000\\[2pt](.0010)} &
\substack{.0006\\[2pt](.0047)} &
\substack{-.0001\\[2pt](.0007)}
\\
\substack{.0001\\[2pt](.0008)} &
\substack{-.0337\\[2pt](.0093)} &
\substack{-.0297\\[2pt](.0093)} &
\substack{-.0001\\[2pt](.0009)} &
\substack{-.0118\\[2pt](.0097)} &
\substack{.0339\\[2pt](.0087)} &
\substack{.0260\\[2pt](.0101)} &
\substack{.0170\\[2pt](.0087)}
\\
\substack{.0001\\[2pt](.0009)} &
\substack{-.0298\\[2pt](.0090)} &
\substack{-.0275\\[2pt](.0101)} &
\substack{-.0000\\[2pt](.0014)} &
\substack{-.0133\\[2pt](.0095)} &
\substack{.0391\\[2pt](.0096)} &
\substack{.0222\\[2pt](.0085)} &
\substack{.0199\\[2pt](.0086)}
\\
\substack{.0001\\[2pt](.0005)} &
\substack{-.0181\\[2pt](.0073)} &
\substack{-.0243\\[2pt](.0088)} &
\substack{.0001\\[2pt](.0006)} &
\substack{.0217\\[2pt](.0092)} &
\substack{-.0310\\[2pt](.0106)} &
\substack{-.0223\\[2pt](.0082)} &
\substack{-.0188\\[2pt](.0087)}
\\
\substack{-.0000\\[2pt](.0008)} &
\substack{-.0191\\[2pt](.0090)} &
\substack{-.0234\\[2pt](.0093)} &
\substack{.0000\\[2pt](.0004)} &
\substack{.0225\\[2pt](.0104)} &
\substack{-.0282\\[2pt](.0076)} &
\substack{-.0235\\[2pt](.0090)} &
\substack{-.0188\\[2pt](.0112)}
\end{pmatrix}
\end{equation}
\normalsize

For the first perturbed example, we obtain
\small
\begin{equation}
\begin{array}{c}\hat\beta\\(\mathrm{sd}(\hat\beta))\end{array} 
=
\renewcommand{\arraystretch}{1.3}
\setlength{\arraycolsep}{5pt}
\begin{pmatrix}
\substack{.7511\\[2pt](.0555)} &
\substack{.2083\\[2pt](.0150)} &
\substack{.0001\\[2pt](.0008)} &
\substack{.0000\\[2pt](.0006)} &
\substack{.0000\\[2pt](.0027)} &
\substack{-.0001\\[2pt](.0013)} &
\substack{.0009\\[2pt](.0053)} &
\substack{.0001\\[2pt](.0012)}
\\
\substack{-.0004\\[2pt](.0010)} &
\substack{.0001\\[2pt](.0026)} &
\substack{-.0002\\[2pt](.0055)} &
\substack{-.0011\\[2pt](.0041)} &
\substack{-.0005\\[2pt](.0024)} &
\substack{-.0003\\[2pt](.0012)} &
\substack{-.0008\\[2pt](.0023)} &
\substack{-.0003\\[2pt](.0008)}
\\
\substack{.0000\\[2pt](.0011)} &
\substack{.0003\\[2pt](.0016)} &
\substack{-.0002\\[2pt](.0019)} &
\substack{-.0001\\[2pt](.0013)} &
\substack{.0004\\[2pt](.0022)} &
\substack{.0001\\[2pt](.0011)} &
\substack{.0001\\[2pt](.0012)} &
\substack{.0002\\[2pt](.0019)}
\\
\substack{-.0010\\[2pt](.0028)} &
\substack{.0002\\[2pt](.0029)} &
\substack{-.0003\\[2pt](.0013)} &
\substack{-.0005\\[2pt](.0019)} &
\substack{-.0010\\[2pt](.0044)} &
\substack{-.0003\\[2pt](.0015)} &
\substack{-.0002\\[2pt](.0008)} &
\substack{-.0006\\[2pt](.0037)}
\\
\substack{.0003\\[2pt](.0017)} &
\substack{.0001\\[2pt](.0008)} &
\substack{.0002\\[2pt](.0022)} &
\substack{-.0001\\[2pt](.0015)} &
\substack{-.0002\\[2pt](.0021)} &
\substack{-.0003\\[2pt](.0034)} &
\substack{.0003\\[2pt](.0018)} &
\substack{-.0003\\[2pt](.0029)}
\\
\substack{-.0006\\[2pt](.0016)} &
\substack{-.0004\\[2pt](.0017)} &
\substack{.0746\\[2pt](.0123)} &
\substack{-.0007\\[2pt](.0037)} &
\substack{-.0017\\[2pt](.0047)} &
\substack{-.0001\\[2pt](.0031)} &
\substack{-.0009\\[2pt](.0030)} &
\substack{-.0007\\[2pt](.0034)}
\\
\substack{.0016\\[2pt](.0027)} &
\substack{-.0001\\[2pt](.0041)} &
\substack{.0016\\[2pt](.0054)} &
\substack{.0001\\[2pt](.0016)} &
\substack{.0007\\[2pt](.0022)} &
\substack{.0004\\[2pt](.0026)} &
\substack{-.0000\\[2pt](.0027)} &
\substack{.0001\\[2pt](.0015)}
\\
\substack{-.0005\\[2pt](.0017)} &
\substack{-.0004\\[2pt](.0018)} &
\substack{-.1027\\[2pt](.0135)} &
\substack{-.0005\\[2pt](.0024)} &
\substack{-.0007\\[2pt](.0026)} &
\substack{-.0002\\[2pt](.0036)} &
\substack{-.0023\\[2pt](.0048)} &
\substack{-.0004\\[2pt](.0035)}
\end{pmatrix}
\end{equation}
\normalsize

For the second perturbed example, we obtain
\small
\begin{equation}
\begin{array}{c}\hat\beta\\(\mathrm{sd}(\hat\beta))\end{array} 
=
\renewcommand{\arraystretch}{1.3}
\setlength{\arraycolsep}{5pt}
\begin{pmatrix}
\substack{.7426\\[2pt](.0420)} &
\substack{.2060\\[2pt](.0172)} &
\substack{-.0002\\[2pt](.0032)} &
\substack{-.0000\\[2pt](.0004)} &
\substack{-.0000\\[2pt](.0002)} &
\substack{-.0000\\[2pt](.0006)} &
\substack{-.0000\\[2pt](.0003)} &
\substack{-.0001\\[2pt](.0007)}
\\
\substack{-.0001\\[2pt](.0003)} &
\substack{.0000\\[2pt](.0003)} &
\substack{-.0000\\[2pt](.0001)} &
\substack{.0001\\[2pt](.0008)} &
\substack{.0000\\[2pt](.0006)} &
\substack{.0001\\[2pt](.0007)} &
\substack{.0003\\[2pt](.0019)} &
\substack{.0001\\[2pt](.0007)}
\\
\substack{.0001\\[2pt](.0004)} &
\substack{-.0000\\[2pt](.0002)} &
\substack{-.0001\\[2pt](.0006)} &
\substack{-.0001\\[2pt](.0006)} &
\substack{-.0001\\[2pt](.0005)} &
\substack{.0000\\[2pt](.0004)} &
\substack{-.0003\\[2pt](.0015)} &
\substack{-.0002\\[2pt](.0015)}
\\
\substack{-.0001\\[2pt](.0004)} &
\substack{-.0003\\[2pt](.0041)} &
\substack{-.0000\\[2pt](.0000)} &
\substack{-.0002\\[2pt](.0028)} &
\substack{.0001\\[2pt](.0015)} &
\substack{-.0002\\[2pt](.0018)} &
\substack{.0001\\[2pt](.0014)} &
\substack{-.0000\\[2pt](.0003)}
\\
\substack{.0004\\[2pt](.0009)} &
\substack{-.0001\\[2pt](.0008)} &
\substack{.0001\\[2pt](.0013)} &
\substack{.0001\\[2pt](.0020)} &
\substack{.0001\\[2pt](.0024)} &
\substack{-.0008\\[2pt](.0038)} &
\substack{-.0001\\[2pt](.0006)} &
\substack{.0000\\[2pt](.0003)}
\\
\substack{.0001\\[2pt](.0006)} &
\substack{-.0000\\[2pt](.0005)} &
\substack{.1028\\[2pt](.0151)} &
\substack{-.0002\\[2pt](.0007)} &
\substack{.0000\\[2pt](.0000)} &
\substack{-.0001\\[2pt](.0004)} &
\substack{.0001\\[2pt](.0015)} &
\substack{.0001\\[2pt](.0010)}
\\
\substack{.0002\\[2pt](.0006)} &
\substack{-.0000\\[2pt](.0001)} &
\substack{.0002\\[2pt](.0015)} &
\substack{.0001\\[2pt](.0024)} &
\substack{-.0001\\[2pt](.0006)} &
\substack{-.0000\\[2pt](.0014)} &
\substack{-.0001\\[2pt](.0009)} &
\substack{.0000\\[2pt](.0002)}
\\
\substack{.0002\\[2pt](.0005)} &
\substack{-.0006\\[2pt](.0019)} &
\substack{-.0833\\[2pt](.0117)} &
\substack{-.0030\\[2pt](.0042)} &
\substack{.0002\\[2pt](.0012)} &
\substack{-.0023\\[2pt](.0034)} &
\substack{-.0009\\[2pt](.0030)} &
\substack{.0001\\[2pt](.0007)}
\end{pmatrix}
\end{equation}
\normalsize

For the third perturbed example, we obtain
\small
\begin{equation}
\begin{array}{c}\hat\beta\\(\mathrm{sd}(\hat\beta))\end{array} 
=
\renewcommand{\arraystretch}{1.3}
\setlength{\arraycolsep}{5pt}
\begin{pmatrix}
\substack{.7358\\[2pt](.0252)} &
\substack{.2034\\[2pt](.0126)} &
\substack{.0002\\[2pt](.0018)} &
\substack{.0002\\[2pt](.0019)} &
\substack{-.0000\\[2pt](.0014)} &
\substack{.0001\\[2pt](.0009)} &
\substack{-.0000\\[2pt](.0005)} &
\substack{.0001\\[2pt](.0031)}
\\
\substack{-.0001\\[2pt](.0008)} &
\substack{.0001\\[2pt](.0010)} &
\substack{-.0001\\[2pt](.0014)} &
\substack{-.0000\\[2pt](.0007)} &
\substack{.0001\\[2pt](.0006)} &
\substack{.0002\\[2pt](.0016)} &
\substack{.0003\\[2pt](.0015)} &
\substack{.0000\\[2pt](.0011)}
\\
\substack{.0002\\[2pt](.0010)} &
\substack{-.0001\\[2pt](.0010)} &
\substack{.0002\\[2pt](.0021)} &
\substack{-.0002\\[2pt](.0023)} &
\substack{.0003\\[2pt](.0016)} &
\substack{-.0001\\[2pt](.0006)} &
\substack{-.0002\\[2pt](.0013)} &
\substack{.0001\\[2pt](.0030)}
\\
\substack{-.0001\\[2pt](.0007)} &
\substack{.0003\\[2pt](.0028)} &
\substack{.0001\\[2pt](.0020)} &
\substack{.0002\\[2pt](.0017)} &
\substack{.0003\\[2pt](.0024)} &
\substack{.0001\\[2pt](.0007)} &
\substack{.0002\\[2pt](.0012)} &
\substack{-.0000\\[2pt](.0003)}
\\
\substack{.0003\\[2pt](.0012)} &
\substack{.0002\\[2pt](.0023)} &
\substack{-.0002\\[2pt](.0013)} &
\substack{.0000\\[2pt](.0007)} &
\substack{-.0002\\[2pt](.0010)} &
\substack{-.0000\\[2pt](.0005)} &
\substack{.0002\\[2pt](.0012)} &
\substack{-.0001\\[2pt](.0007)}
\\
\substack{.0001\\[2pt](.0008)} &
\substack{-.0003\\[2pt](.0018)} &
\substack{.0969\\[2pt](.0147)} &
\substack{.0002\\[2pt](.0014)} &
\substack{.0001\\[2pt](.0012)} &
\substack{-.0001\\[2pt](.0012)} &
\substack{-.0000\\[2pt](.0005)} &
\substack{.0000\\[2pt](.0009)}
\\
\substack{.0003\\[2pt](.0009)} &
\substack{-.0000\\[2pt](.0010)} &
\substack{.0004\\[2pt](.0025)} &
\substack{-.0003\\[2pt](.0017)} &
\substack{-.0001\\[2pt](.0007)} &
\substack{.0000\\[2pt](.0005)} &
\substack{.0000\\[2pt](.0009)} &
\substack{.0002\\[2pt](.0020)}
\\
\substack{-.0000\\[2pt](.0002)} &
\substack{-.0001\\[2pt](.0006)} &
\substack{-.0966\\[2pt](.0119)} &
\substack{-.0001\\[2pt](.0007)} &
\substack{-.0002\\[2pt](.0012)} &
\substack{-.0002\\[2pt](.0014)} &
\substack{-.0001\\[2pt](.0020)} &
\substack{.0002\\[2pt](.0013)}
\end{pmatrix}
\end{equation}
\normalsize

\end{document}